\documentclass[oneside,reqno,11pt]{amsart}

\usepackage{graphicx}
\usepackage{epstopdf}
\usepackage{amsmath}
\usepackage{amssymb}
\usepackage[top=1in, bottom=1.25in, left=1.25in, right=1.25in]{geometry}
\usepackage{textcomp}
\usepackage{mathrsfs}
\usepackage{tikz}
\usepackage{tikz-cd}
\usepackage[utf8]{inputenc}
\usepackage{mathabx}
\usepackage{bm}
\usepackage{amsthm}
\usepackage{amsfonts}
\usepackage[colorlinks=true,linkcolor=blue]{hyperref}
\usepackage{pb-diagram}
\usepackage{lipsum}
\usepackage{secdot} 
\usepackage{color}
\usepackage[normalem]{ulem}
\usepackage{enumitem}
\usepackage{dsfont}
\usepackage{appendix} 
\usepackage{color}
\usepackage{calc}
\usepackage{slantsc}
\usepackage[sc]{mathpazo}
\usepackage{indentfirst}
\usepackage{mathtools}
\usepackage{amscd}
\usepackage{caption}
\usepackage{subcaption}
\usepackage{fancybox}
\usepackage{wrapfig}
\usepackage[all]{xy}
\usepackage{verbatim}
\usepackage{stmaryrd}

\theoremstyle{plain}
\newtheorem{theorem}{Theorem}[section]
\newtheorem{lemma}[theorem]{Lemma}
\newtheorem{corollary}[theorem]{Corollary}

\newtheorem{prop}[theorem]{Proposition}
\newtheorem{defn}[theorem]{Definition}
\newtheorem{example}[theorem]{Example}
\newtheorem{remark}[theorem]{Remark}

\numberwithin{equation}{section}

\newcommand{\Image}{\operatorname{Im}}

\newcommand{\twobar}{/\kern-0.5em/}
\newcommand{\threebar}{/\kern-0.5em/\kern-0.5em/}

\newcommand{\N}{\mathbb{N}}
\newcommand{\Z}{\mathbb{Z}}
\newcommand{\Q}{\mathbb{Q}}
\newcommand{\R}{\mathbb{R}}
\newcommand{\C}{\mathbb{C}}
\newcommand{\bS}{\mathbb{S}}
\newcommand{\bP}{\mathbb{P}}
\newcommand{\bb}{\boldsymbol{b}}

\newcommand{\bL}{\mathbf{L}}

\newcommand{\cS}{\mathcal{S}}

\newcommand{\bT}{\mathbf{T}}
\newcommand{\cI}{\mathcal{I}}
\newcommand{\bbL}{\mathbb{L}}

\newcommand{\wT}{\triangledown}
\newcommand{\blT}{\blacktriangledown}
\newcommand{\gT}{\textcolor{gray}{\blacktriangledown}}

\newcommand{\wunit}{\bm{1}^{\triangledown}}
\newcommand{\bunit}{\bm{1}^{\blacktriangledown}}
\newcommand{\gunit}{\bm{1}^{\textcolor{gray}{\blacktriangledown}}}
\newcommand{\wlambda}{\bm{\lambda}^{\triangledown}}
\newcommand{\blambda}{\bm{\lambda}^{\blacktriangledown}}
\newcommand{\glambda}{\bm{\lambda}^{\textcolor{gray}{\blacktriangledown}}}

\newcommand{\into}{\hookrightarrow}

\newcommand{\conn}{\nabla}

\newcommand{\tr}{\mathrm{tr} \,}
\newcommand{\unu}{\underline{\nu}}

\newcommand{\one}{\mathbf{1}}

\newcommand{\ev}{\mathrm{ev}}

\newcommand{\cM}{\mathcal{M}}

\newcommand{\cL}{\mathcal{L}}
\newcommand{\pt}{\mathrm{pt}}

\newcommand{\fm}{\mathfrak{m}}

\newcommand{\fs}{\mathfrak{s}}
\newcommand{\ff}{\mathfrak{f}}

\newcommand{\re}{\operatorname{Re}}

\newcommand{\Crit}{\mathrm{Crit}}

\newcommand{\isoto}{\xrightarrow{\sim}}

\makeatletter
\def\subsubsection{\@startsection{subsubsection}{3}%
	\z@{.5\linespacing\@plus.7\linespacing}{-.5em}%
	{\normalfont\bfseries}}
\makeatother

\title{$T$-equivariant disc potentials for toric Calabi-Yau manifolds}

\author{Hansol Hong}
\address{Department of Mathematics, Yonsei University, 50 Yonsei-Ro, Seodaemun-Gu, Seoul 03722, Korea} 
\email{hansolhong@yonsei.ac.kr}

\author{Yoosik Kim}
\address{ Department of Mathematics, Brandeis University, 415 South Street
Waltham, MA 02453, USA \& 
Center of Mathematical Sciences and Applications, Harvard University, 20 Garden Street, Cambridge, MA 02138, USA} 
\email{yoosik@brandeis.edu, yoosik@cmsa.fas.harvard.edu}

\author{Siu-Cheong Lau}
\address{Department of Mathematics and Statistics, Boston University, 111 Cummington Mall, Boston MA 02215, USA}
\email{lau@math.bu.edu}

\author{Xiao Zheng}
\address{Department of Mathematics and Statistics, Boston University, 111 Cummington Mall, Boston MA 02215, USA}
\email{xiaoz259@bu.edu}
\begin{document}
	
\begin{abstract}
We study the equivariant disc potentials for immersed SYZ fibers in toric Calabi-Yau manifolds.  The immersed Lagrangians play a crucial role in the partial compactification of the SYZ mirrors.  Morever, their equivariant disc potentials have a close relation with that of Aganagic-Vafa branes.  We show that the potentials can be computed by using an equivariant version of isomorphisms in the Fukaya category.
\end{abstract}

\maketitle


\section{Introduction}
Toric Calabi-Yau varieties form an important class of local models of Calabi-Yau varieties.  They have a very rich Gromov-Witten theory and Donaldson-Thomas theory \cite{MOOP}.  Moreover, they provide interesting classes of local singularities, including the conifold which play an important role in geometric transitions and string theory.  Furthermore, its derived category is deeply related with quiver theory \cite{Bocklandt} and noncommutative resolutions \cite{VdB,S-VdB}.

The Lagrangian branes found by Aganagic-Vafa \cite{AV} provide an important class of objects in the Fukaya category of toric Calabi-Yau threefolds.
Their open Gromov-Witten invariants were predicted by \cite{AV,AV-knots,AKV,BKMP} using physical methods such as large N-duality.  The pioneering works of Katz-Liu \cite{Katz-Liu}, Graber-Zaslow \cite{GZ}, Li-Liu-Liu-Zhou \cite{LLLZ} used $\bS^1$-equivariant localization to formulate and compute these invariants. There have been several recent developments \cite{Fang-Liu,fang-liu-tseng,FLZ} in formulating and proving the physicists' predictions using the localization technique.  There are also vast conjectural generalizations of these invariants in relation with knot theory, see for instance \cite{AENV,TZ}.

\begin{figure}[htb!]
	\includegraphics[scale=0.5]{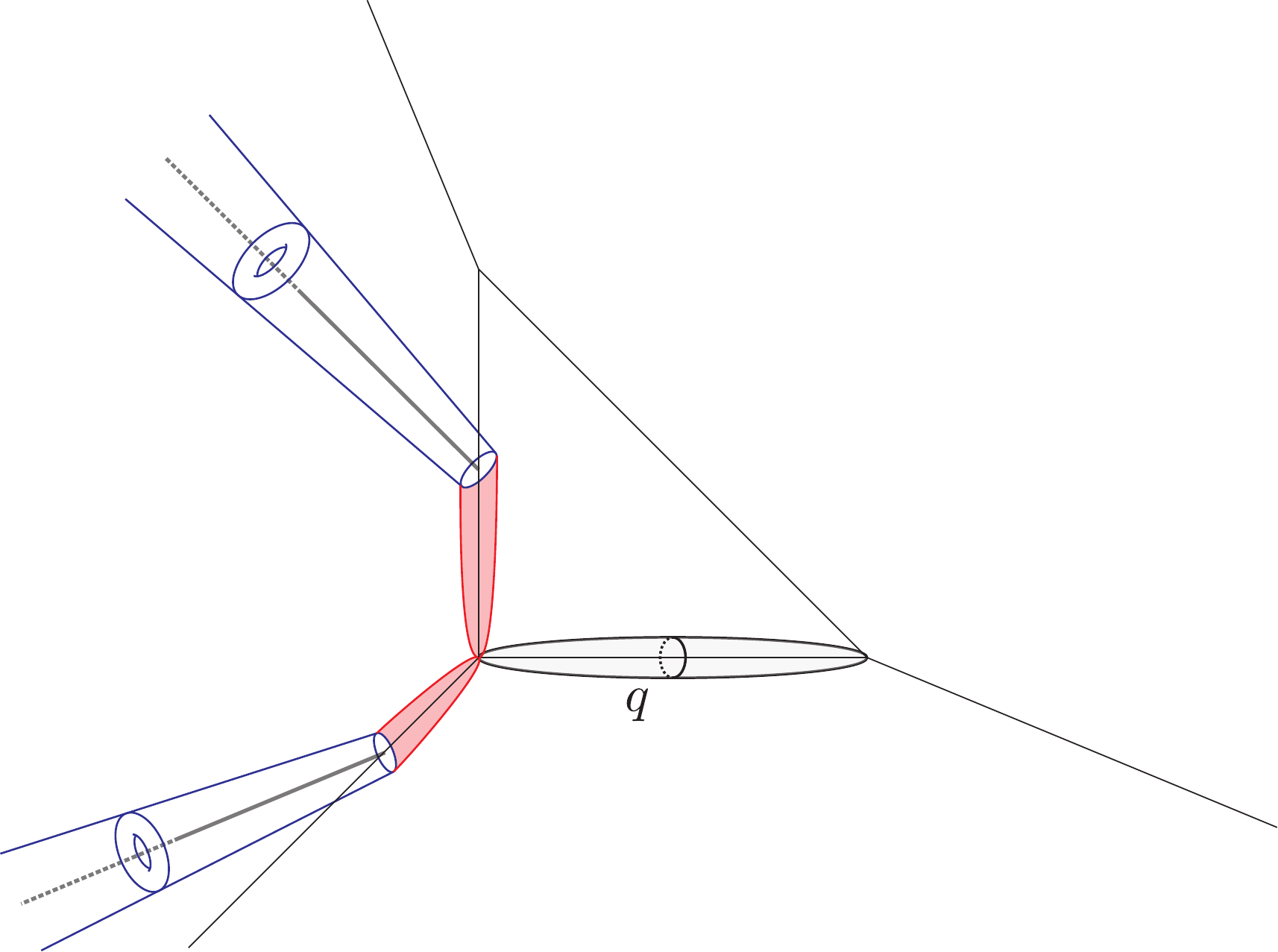}
	\caption{\small{The Aganagic-Vafa branes in $K_{\bP^{2}}$.}}
	\label{fig:KP2rere}
\end{figure}

On the other hand, the third-named author together with his collaborators \cite{CLL,CLT11,CCLT13} proved that the generating functions of open Gromov-Witten invariants for a Lagrangian toric fiber of a toric Calabi-Yau manifold can be computed from the inverse mirror map.  This realizes the $T$-duality approach to mirror symmetry by Strominger-Yau-Zaslow \cite{SYZ}.  The wall-crossing technique \cite{KS-tor,KS-affine,GS1,GS2,GS07, Auroux07} was crucial in the construction.
%
%

In this paper, we aim to relate these two different approaches and their corresponding invariants. It is illustrative to first examine the typical example of a toric Calabi-Yau manifold, that is, the total space of the canonical line bundle of the projective space $\bP^{n-1}$ denoted by $K_{\bP^{n-1}}$.  

There is a $T^{n-1}$-action on $K_{\bP^{n-1}}$ whose symplectic quotients are identified with the complex plane.  The Aganagic-Vafa Lagrangian brane $L^{AV}$ for $n=3$ can be realized as a ray in the moment polytope (see Figure \ref{fig:KP2rere}).
In this dimension, the genus-zero open Gromov-Witten potential of $L^{AV}$ is equal to the integral
\begin{equation}
\int \log (-z_1(z_2,q)) dz_2,
\label{eq:int}
\end{equation}
where $z_1(z_2,q)$ is obtained by solving the mirror curve equation
$$ z_1 + z_2 + \frac{q}{z_1 z_2} + \exp (\phi_3(q)/3) =0. $$
In the expression, $\phi_3(q)$ is the inverse mirror map on the K\"ahler parameter $q$, whose explicit expression can be obtained by solving the associated Picard-Fuchs equation.  The open Gromov-Witten potential was mathematically formulated and derived by \cite{LLLZ,Fang-Liu} via localization.

The SYZ approach uses the Lagrangian torus fibration on a toric Calabi-Yau manifold constructed by \cite{Gross-eg,Goldstein}.  In \cite{CLL}, the mirror dual to this Lagrangian fibration was constructed via wall-crossing, and \cite{CCLT13} computed the generating function of open Gromov-Witten invariants for a Lagrangian toric fiber, which turns out to coincide with the inverse mirror map. Restricting to this example $K_{\bP^{n-1}}$, it states as follows.

\begin{theorem}[\cite{CLL,LLW10,CCLT13,L15}] \label{thm:SYZ}
	The SYZ mirror of $K_{\bP^{n-1}}$ equals to
	\begin{equation}
	uv = z_1 + \ldots + z_{n-1} + \frac{q}{z_1\ldots z_{n-1}} + (1+\delta(q))
	\label{eq:mirKP}
	\end{equation}
	where $(1+\delta(q))$ is the generating function of one-pointed open Gromov-Witten invariants of a moment-map fiber.  Moreover, $(1+\delta(q))$ equals to $\exp (\phi_n(q)/n)$ for the inverse mirror map $\phi_n(q)$.  The right hand side of the above mirror equation equals to the Gross-Siebert's normalized slab function.
\end{theorem}

In this paper, we use the gluing method developed in \cite{CHL-glue, HKL} and the Morse model of equivariant Lagrangian Floer theory \cite{SS10,HLS16a,HLS16b,BH18,DF17} formulated in the recent work \cite{KLZ19} to find the relation between the SYZ mirror \eqref{eq:mirKP} and the potential \eqref{eq:int}.

More specifically, we formulate and compute the disc potentials of certain immersed SYZ fibers in toric Calabi-Yau manifolds.  We are motivated by the important observation that these immersed fibers and the Aganagic-Vafa branes bound the same collection of holomorphic discs.  In addition to these discs, the immersed fibers bound holomorphic polygons which have corners at the immersed sectors.  

The immersed fibers are crucial for the compactification of the SYZ mirrors \cite{HKL}.  Furthermore, weakly unobstructedness is the main technical reason that we focus on the immersed fibers instead of Aganagic-Vafa branes themselves.

The main task is to compute the $\bS^1$-equivariant disc potential for the immersed SYZ fiber.  It plays the role of the genus-zero open Gromov-Witten invariants of Aganagic-Vafa branes.  Note that the stable polygons that we consider have an output marked point, which is important for defining the Lagrangian Floer theory.  The output marked point is acted freely by the $\bS^1$-action.  It obscures the application of localization to compute these invariants.

The following is the main theorem of the paper.
Our approach works in all dimensions and computes the contribution of holomorphic polygons (bounded by the immersed SYZ fiber). In dimension three, the disc part of the $\bS^1$-equivariant potential we obtain agrees with the derivative of the generating function of genus-zero open Gromov-Witten invariants for Aganagic-Vafa branes. 

\begin{theorem}[Theorem \ref{thm:equiv}]\label{thm:ourmain1}
	Let $X$ be a toric Calabi-Yau $n$-fold and $L_0 \cong \cS^2 \times T^{n-2}$ an immersed SYZ fiber intersecting a codimension-two toric stratum, where $\cS^2$ denotes the immersed sphere with a single nodal point.  Recall that the SYZ mirror (for a given choice of a chamber and a basis of $\Z^n$) takes the form of
	\[
	\left\{(u,v,z_1,\ldots,z_{n-1})\in \C^2 \times (\C^\times)^{n-1} \mid uv =f(z_1,\ldots,z_{n-1})\right\}
	\] 
	where $f$ is a Laurent polynomial in the variables $z_1,\ldots,z_{n-1}$ (and also a series in K\"ahler parameters).
	
	Then the $\bS^1$-equivariant disc potential (with respect to the same choice of chamber and basis) takes the form
	\[
	 \lambda \cdot \log g(uv,z_2,\ldots,z_{n-1})
	\]
	where $\lambda$ is the $\bS^1$-equivariant parameter, and $-z_1 = g(uv,z_2,\ldots,z_{n-1})$ is a solution to the defining equation $uv =f(z_1,\ldots,z_{n-1})$ of the mirror.
\end{theorem}

In the above theorem, $f(z_1,\ldots,z_{n-1})$ is the generating function of open Gromov-Witten invariants bounded by a Lagrangian torus fiber.  In \cite{CCLT13}, it was proved that $f(z_1,\ldots,z_{n-1})$ (which is a priori a formal power series in the K\"ahler parameters) is actually convergent over $\C$ and hence serves as a holomorphic function.  

Using the theorem, the equivariant counting of stable polygons can be computed explicitly.  See the tables in Section \ref{sec:eg} for the examples of $K_{\bP^2}$, $K_{\bP^3}$ and local K3 surfaces.  These stable polygons play a crucial role as quantum corrections to the equivariant SYZ mirrors of toric Calabi-Yau manifolds.

There is a delicate dependence of the above equivariant disc potential on the choice of `a chamber and basis'.  First, the immersed Lagrangian is located at a component of the codimension-two toric strata, which are represented by edges in Figure \ref{fig:KP2}.  Different components give different disc potentials.  Second, the potential depends on the direction of the $\bS^1$-action, which is given by a vector parallel to the codimension-two strata.  Third, we also need to fix a Morse function on the immersed Lagrangian to define the disc potential.  The Morse function is fixed by trivializing the SYZ fibration over a chart.  This involves a choice of a basis of $\pi_1(T)$ of a torus $T$, and one of the two adjacent chambers of the codimension-two strata.

The method we use to derive the formula in Theorem \ref{thm:ourmain1} provides an alternative to the localization method for computing open Gromov-Witten invariants.  First, we compute the equivariant disc potential of a Lagrangian torus which is isotopic to a smooth SYZ fiber.   This uses the machinery developed in \cite{KLZ19}.  Second, we derive the gluing formula for the isomorphisms between the formal deformations of the immersed Lagrangian and the torus as objects in the Fukaya category.  The gluing formula is closely related to the expression of the SYZ mirror in Theorem \ref{thm:SYZ}.  Wall-crossing plays a crucial role in deducing the formula.   Finally, applying the gluing formula (which can be understood as an analytic continuation) to the disc potential of the torus gives the above expression for the disc potential of the immersed Lagrangian.


Strictly speaking, Lagrangian Floer theory, and in particular the disc potential above, should be defined over the Novikov ring $\Lambda_0$, where
\[
\Lambda_0 := \left\{\sum_{i=0}^\infty a_i \bT^{A_i} \mid A_i\ge 0 \textrm{ and increases to } +\infty \right\}.
\]
Comparing with the notation in \eqref{eq:int}, the 
K\"ahler parameter $q^C$ of a curve class $C$ is substituted by $\bT^{\omega(C)}$, and $z_i$ for $i=2,\ldots,n$ are replaced by $\bT^{A_i} z_i$ respectively, where $A_i>0$ are symplectic areas of certain primitive discs depending on the position of $L_0$ in the codimension-2 toric strata.  
The leading order term of $g$ with respect to the $T$-adic valuation is $1$, and hence $\log g$ makes sense as a series (where $\log 1 =0$).  

We will also use the maximal ideal
\[
\Lambda_+ := \left\{\sum_{i=0}^\infty a_i \bT^{A_i} \mid A_i>0 \textrm{ and increases to } +\infty \right\}
\]
and the group of invertible elements
\[
\Lambda_\mathrm{U}:=\{a_0 + \lambda: a_0 \in \C^\times \textrm{ and } \lambda \in \Lambda_+ \}.
\]

The organization of the paper is as follows. We review SYZ mirror symmetry for toric Calabi-Yau manifolds and the Morse model of equivariant Lagrangian Floer theory in Section \ref{sec:review}. In Section \ref{sec:immersed_SYZ}, we study the gluing formulas between the immersed SYZ fiber and other Lagrangian tori, and, as a result, obtain the equivariant disc potential of the immersed fiber. Finally, in Section \ref{sec:immtoriequiv}, we compute the potential of a certain immersed Lagrangian torus which plays an important role in the mirror construction of toric Calabi-Yau manifolds.

\addtocontents{toc}{\protect\setcounter{tocdepth}{1}}

\subsection*{Acknowledgment}
The authors express their gratitude to Cheol-Hyun Cho for useful discussions on Floer theory and the gluing scheme.  The third
named author is grateful to Naichung Conan Leung and Eric Zaslow for introducing this subject to him in the early days of his career.  The work of the first named author is supported by the Yonsei University Research Fund of 2019 (2019-22-0008) and  the National Research Foundation of Korea (NRF) grant
funded by the Korea government (MSIT) (No. 2020R1C1C1A01008261). The second named author was supported by the Simons Collaboration Grant on
Homological Mirror Symmetry and Applications.  The third named author was partially supported by the Simons Collaboration Grant.

\section{A review on toric Calabi-Yau manifolds and Lagrangian Floer theory} \label{sec:review}

\subsection{Toric Calabi-Yau manifolds}\label{sec:toric_CY}
 
Let $\textbf{N}\cong \Z^n$ be a lattice of rank $n$ and $\textbf{M}=\textbf{N}^\vee$ the dual lattice. We fix a primitive vector $\unu \in \textbf{M}$ and a closed lattice polytope $\Delta$ of dimension $n-1$ contained in the affine hyperplane $\{v \in \textbf{N}_\R \mid \unu(v) = 1\}$.  By choosing a lattice point $v\in \Delta$, we have a lattice polytope $\Delta - v$ in the hyperplane $\unu^\perp_\R \subset \textbf{N}_\R$.  We triangulate $\Delta$ such that each maximal cell is a standard simplex.  By taking a cone over this triangulation, we obtain a fan $\Sigma$ supported in $\textbf{N}_\R$. Then, $\Sigma$ defines a toric Calabi-Yau manifold $X=X_\Sigma$. 
We denote by $w$ the toric holomorphic function corresponding to $\unu \in \textbf{M}$.

Let $v_1,\ldots,v_m$ be the lattice points in $\Delta$ corresponding to primitive generators of the one-dimensional cones of $\Sigma$. By relabeling if necessary, we may assume that $\{v_1,\ldots,v_n\}$ is a basis of \textbf{N} and generates a maximal cone $\sigma$ of $\Sigma$. 
For each $i \in \{1,\ldots,m\}$, $v_i$ can be uniquely written as $\sum_{\ell=1}^n a_{i,\ell} v_\ell$ for some $a_{i,\ell} \in \Z$. In particular, $a_{i,\ell} = \delta_{i\ell}$ for $i \in \{1,\ldots,n\}$.

We denote by $D_i$ the toric prime divisor corresponding to $v_i$. 
For each toric prime divisor $D_i$ and a Lagrangian toric fiber $L\cong (\bS^1)^n$ in $X$, one can associate a \emph{basic disc class} $\beta_i^L \in \pi_2(X,L)$ represented by a holomorphic disc emanated from $D_i$ and bounded by $L$ (see \cite{CO}). 
It is well-known that $\pi_2(X,L)$ is generated by the basic disc classes, that is, $\pi_2(X,L) \simeq \Z \, \langle \beta^L_1,\ldots,\beta^L_m\rangle$, and there is an exact sequence
\begin{equation}
\label{eq:toric_ses}
0 \to H_2(X;\Z) \to H_2(X,L; \Z) (\cong \pi_2(X,L)) \to H_1(L;\Z) (\cong \textbf{N}) \to 0.
\end{equation}
For a disc class $\beta \in \pi_2(X,L)$, its Maslov index $\mu_L(\beta)$ is equal to $2\sum_{i=1}^m D_i \cdot \beta$ (see \cite{CO,Auroux07}).  In particular, the basic disc classes are of Maslov index two.

For $i=n+1,\ldots,m$, consider the curve class $C_i$ given by
\begin{equation} \label{eq:C}
C_i := \beta_{i}^L - \sum_{\ell=1}^n a_{i,\ell} \beta^L_\ell.
\end{equation}
Then, $\{C_{n+1},\ldots,C_m\}$ forms basis of $H_2(X;\Z)$ and generates the monoid of effective curve classes $H_2^{\mathrm{eff}}(X)\subset H_2(X;\Z)$. The corresponding K\"ahler parameters will be denoted by $q^{C_{n+1}},\ldots,q^{C_m}$ respectively.
Since $X$ is Calabi-Yau, we have $c_1(\alpha) := -K_X \cdot \alpha = \sum_{i=1}^m D_i \cdot \alpha = 0$ for all $\alpha \in H_2(X;\Z)$.  

Let $T^n$ be the $n$-dimensional torus $(\bS^1)^n$ acting on $X$. We have an $(n-1)$-dimensional subtorus $\unu^\perp_\R/\unu^\perp \cong T^{n-1}$ which acts trivially on $-K_X$. The moment map $\rho$ of the $T^{n-1}$-action on $X$ is simply given by the composition of the moment map of the $T^n$-action with the projection to $\textbf{M}_\R/\R\cdot\unu \cong \R^{n-1}$. 
Since the holomorphic function $w$ is invariant under the $T^{n-1}$-action, it descends to a holomorphic function on the symplectic quotients $X\sslash_{a_1} T^{n-1}$ for $a_1 \in \textbf{M}_\R/\R\cdot\unu$, which is in fact an isomorphism $w:X\sslash_{a_1} T^{n-1}\isoto \C$. Then,  the preimage of each embedded loop in $X\sslash_{a_1} T^{n-1}\cong \C$ is a Lagrangian submanifold of $X$ (contained in the level set $\rho^{-1}(a_1)$). This method of constructing Lagrangian submanifolds using symplectic reduction was introduced in \cite{Gross-eg, Goldstein}.

For a circle centered at $w=0$, the corresponding Lagrangian is a regular toric fiber $L$, which bounds holomorphic discs of Maslov index two emanated from the toric prime divisors.  These holomorphic discs exhibit interesting sphere bubbling phenomenon and their classes in $H_2(X,L;\Z)$ are of the form $\beta^L_i + \alpha$, where $\beta^L_i$ is the basic disc class corresponding to primitive generator $v_i$ and $\alpha\in H_2^{\mathrm{eff}}(X)$ is an effective curve class. 

Let $n_1(\beta^L_i+\alpha)$ be the one point open Gromov-Witten invariant associated to the disc class $\beta^L_i+\alpha$. We note that $n_1(\beta^L_i)=1$. In \cite{CCLT13}, the generating functions $1+\delta_i(\bT)$ (known as slab functions for wall-crossing in the Gross-Siebert program \cite{GS07}) defined by
\begin{equation}
1+\delta_i(\bT) = \sum_{\alpha \in H_2^{\mathrm{eff}}(X)} n_1(\beta^L_i+\alpha) \bT^{\omega(\alpha)},
\end{equation}
were proved to be coefficients of the inverse mirror maps which have explicit expressions. 

The following results will play an important role in this paper.

\begin{theorem}[{\cite[Theorem 4.37]{CLL}}]
The SYZ mirror of the toric Calabi-Yau manifold $X$ is given by
\[
X^{\vee}=\{(u,v,z_1,\ldots,z_{n-1})\in \C^2\times (\C^{\times})^{n-1} \mid uv=f(z_1,\ldots,z_{n-1})\},
\]
where $f$ is the generation function of open Gromov-Witten invariants
\begin{equation}\label{eqn:fslabftn}
f(z_1,\ldots,z_{n-1}) = \sum_{i=1}^m \bm{T}^{\omega(\beta^L_i-\beta^L_1)} (1+\delta_i(\bT)) \vec{z}^{v_i'}, \quad v_i'=v_i - v_1.
\end{equation}
\end{theorem}

\begin{theorem}{\cite[Theorem 1.4 restricted to the manifold case]{CCLT13}}
	Given a toric Calabi-Yau manifold $X$ as above, for each $i=1,\ldots,m$, let
	\begin{equation}\label{eqn:funcn_g}
	g_i(\check{q}):=\sum_{\alpha}\frac{(-1)^{(D_i \cdot \alpha)}(-(D_i \cdot \alpha)-1)!}{\prod_{j\neq i} (D_j\cdot \alpha)!}\check{q}^\alpha,
	\end{equation}
	in which the summation is over all effective curve classes $\alpha\in H_2^{\mathrm{eff}}(X)$ satisfying
	\[
	-K_X\cdot \alpha=0, D_i\cdot \alpha <0 \text{ and } D_j\cdot \alpha \geq 0 \text{ for all } j\neq i.
	\]
	Then
	\begin{equation}
	1+\delta_i(q) = \exp g_i (\check{q}(q))
	\end{equation}
	where $\check{q}(q)$ is the inverse of the mirror map
	\[
	q^{C_k}(\check{q}) := \check{q}^{C_k} \cdot \exp \left(-\sum_i (C_k,D_i) g_i(\check{q})\right), \quad k=n+1,\ldots,m.
	\]
	Here $C_{n+1},\ldots,C_m$ are the curve classes given by equation \eqref{eq:C}.
\end{theorem}


\subsection{Immersed Lagrangian Floer theory}\label{sec:Morse model}

Let $(X,\omega)$ be a $2n$-dimensional tame (compact, convex at infinity or geometrically bounded) symplectic manifold.
Let $L$ be a closed, connected, relatively spin, and immersed Lagrangian submanifold of $X$ with clean self-intersections. We denote by $\iota:\tilde{L}\to X$ an immersion with image $L$ and by $\cI\subset L$ the self-intersection.
As in Akaho-Joyce \cite{AJ}, the inverse image of $\cI$ under the immersion $\iota$ is assumed to be the disjoint union $\cI^-\coprod \cI^+\subset \tilde{L}$ of \emph{two branches}, each of which is diffeomorphic to $\cI$. Then the following fiber product
\[
\tilde{L} \times_L \tilde{L}=\coprod_{i=-1,0,1} R_i,
\]
consists of the diagonal component $R_0$, and the two \emph{immersed sectors} 
\begin{equation}\label{eqn:r1r-1}
\begin{array}{l}
R_{1}=\{(p_{-}, p_+)\in \tilde{L} \times \tilde{L} \mid p_-\in\cI^-, p_+\in\cI^+, \iota(p_-)=\iota(p_+)  \}, \\
R_{-1}=\{(p_{+}, p_{-})\in \tilde{L} \times \tilde{L} \mid p_+\in\cI^+, p_-\in\cI^-, \iota(p_+)=\iota(p_-) \}.
\end{array}
\end{equation}
We have canonical isomorphisms $R_0\cong\tilde{L}$ and $R_{-1}\cong R_{1}\cong \cI$. 
Also, we have the involution $\sigma:R_{-1}\coprod R_{1} \to R_{-1}\coprod R_{1}$ swapping the two immersed sectors, i.e, $\sigma(p_{-},p_{+})=(p_{+},p_{-})$. 
 
In \cite{AJ}, Akaho and Joyce developed immersed Lagrangian Floer theory on the singular chain model. For an immersed Lagrangian $L$ with transverse self-intersection (which is a nodal point), they produced an $A_{\infty}$-algebra structure on a countably generated subcomplex $C^{\bullet}(L;\Lambda_0)$ of the smooth singular cochain complex $S^{\bullet}(\tilde{L} \times_L \tilde{L};\Lambda_0)$, generalizing the work of Fukaya-Oh-Ohta-Ono \cite{FOOO} in the embedded case. In the following, we describe a further generalization of their construction to the case where $L$ is an immersed Lagrangian with clean self-intersection as described above. (See \cite{Sch-clean, CW15, Fuk_gauge} for the development of Floer theory of clean intersections in different settings.) 

Let us choose a compatible almost complex structure $J$ on $(X,\omega)$. For $\alpha:\{0,\ldots, k\}\to \{-1,0,1\}$, we consider quintuples $(\Sigma,\vec{z},u,\tilde{u},l)$ where
\begin{itemize}
\item $\Sigma$ is a prestable genus $0$ bordered Riemann surface,
\item $\vec{z}=(z_0,\ldots,z_k)$ are distinct counter-clockwise ordered smooth points on $\partial\Sigma$, 
\item $u:(\Sigma,\partial \Sigma)\to (X,L)$ is a $J$-holomorphic map with $(\Sigma,\vec{z},u)$ stable,
\item $\tilde{u}:\bS^1 \setminus \{\zeta_i:=l^{-1}(z_i):\alpha(i)\ne 0\} \to \tilde{L}$ is a local lift of $u|_{\partial\Sigma}$, i.e., 
\[
\iota\circ \tilde{u}=u\circ l \mbox{ and } \left(\lim_{\theta\to 0^-}\tilde{u}(e^{\mathbf{i}\theta} \zeta_i),\lim_{\theta \to 0^+}\tilde{u}(e^{\mathbf{i}\theta} \zeta_i)\right)\in R_{\alpha(i)},
\]
where $\alpha(i)\ne 0$ and $\mathbf{i}:=\sqrt{-1}$. 
\item $l:\bS^1 \to \partial \Sigma$ is an orientation preserving continuous map (unique up to a reparametrization) characterized by that the inverse image of a smooth point is a point and the inverse image of a singular point consists of two points. 
\end{itemize}
Let $[\Sigma,\vec{z},u,\tilde{u},l]$ be the equivalence class of $(\Sigma,\vec{z},u,\tilde{u},l)$ up to automorphisms. 
For $\beta\in H_2(X,L;\Z)$ and $\alpha$ a map as described above, we denote by $\cM_{k+1}(\alpha,\beta)$ the moduli space of such quintuples $[\Sigma,\vec{z},u,\tilde{u},l]$ satisfying $u_*([\Sigma])=\beta$. The moduli spaces come with the evaluation maps $\ev_i:\cM_{k+1}(\alpha,\beta) \to \tilde{L} \times_L \tilde{L}$ defined by
\[
\ev_i([\Sigma,\vec{z},u,\tilde{u},l])=\begin{cases} \tilde{u}(z_i)\in R_0 & \alpha(i)=0 \\
\left(\lim_{\theta\to 0^-}\tilde{u}(e^{\mathbf{i}\theta} \zeta_i),\lim_{\theta \to 0^+}\tilde{u}(e^{\mathbf{i}\theta} \zeta_i)\right)\in R_{\alpha(i)} & \alpha(i)\ne 0,
\end{cases}
\]
at the input marked points $i=1,\ldots, k$, and 
\[
\ev_0([\Sigma,\vec{z},u,\tilde{u},l])=\begin{cases} \tilde{u}(z_0)\in R_0 & \alpha(0)=0 \\
\sigma\left(\lim_{\theta\to 0^-}\tilde{u}(e^{\mathbf{i}\theta} \zeta_0),\lim_{\theta \to 0^+}\tilde{u}(e^{\mathbf{i}\theta} \zeta_0)\right)\in R_{-\alpha(0)} & \alpha(0)\ne 0,
\end{cases}
\]
at the output marked point. 

For the convenience of writing, we will call an element of $\cM_{k+1}(\alpha,\beta)$ a \emph{stable polygon} if $\alpha(i)\ne 0$ for some $i\in\{0,\ldots,k\}$. In this case, the \emph{corners} of a polygon are the boundary marked points $z_i$ with $\alpha(i)\ne 0$. If $\alpha(i)=0$ for all $i$, we will simply refer to an element of $\cM_{k+1}(\alpha,\beta)$ as a \emph{stable disc}.

The Kuranishi structures on $\cM_{k+1}(\alpha,\beta)$ are taken to be \textit{weakly submersive}, which means that the evaluation map $\ev=(\ev_1,\ldots,\ev_k)$ from the Kuranishi neighborhood of each $[\Sigma,\vec{z},u,\tilde{u},l]\in \cM_{k+1}(\alpha,\beta)$ to $\tilde{L} \times_L \tilde{L}$ is a submersion. 

The Floer complex $C^{\bullet}(L;\Lambda_0)$ is a quasi-isomorphic subcomplex of  the singular cochain complex $S^{\bullet}(\tilde{L} \times_L \tilde{L};\Lambda_0)$ inductively constructed as in \cite{FOOO,AJ}. For a $k$-tuple $\vec{P}=(P_1,\ldots, P_k)$  of singular chains $P_1,\ldots,P_k\in C^{\bullet}(L;\Q)$, 
we denote by $\cM_{k+1}(\alpha,\beta,\vec{P})$ the fiber product 
\[
\cM_{k+1}(\alpha,\beta,\vec{P})=\cM_{k+1}(\alpha,\beta) \times_{(\tilde{L} \times_L \tilde{L})^k} \vec{P}
\]
in the sense of Kuranishi structures. We write $\mathcal{M}_{k+1}(\alpha,\beta;\vec{P})^{\fs}=\fs^{-1}(0)$, where $\fs$ is a multi-valued section of the obstruction bundle $E$ transversal to the zero section, chosen in the inductive construction of $C^{\bullet}(L;\Lambda_0)$. 

The $A_\infty$-structure maps $\{\tilde{\fm}_{k}\}$ for $k\ge 0$ are defined as follows. 
For the constant disc class $\beta_0$ and $k = 0, 1$, we set
\[
\begin{cases}
\tilde{\fm}_{0,\beta_0}(1)= 0,  \\
\tilde{\fm}_{1,\beta_0}(P)=(-1)^n \partial P,
\end{cases}
\]
where $\partial$ is the coboundary operator on $C^{\bullet}(L;\Lambda_0)$. For $(k,\beta)\ne (1,\beta_0), (0, \beta_0)$, we set 
\begin{equation*}
\tilde{\fm}_{k,(\alpha,\beta)}(P_1,\ldots,P_k)=(\ev_0)_*\left(\mathcal{M}_{k+1}(\alpha,\beta;\vec{P})^{\fs}\right),
\end{equation*}
and
\begin{equation*}
\tilde{\fm}_{k,\beta}(P_1,\ldots,P_k)=\sum_{\alpha } \tilde{\fm}_{k,(\alpha,\beta)}(P_1,\ldots,P_k).
\end{equation*}
Notice that $\fm_{k,(\alpha,\beta)}(P_1,\ldots,P_k)= 0$ unless $P_i$ is a singular chain on $R_{\alpha(i)}$ for all $i$. 
The map $\tilde{\fm}_k:C^{\bullet}(L;\Lambda_0)^{\otimes k}\to C^{\bullet}(L;\Lambda_0)$ is defined by
\begin{equation*}
\tilde{\fm}_k(P_1,\ldots,P_k)=\sum_{\beta \in H_2^{\mathrm{eff}}(X,L)} \bT^{\omega(\beta)}\tilde{\fm}_{k,\beta}(P_1,\ldots,P_k).
\end{equation*}
Here $H^{\mathrm{eff}}_2(X,L)$ denotes the monoid of effective disc classes in $H_2(X,L;\Z)$. 

\subsubsection{Anti-symplectic involutions}
Let $\tau:X\to X$ be an anti-symplectic involution, i.e.,  $\tau^*\omega=-\omega$. We consider a $\tau$-invariant immersed Lagrangian $L$ such that the immersed locus $\cI$ is also $\tau$-invariant. Then $\tau|_L$ lifts to a diffeomorphism $\tilde{\tau}:\tilde{L}\to \tilde{L}$ satisfying $\tilde{\tau}^2=id$, $\tilde{\tau}(\cI^-)=\cI^{+}$, and $\tilde{\tau}(\cI^+)=\cI^{-}$. Then $\tau$ induces the involution $\sigma:R_{-1}\coprod R_{1}\to R_{-1}\coprod R_{1}$ swapping the immersed sectors. 
Suppose the compatible almost complex structure $J$ is $\tau$-anti-invariant, i.e., $\tau^*J=-J$. Then $\tau$ induces an involution on the moduli spaces.

Let us fix a non-negative integer $k$, $\alpha:\{0,\ldots, k\}\to \{-1,0,1\}$, and $\beta \in H_2^{\mathrm{eff}}(X,L)$. Let $[\Sigma,\vec{z},u,\tilde{u},l]\in \cM_{k+1}(\alpha,\beta)$. We define $\hat{u}:(\Sigma,\partial \Sigma)\to (X,L)$
by
\[
\hat{u}(z):=\tau\circ u(\bar{z}). 
\]

Let $\hat{\alpha}:\{0,\ldots,k\}\to \{-1,0,1\}$ be given by $\hat{\alpha}(0)=\alpha(0)$ and $\hat{\alpha}(i)=\alpha(k+1-i)$ for $i=1,\ldots,k$.  We put $\hat{\vec{z}}=(\hat{z}_0,\ldots,\hat{z}_k) := (\bar{z}_0,\bar{z}_k,\bar{z}_{k-1},\ldots,\bar{z}_1)$ and define $\hat{\tilde{u}} : \bS^1 \setminus \{ \bar{\zeta_i}:=l^{-1}(\bar{z}_i) \mid \hat{\alpha}(i)\ne 0\} \to \tilde{L}$ by
\[
\hat{\tilde{u}}(z) := \tilde{\tau}\circ \tilde{u}(\bar{z}). 
\]
Then $\hat{\tilde{u}}$ satisfies 
\[
\iota\circ \hat{\tilde{u}}=\hat{u}\circ l
\] 
and
\[
\left(\lim_{\theta\to 0^-}\hat{\tilde{u}}(e^{\mathbf{i}\theta} \bar{\zeta}_i),\lim_{\theta \to 0^+}\hat{\tilde{u}}(e^{\mathbf{i}\theta} \bar{\zeta}_i)\right)\in R_{\hat{\alpha}(i)}, \quad \hat{\alpha}(i)\ne 0.
\]
Note that both the complex conjugation on the domain and the involution on $X$ swap the immersed sectors and $\hat{\alpha}$ is obtained from $\alpha$ simply by relabeling the boundary marked points.

For $\beta=[u]$, setting $\hat{\beta}=[\hat{u}]$, the map $\tau^{main}_*:\cM_{k+1}(\alpha,\beta)\to \cM_{k+1}(\hat{\alpha},\hat{\beta})$ defined by
\[
[\Sigma,\vec{z},u,\tilde{u},l]\mapsto [\Sigma,\hat{\vec{z}},\hat{u},\hat{\tilde{u}},l]
\]
is a homeomorphism (of topological spaces) satisfying $\tau^{main}\circ \tau^{main}=id$. We can then choose Kuranishi structures respecting the involution $\tau$ as follows.

\begin{theorem}[{\cite[Theorem 4.11]{FOOO-anti}}] 
\label{thm:tau_*}
The map $\tau^{main}_*$ is induced by an isomorphism of Kuranishi structures.
\end{theorem}	

The choice of a relative spin structure on $L$ together with the choice of a path in the Lagrangian Grassmannian of $T_{p}X$ (for a base point $p$ in each connect component of $\cI$) connecting $d\iota (T_{p_-}L)$ and $d\iota (T_{p_+}L)$, $\iota(p_-)=\iota(p_+)=p$, determine orientations on the moduli spaces $\cM_{k+1}(\alpha,\beta)$ \cite[Section 5]{AJ} (see also \cite[Chapter 8.8]{FOOO} for the Bott-Morse version). We will fix a choice of connecting paths for the following discussion.

Let $\sigma\in H_2(X,L;\Z/2\Z)$ be a (stable conjugacy class of) relative spin structure on $L$.  We write $\cM_{k+1}(\alpha,\beta)^{\sigma}$ to emphasize the Kuranishi structure $\cM_{k+1}(\alpha,\beta)$ is equipped with the orientation determined by $\sigma$. Let $\cM_{k+1}^{\mathrm{clock}}(\alpha,\beta)$ denote the moduli space with the boundary marked points respecting the \textit{clockwise} order and put $\vec{\bar{z}}=(\bar{z}_0,\ldots,\bar{z}_k)$. By {\cite[Theorem 4.10]{FOOO-anti}, the map
\begin{equation}\label{eq:orientation}
\begin{aligned}
\tau_*:\cM_{k+1}(\alpha,\beta)^{\tau^*\sigma} &\to \cM_{k+1}^{\mathrm{clock}}(\alpha,\hat{\beta})^{\sigma},\\
[\Sigma,\vec{z},u,\tilde{u},l] &\mapsto [\Sigma, \vec{\bar{z}},\hat{u},\hat{\tilde{u}},l],
\end{aligned}
\end{equation}
is an orientation preserving isomorphism of Kuranishi structures if and only if $\mu_L(\beta)/2+k+1$ is even.

Let $P_1,\ldots,P_k\in C^{\bullet}(R_{-1}\coprod R_1)\subset C^{\bullet}(L;\Lambda_0)$ be singular chains on the immersed sectors. We have an isomorphism
\begin{equation}
\label{eq:inv}
\tau^{main}_*: \cM_{k+1}(\alpha,\beta;P_1,\ldots,P_k)^{\tau^*\sigma}\to  \cM_{k+1}(\hat{\alpha},\hat{\beta};P_k,\ldots,P_1)^{\sigma}
\end{equation}
of Kuranishi structures, satisfying $\tau^{main}\circ \tau^{main}=id$.  

Let $\cM_{k+1}^{\mathrm{unordered}}(\alpha,\beta;P_1,\ldots,P_k)$ be the moduli space with unordered boundary marked points. 
Note that the unordered moduli space 
contains both $\cM_{k+1}(\alpha,\beta;P_1,\ldots,P_k)$ and $\cM_{k+1}^{\mathrm{clock}}(\alpha,\beta;P_1,\ldots,P_k)$ as connected components. Let $\{i,i+1\}\subset \{1,\ldots,k\}$ and let $\alpha^{i\leftrightarrow i+1}:\{0,\ldots,k\}\to \{-1,0,1\}$ be the map obtained from $\alpha$ by swapping $\alpha(i)$ and $\alpha(i+1)$. By \cite[Lemma 3.17]{FOOO-anti}, the action of changing the order of marked points induces an orientation preserving isomorphism
\begin{equation} \label{eq:reordering}
\begin{array}{l}
\cM_{k+1}^{\mathrm{unordered}}(\alpha,\beta;P_1,\ldots,P_i,P_{i+1},\ldots,P_k) \\\xrightarrow{\sim} (-1)^{(\deg P_i+1) (\deg P_{i+1}+1)} \cM_{k+1}^{\mathrm{unordered}}(\alpha^{i\leftrightarrow i+1},\beta;P_1,\ldots,P_{i+1},P_i,\ldots,P_k).
\end{array}
\end{equation}
Combining \eqref{eq:orientation} and \eqref{eq:reordering}, we derive the following theorem. 
\begin{theorem}
\label{thm:inv}
The map \eqref{eq:inv} is orientation preserving (resp. reversing) if $\epsilon$ is even (resp. odd) where
\begin{equation}
\label{eq:epsilon}
\epsilon=\cfrac{\mu_L(\beta)}{2}+k+1+\sum_{1\le i<j\le k} (\deg P_i+1) (\deg P_j+1).
\end{equation}
\end{theorem}

\subsubsection{Pearl complexes} \label{sec:pearl_complex}
Let us fix a Morse function $f:\tilde{L}\times_L \tilde{L} \to \R$. 
For simplicity, we will assume $f$ has a unique maximum point $\bunit$ on $\tilde{L}$. 
Let $C^{\bullet}(f;\Lambda_0)$ be the cochain complex generated by the critical points of $f$. Namely, 
\begin{equation}\label{equ_morsecri}
C^{\bullet}(f;\Lambda_0)=\bigoplus_{p\in \Crit(f)} \Lambda_0 \cdot p.
\end{equation}
We begin by setting up some notations. We denote
\begin{itemize}
\item by $\mathscr{V}$ a (negative) pseudo gradient vector field of $f$ satisfying the Smale condition,
\item by $\Phi_t$ the flow of $\mathscr{V}$
\item by $W^s(p)$ (resp. $W^u(p)$) the stable (resp. unstable) submanifold of $p \in \Crit(f)$, 
\item by $\overline{W^s(p)}$ (resp. $\overline{W^u(p)}$) the natural compactification of $W^s(p)$ (resp. $W^u(p)$) to a smooth manifold with corners.
\end{itemize}

The $A_{\infty}$-structure maps on the complex $C^{\bullet}(f;\Lambda_0)$ can be defined by counting configurations called \emph{pearly trees} (see Figure~\ref{fig:pearl} for an illustration). This was systematically developed by Biran and Cornea \cite{BC-pearl, BC-survey} under the monotonicity assumption on Lagrangians (a similar complex previously appeared in \cite{Oh96R}).
Since we will not restrict ourselves to the setting of monotone Lagrangians, we shall derive the Morse model from the singular chain model using the homological method as in Fukaya-Oh-Ohta-Ono \cite{FOOO-can}.

We identify $C^{\bullet}(f;\Lambda_0)$ with a subcomplex $C^{\bullet}_{(-1)}(L;\Lambda_0)$ of $C^{\bullet}(L;\Lambda_0)$ is generated by certain singular chains $\Delta_p$ representing $\overline{W^u(p)}$ for $p\in\Crit(f)$.
Here, the chain $\Delta_p$ is chosen so that the assignment $p\to \Delta_p$ is a chain map inducing an isomorphism on cohomology (see \cite[Theorem 2.3]{KLZ19}). We will refer to $\Delta_p$ as the the \textit{unstable chain} of $p$, and implicitly identify  $C^{\bullet}(f;\Lambda_0)$ with $C^{\bullet}_{(-1)}(L;\Lambda_0)$ from now on.

The $A_{\infty}$-maps $\fm^{\blT}:=\{\fm^{\blT}_k\}_{k\ge 0}$ on $C^{\bullet}(f;\Lambda_0)$  are defined in terms of \textit{decorated planar rooted trees}, for which we recall the definition below.

\begin{defn}[\cite{FOOO-can}]
		A \emph{decorated planar rooted tree} is a quintet $\Gamma=(T,\iota,v_0,V_{tad},\eta)$ consisting of 
	\begin{itemize}
		\item $T$ is a tree;
		\item $\iota:T\to D^2$ is an embedding into the unit disc; 
		\item $v_0$ is the root vertex and $\iota(v_0)\in\partial D^2$;
		\item $V_{tad}$ is the set of interior vertices with valency $1$;
		\item $\eta=(\eta_1,\eta_2):V(\Gamma)_{int}\to \Z_{\ge 0}\oplus \Z_{\ge 0}$,
	\end{itemize}
	where $V(\Gamma)$ is the set of vertices, $V(\Gamma)_{ext}=\iota^{-1}(\partial D^2)$ is the set of exterior vertices and $V(\Gamma)_{int}=V(\Gamma)\setminus V(\Gamma)_{ext}$ is the set of interior vertices. For $k \ge 0$, denote by $\bm{\Gamma}_{k+1}$ the set of isotopy classes represented by $\Gamma=(T,\iota,v_0,\eta)$ with $|V(\Gamma)_{ext}|=k+1$ and $\eta(v)>0$ if the valency $\ell(v)$ of $v$ is $1$ or $2$. In other words, the elements of $\bm{\Gamma}_{k+1}$ are stable. We will refer to the elements of $\bm{\Gamma}_{k+1}$ as stable trees.
\end{defn}

For each $\Gamma\in \bm{\Gamma}_{k+1}$, we label the exterior vertices by $v_0,\ldots,v_k$ respecting the counter-clockwise orientation, and orient the 
edges along the direction from the $k$ input vertices $v_1,\ldots,v_k$ towards the root vertex $v_0$. 

We define the map $\Pi:C^{\bullet}(L;\Lambda_0)\to C^{\bullet}(f;\Lambda_0)$ by
\[
\Pi(P)=
\begin{cases}
\displaystyle \sum_{\substack{p\in \Crit(f),\\ \deg p=n-\deg P}} \sharp (P\cap W^s(p))\cdot \Delta_p &\mbox{if the intersection $P\cap W^s(p)$ is transverse}, \\
\quad \quad \quad \quad \quad 0 &\mbox{if the intersection $P\cap W^s(p)$ is \emph{not} transverse}.
\end{cases}
\]
For the map $G:C^{\bullet}(L;\Lambda_0)\to C^{\bullet-1}(L;\Lambda_0)$, $G(P)$ is defined by a singular chain representing the \textit{forward orbit} of $P$ satisfying
\begin{equation}
\label{eq:hp}
\Pi(P)-P=\partial G(P)+G(\partial P).
\end{equation}
See \cite[Theorem 2.3]{KLZ19} for the construction of $G(P)$.

Denote by $\Gamma^0\in \bm{\Gamma}_2$ the unique tree with no interior vertices.
For this tree $\Gamma^0$, we define
\[
\begin{split}
\fm_{\Gamma_0}: C^{\bullet}(f;\Lambda_0)\to C^{\bullet}(f;\Lambda_0) &\mbox{ by $\fm_{\Gamma_0}:= \tilde{\fm}_{1,\beta_0}$}, \\
\ff_{\Gamma_0}:C^{\bullet}(f;\Lambda_0)\hookrightarrow C^{\bullet}(L;\Lambda_0) &\mbox{ by the inclusion}. 
\end{split} 
\]
For each $k\ge 0$, $\bm{\Gamma}_{k+1}$ contains a unique element that has a single interior vertex $v$, which we denote by $\Gamma_{k+1}$. Let $\bm{\alpha}_{k+1}$ denote the set of maps $\alpha:\{0,\ldots, k\}\to \{-1,0,1\}$. We fix a labeling $\{\beta_0,\beta_1,\ldots\}$ of elements of $H_2^{\mathrm{eff}}(X,L)$ with $\beta_0$ the constant disc class and labeling $\{\alpha_{k+1,0},\alpha_{k+1,1},\ldots\}$ of elements of $\bm{\alpha}_{k+1}$ with $\alpha_{k+1,0}$ the map $\alpha_{k+1,0}(i)=0$ for $i=1,\ldots,k$. We define
\[
\fm_{\Gamma_{k+1}} := \Pi\circ \tilde{\fm}_{k,(\alpha_{\ell(v),\eta_1(v)},\beta_{\eta_2(v)})} \mbox{ and }
\ff_{\Gamma_{k+1}} := G\circ \tilde{\fm}_{k,(\alpha_{\ell(v),\eta_1(v)},\beta_{\eta_2(v)})}.
\]

For a general rooted tree $\Gamma\in \bm{\Gamma}_{k+1}$, cut it at the vertex $v$ closest to the root vertex $v_0$ so that $\Gamma$ is decomposed into $\Gamma^{(1)},\ldots,\Gamma^{(\ell)}$ and an interval adjacent to $v_0$ in the counter-clockwise order. The maps $\fm_{\Gamma}:C^{\bullet}(f;\Lambda_0)^{\otimes k}\to C^{\bullet}(f;\Lambda_0)$ and $\ff_{\Gamma}:C^{\bullet}(f;\Lambda_0)^{\otimes k}\to C^{\bullet}(L;\Lambda_0)$ are inductively defined by
\begin{equation}\label{mgammaeq}
\fm_{\Gamma} := \Pi\circ \tilde{\fm}_{\ell,(\alpha_{\ell(v),\eta_1(v)},\beta_{\eta_2(v)})}\circ (\ff_{\Gamma^{(1)}} \otimes\ldots\otimes\ff_{\Gamma^{(\ell)}})
\end{equation}
and 
\begin{equation}\label{fgammaeq}
\ff_{\Gamma} := G\circ\tilde{\fm}_{\ell,(\alpha_{\ell(v),\eta_1(v)},\beta_{\eta_2(v)})}\circ (\ff_{\Gamma^{(1)}} \otimes\ldots\otimes\ff_{\Gamma^{(\ell)}}).
\end{equation}
At last, we define the $A_{\infty}$-maps $\fm^{\blT}_k:C^{\bullet}(f;\Lambda_0)^{\otimes k}\to C^{\bullet}(f;\Lambda_0)$ by 

\begin{equation}
\label{eq:m_k}
\fm^{\blT}_k := \sum_{\Gamma\in\bm{\Gamma}_{k+1}} \bT^{\omega(\Gamma)} \fm_{\Gamma},
\end{equation}
where $\omega(\Gamma)=\sum_{v} \omega(\beta_{\eta_2(v)})$.

The $A_\infty$-algebra $(C^{\bullet}(f;\Lambda_0),\fm^{\blT})$ does \emph{not} have a strict unit in general. In \cite{KLZ19}, a unital $A_\infty$-algebra $(CF^{\bullet}(L;\Lambda_0),\fm)$ was constructed from  $(C^{\bullet}(f;\Lambda_0),\fm^{\blT})$ by applying the homotopy unit construction \cite[Chapter 7]{FOOO}. 
We recall the key properties of $(CF^{\bullet}(L;\Lambda_0),\fm)$ below.

\begin{itemize}
\item We have
\begin{equation}\label{equ_cfllambda}
CF^{\bullet}(L;\Lambda_0)=C^{\bullet}(f;\Lambda_0)\oplus \Lambda_0 \cdot \wunit\oplus \Lambda_0 \cdot \gunit
\end{equation}
as graded modules. $\wunit$ and $\gunit$ are generators in degree $0$ and $-1$, respectively. 
\item The restriction of $\fm$ to $C^{\bullet}(f;\Lambda_0)$ agrees $\fm^{\blT}$.
\item $\wunit$ is the strict unit, i.e.,
\[
\fm_2(\wunit,x)=(-1)^{\deg x} \fm_2(x,\wunit)=x, 
\] 
for $x\in CF^{\bullet}(L;\Lambda_0)$, and 
\[
\fm_k(\ldots,\wunit,\ldots)=0
\]
for $k\ge 2$.
\item Assuming the minimal Maslov index of $L$ is nonnegative, we have
\[
\fm_1(\gunit)=\wunit-(1-h)\bunit, \quad h\in \Lambda_+.
\]
\end{itemize}
The maximum point $\bunit$ is a homotopy unit in the sense of \cite[Definition 3.3.2]{FOOO}.\\

For a pair $(L_1,L_2)$ of closed, connected, relatively spin, and embedded Lagrangian submanifolds intersecting cleanly, the union $L=L_1\cup L_2$ is an immersed Lagrangian with clean self-intersections and with normalization $\tilde{L}=L_1\coprod L_2$. We choose the splitting $\iota^{-1}(\cI)=\cI^-\coprod \cI^+$ so that $\cI^-\subset L_1$ and $\cI^+\subset L_2$. In this case, $R_1$ indicates a branch jump from $L_1$ to $L_2$ and $R_{-1}$ indicates a branch jump from $L_2$ to $L_1$, see \eqref{eqn:r1r-1}. 

We can define a pearl complex $(CF^{\bullet}(L_1,L_2;\Lambda_0),\fm_1)$ for the Lagrangian intersection Floer theory of $(L_1,L_2)$ as follows: 
\[
CF^{\bullet}(L_1,L_2;\Lambda_0)=\bigoplus_{p\in \Crit(f|_{R_{1}}) } \Lambda_0 \cdot p
\]
is the subcomplex of $CF^{\bullet}(L;\Lambda_0)$ generated by critical points of $f$ in $R_1$, and the differential $\fm_1$ counts stable pearly trees in $\bm{\Gamma}_2$ with both input and output vertices in $R_1$.

\begin{figure}[h]
	\begin{center}
		\includegraphics[scale=0.5]{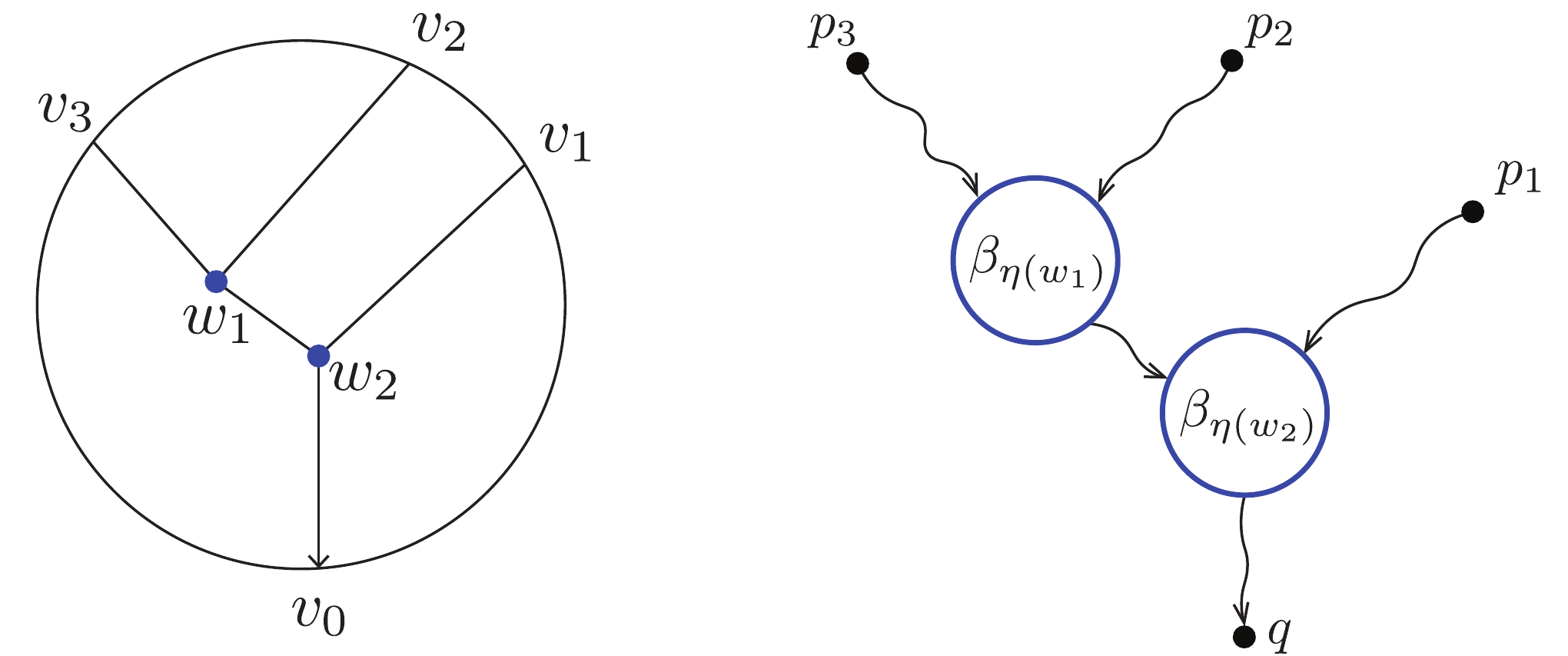}
		\caption{A pearly tree}\label{fig:pearl}
	\end{center}
\end{figure}


\subsubsection{Maurer-Cartan space}
Let $(A,\fm)$ be an $A_{\infty}$-algebra over $\Lambda_0$ with the strict unit $e_A$,
and set 
\[
A_+=\{x\in A \mid x\equiv 0 \mod \Lambda_+\cdot A \}.
\]
The \emph{weak Maurer-Cartan equation} for an element $b\in A_+$ is given by
\begin{equation*}
\label{eq:MC}
\fm^b_0(1)=\fm_0(1)+\fm_1(b)+\fm_{2}(b,b)+\ldots \in \Lambda_0 \cdot e_A.
\end{equation*}
The condition that $b\in A_+$ ensures the convergence of $\fm^b_0(1)$. A solution $b$ of \eqref{eq:MC} is called a \emph{weak bounding cochain}. We denote by 
\begin{equation*}
\mathcal{MC}(A)= \left\{b\in A^{odd}_+ \mid \fm_0^b(1)\in \Lambda_0 \cdot e_A \right\},
\end{equation*}
the space of weak Maurer-Cartan elements. 
We say an $A_{\infty}$-algebra $A$ is \emph{weakly unobstructed} if $MC(A)$ is nonempty, in which case we have $(\fm^b_1)^2=0$ for any $b\in \mathcal{MC}(A)$, thus defining a cohomology theory $H^{\bullet}(A,\fm^b_1)$. 

Let us put 
\begin{equation*}
e^b:=1+b+b\otimes b+\ldots.
\end{equation*}
For an element $b\in \mathcal{MC}(A)$, we can define a deformation $\fm^b$ of the $A_{\infty}$-structure $\fm$ by
\begin{equation} \label{eq:deformation}
\fm^b_k(x_1,\ldots,x_k)=\fm(e^b,x_1,e^b,x_2,e^b,\ldots,e^b,x_k,e^b), \quad x_1,\ldots,x_k\in A.
\end{equation}

Now, for $(A,\fm)=(CF^{\bullet}(L;\Lambda_0),\fm)$, we denote by $\mathcal{MC}(L)$ the space of odd degree 
weak bounding cochains. 

We say that the deformation of $L$ by $b$ (or simply $(L,b)$) is \emph{unobstructed} (resp. \emph{weakly unobstructed}), if $\mathfrak{m}^{b}(1)=0$ (resp. $b\in \mathcal{MC}(L)$). In particular, if $b=0$, we will simply say $L$ is unobstructed.

The following lemma concerns the weakly unobstructedness of $(CF^{\bullet}(L;\Lambda_0),\fm)$. This technique of finding weak bounding cochains in the presence of a homotopy unit was introduced in \cite[Chapter 7]{FOOO} and \cite[Lemma 2.44]{CW15}.

\begin{lemma}[{\cite[Lemma 2.7]{KLZ19}}]
	\label{lem:unobstr}
Let $b\in CF^{1}(L;\Lambda_+)$. Suppose $\fm_0^b=W(b)\bunit$ and the minimal Maslov index of $L$ is nonnegative. Then there exists $b^+\in CF^{odd}(L;\Lambda_{+})$ such that $\fm_0^{b^+}(1)=W^{\wT}(b)\wunit$, i.e., $(CF^{\bullet}(L;\Lambda_0),\fm)$ is weakly unobstructed. In particular, if the minimal Maslov index of $L$ is at least two, then, $W^{\wT}(b)=W(b)$.
\end{lemma}



\subsection{Equivariant Lagrangian Floer theory}\label{sec:G-Morse}
Equivariant Lagrangian Floer theory has been substantially developed in the recent years \cite{SS10,HLS16a,HLS16b,BH18,DF17}. In \cite{KLZ19}, a Morse model of equivariant Lagrangian Floer theory was constructed by counting pearly trees in the Borel construction. It was inspired by the family Morse theory of Hutchings \cite{hutchings08} and the pearl complex of Biran-Cornea \cite{BC-pearl,BC-survey}. It also incorporated the works of Fukaya-Oh-Ohta-Ono \cite{FOOO,FOOO-can}, allowing the theory to work in a very general setting. We give a brief overview of this equivariant Morse model in this section for the purpose of applications in this paper.

\subsubsection{Equivariant Floer complexes}\label{subsubsec:equifcpx1}
Let $(X,\omega)$ be a tame symplectic manifold equipped with a symplectic action of a compact Lie group $G$. Let $L\subset X$ be closed, connected, relatively spin, and $G$-invariant immersed Lagrangian submanifold with $G$-invariant clean self-intersection. We begin by choosing smooth finite dimensional approximations for the universal bundle $EG\to BG$ over classifying space. Namely, we have a commutative diagram
\begin{equation}
\label{eq:BG(N)}
\begin{tikzcd}
EG(0) = G \arrow[hookrightarrow]{r} \arrow{d}   & EG(1) \arrow[hookrightarrow]{r} \arrow{d} & EG(2) \arrow[hookrightarrow]{r} \arrow{d} & \cdots \\
BG(0) = \pt \arrow[hookrightarrow]{r}    & BG(1) \arrow[hookrightarrow]{r}  & BG(2) \arrow[hookrightarrow]{r}  & \cdots 
\end{tikzcd}
\end{equation}
where $EG(N)$ and $BG(N)$ are compact smooth manifolds, and the horizontal arrows are smooth embeddings.

Let $\mu_{N}:T^*EG(N)\to\mathfrak{g}^*$ be the moment map for the Hamiltonian $G$-action on $T^*EG(N)$ lifted from that on $EG(N)$. Since $G$ acts on 
$T^*EG(N)$ freely, we have
$$
T^*EG(N)\twobar G:=\mu_N^{-1}(0)/G\cong T^*BG(N),
$$
canonically, as symplectic manifolds. Let us set $L(N)=L\times_G EG(N)$, $\tilde{L}(N)=\tilde{L}\times_G EG(N)$ and  $X(N)=X\times_G \mu_N^{-1}(0)$. Then $\eqref{eq:BG(N)}$ induces a commutative diagram:

\begin{equation}
\label{eq:inclusions}
\begin{tikzcd}[row sep=huge]
\tilde{L}(0)=\tilde{L} \arrow[hookrightarrow]{r} \arrow{d}   & \tilde{L}(1) \arrow[hookrightarrow]{r} \arrow{d} & \tilde{L}(2) \arrow[hookrightarrow]{r} \arrow{d} & \cdots \\
X(0)=X \arrow[hookrightarrow]{r}    & X(1) \arrow[hookrightarrow]{r}  & X(2) \arrow[hookrightarrow]{r}  & \cdots 
\end{tikzcd}
\end{equation}
where the vertical arrows are Lagrangian immersions with $\iota(\tilde{L}(N))=L(N)$. 
We note that $X(N)$ and $\tilde{L}(N)$ are fiber bundles over $T^*BG(N)$ and $BG(N)$ with fibers $X$ and $\tilde{L}$, respectively. Since the zero section $BG(N)$ is an exact Lagrangian submanifold of $T^*BG(N)$, we have an identification of the effective disc classes 
\[
H_2^{\mathrm{eff}}(X,L)= H_2^{\mathrm{eff}}(X(N),L(N)).
\]

For simplicity of notations, we will denote $\bL(N)=\tilde{L}(N)\times_{L(N)} \tilde{L}(N)$ and $\bL=\bL(0)$. There is also a sequence of fiber products:
\[
\bL\into \bL(1)\into \bL(2)\into \cdots.
\]

In order to construct an equivariant Morse model, we choose Morse-Smale pairs $(f_{N},\mathscr{V}_N)$ on $\bL(N)$ for $N\ge 0$ satisfying following compatibility conditions:

\begin{defn} [{\cite[Definition~3.4]{KLZ19}}]	
\label{def:morse-smale}
We call a sequence of Morse-Smale pairs $\{(f_N,\mathscr{V}_N)\}_{N\in\N}$ \emph{admissible} if it satisfies the following:	\begin{enumerate}[label=\textnormal{(\arabic*)}]
		\item  	\label{assum:1}
		For each $N\in \N$, there is an inclusion of critical point sets $\Crit(f_{N})\subset \Crit(f_{N+1})$. Under this identification, we have 
		\begin{enumerate}[label=(\roman*)]
			\item For $p\in \Crit(f_{N})$,  $W^u(f_{N+1};p)\times_{\bL(N+1)}\bL(N)=W^u(f_{N};p)$. 
			\item For $p\in \Crit(f_{N})$, the image of $W^s(f_{N};p)$ in $\bL(N)$  coincides with $W^s(f_{N+1};p)$ in $\bL(N+1)$.
			\item For $q\in \bL(N+1)\setminus \bL(N)$, we have $\Phi_t(q)\notin \bL(N)$ for all $t\ge 0$. This implies 
			\[
			W^u(f_{N+1};p)\times_{\bL(N+1)}\bL(N)=\emptyset
			\]
			for $p\in\Crit(f_{N+1})\setminus\Crit(f_{N})$.
		\end{enumerate}	
		
		\item \label{assum:2} 
		For each $\ell\ge 0$, there exists an integer $N(\ell)>0$ such that $|p|>\ell$ for all $N\ge N(\ell)$ and $p\in\Crit(f_N)\setminus\Crit(f_{N-1})$.	
		
		\item 
		The restriction of $f_N$ to the diagonal component $\tilde{L}(N)$ has a unique maixmum point $\bunit_{\tilde{L}(N)}$ and the inclusion $\Crit(f_{N})\subset \Crit(f_{N+1})$ identifies $\bunit_{\tilde{L}(N)}$ with $\bunit_{\tilde{L}(N+1)}$. This allows us to identify $CF^{\bullet}(L(N);\Lambda_0)$ with a subcomplex of $CF^{\bullet}(L(N+1);\Lambda_0)$. We will denote $\bunit_{\tilde{L}(N)}$, $\wunit_{\tilde{L}(N)}$ and $\gunit_{\tilde{L}(N)}$ simply by $\bunit$, $\wunit$, and $\gunit$. 			
	\end{enumerate}
\end{defn}
(See \cite[Proposition~3.5]{KLZ19} for a construction of an admissible collection.)

Now, suppose we are given an admissible collection $\{(f_{N},\mathscr{V}_N)\}_{N\in\N}$, our equivariant Morse model $(CF^{\bullet}_G(L;\Lambda_0),\{\fm^G_k\}_{k\ge 0})$ can be defined as follows. 
Let $(CF^{\bullet}(L(N);\Lambda_0),\fm^N)$ be the unital Morse model associated to $(f_{N},\mathscr{V}_N)$ (as in Section \ref{sec:G-Morse}), and let $\fm^N_{\Gamma}$ denote the operations associated to stable trees in the definition of $\fm^N$. We define the \textit{equivariant Floer complex} $CF^{\bullet}_G(L;\Lambda_0)$ by 
\[
CF^{\bullet}_G(L;\Lambda_0)=\lim_{\to} CF^{\bullet}(L(N);\Lambda_0),
\]
where the arrows are inclusions of graded submodules.

The perturbations for the moduli spaces in the construction of $(CF^{\bullet}(L(N);\Lambda_0),\fm^N)$ can be chosen such that for fixed inputs $p_1,\ldots,p_k \in CF^{\bullet}_G(L;\Lambda_0)$ and a stable tree $\Gamma\in\bm{\Gamma}_{k+1}$, $\fm^N_{\Gamma}(p_1,\ldots,p_k)$ is independent of $N$ for sufficiently large $N$, and $N$ depends on the degrees of the inputs and the Maslov indices of the decorations (see \cite[Proposition 3.6]{KLZ19}). 
We can therefore define an $A_{\infty}$-algebra structure $\fm^G$ on $CF^{\bullet}_G(L;\Lambda_0)$ by
\begin{equation}
\fm^G_{k}=\sum_{\Gamma\in \bm{\Gamma}_{k+1}} \bm{T}^{\omega(\Gamma)} \fm^G_{\Gamma},
\end{equation}
where
\begin{equation}
\fm^G_{\Gamma}(p_1,\ldots,p_k)=\fm^N_{\Gamma}(p_1,\ldots,p_k),
\end{equation}
for sufficiently large $N$ (depending on the degrees of the inputs and the Maslov indices).

We call the resulting $A_{\infty}$-algebra $(CF^{\bullet}_G(L;\Lambda_0), \fm^G)$ the Morse Model for $G$-equivaraiant Lagrangian Floer theory ($G$-equivariant Morse model) of $L$.\\

Let $(L_1,L_2)$ be a pair of closed, connected, relatively spin, and embedded $G$-invariant Lagrangian submanifolds intersecting cleanly such that their intersection is also $G$-invariant. Then, $L=L_1\cup L_2$ is a immersed Lagrangian with $G$-invariant clean self-intersection and $\tilde{L}=L_1\coprod L_2$. We choose the splitting $\iota^{-1}(\cI)=\cI^-\coprod \cI^+$ so that $\cI^-\subset L_1$ and $\cI^+\subset L_2$.

Similar to the non-equivariant setting, we can define an equivariant pearl complex $(CF^{\bullet}_G(L_1,L_2;\Lambda_0),\fm_1^{G})$ as follows: 
\[
CF^{\bullet}_G(L_1,L_2;\Lambda_0)\subset CF^{\bullet}_G(L;\Lambda_0)
\]
is the subcomplex generated by critical points on $(R_1)_G= R_1 \times_G EG$ 
The differential $\fm_1^{G}$ counts stable pearly trees in $\bm{\Gamma}_2$ with both input and output vertices in $(R_1)_G$.

\subsubsection{Equivariant parameters as partial units} \label{sec:partial_units} 
The infinite complex Grassmannian $B(U(k))=Gr(k,\C^{\infty})$ can be embedded into skew-Hermitian matrices on $\C^{\infty}$ by identifying $V\in Gr(k,\C^{\infty})$ with the orthogonal projection $P_V$ onto $V$. Let $A$ be the diagonal matrix with entries $\{1,2,3,\ldots\}$. The map 
\begin{equation*} \label{eq:perfect_Gr}
f_{Gr(k,\C^{\infty})}:Gr(k,\C^{\infty})\to \R, \quad f_{Gr(k,\C^{\infty})}(V)= -\re (\tr(AP_V)),
\end{equation*}
is a perfect Morse function whose restriction to each finite dimensional stratum $Gr(k,\C^{k+N})$ is again a perfect Morse function.

When $G$ is a product of unitary groups, this allows us to choose an admissible  Morse-Smale collection $\{(f_N,\mathscr{V}_N)\}_{N\in\N}$ with the Morse functions $f_N$ of the form 
\[
f_N=\pi_N^*\varphi_N+\phi_N,
\] 
where $\pi_N:\bL(N)\to BG(N)$ is the projection map, $\varphi_N$ is a perfect Morse function on $BG(N)$, and $\phi_N$ is a (generically) fiberwise Morse function over $BG(N)$ such that $\phi_N$ restricted to the fiber $L=\pi_N^{-1}(p)$ over each critical point $ p\in\Crit(\varphi_N)$ is a Morse function $f=\phi_0$ on $\bL$. 

For such a choice of an admissible collection, an $A_{\infty}$-algebra $(CF^{\bullet}_G(L;\Lambda_0)^{\dagger},\fm^{G,\dagger})$ homotopy equivalent to $(CF^{\bullet}_G(L;\Lambda_0), \fm^G)$ was constructed in \cite[Section~3.2]{KLZ19}, with the following properties. We have
\[
CF^{\bullet}_G(L;\Lambda_0)^{\dagger}=CF^{\bullet}(L;\Lambda_0)\otimes_{\Lambda_0} H^{\bullet}_G(\pt;\Lambda_0)
\]
where 
$CF^{\bullet}(L;\Lambda_0)=C^{\bullet}(f;\Lambda_0)\oplus \Lambda_0 \cdot \wunit_L\oplus \Lambda_0 \cdot \gunit_L$,
and $H^{\bullet}_G(\pt;\Lambda_0)=H^{\bullet}(BG;\Lambda_0)$ is a polynomial ring generated by even degree elements. 

For any $\lambda\in H^{\bullet}_G(\pt;\Lambda_0)$, we put 
\[
\blambda=\bunit_L\otimes \lambda, \quad \wlambda=\wunit_L\otimes \lambda, \quad \glambda=\gunit_L\otimes \lambda.
\]
In particular, for $\lambda=1$, we denote
\[
\bunit=\bunit_L\otimes 1, \quad \wunit=\wunit_L\otimes 1, \quad 
\gunit=\gunit_L\otimes 1.
\]

The $A_{\infty}$-structure $\fm^{G,\dagger}$ has the following properties:
\begin{itemize}
\item The restriction of $\fm^{G,\dagger}$ to $CF^{\bullet}_G(L;\Lambda_0)$ agrees with $\fm^G$. 

\item $\wunit$ is the strict unit.	
	
\item The elements $\wlambda$ are \textit{partial units}, namely, they satisfy
\[
\fm^{G,\dagger}_2(\wlambda,x\otimes y)=x\otimes \lambda\cup y =(-1)^{\deg x\otimes y} x\otimes y\cup \lambda = (-1)^{\deg x\otimes y} \fm^{G,\dagger}_2(x\otimes y,\wlambda),
\]
for $x\otimes y\in CF^{\bullet}_G(L;\Lambda_0)^{\dagger}$, where $\cup$ denotes the cup product on $H^{\bullet}_G(\pt)$, and 
\[
\fm^{G,\dagger}_k(\ldots,\wlambda,\ldots)=0,
\]
for $k\ne 2$. 

\item Assuming the minimal Maslov index of $L$ is nonnegative, then  
\[
\fm_1^{G,\dagger}(\glambda)=\wlambda-(1-h)\blambda, \quad h\in \Lambda_+.
\]	
In particular, 
\[
\fm_1^{G,\dagger}(\gunit)=\wunit-(1-h)\bunit.
\]
\end{itemize}	

Let us denote $\fm^{G,\dagger}_2(\wlambda,x\otimes y)$ by $\lambda\cdot x\otimes y$. It follows from the $A_{\infty}$-relations that 

\begin{theorem} [{\cite[Theorem 3.7]{KLZ19}}] \label{thm:pull-out-lambda}
Assume that $L$ has non-negative minimal Maslov index. Let $X_1,\ldots,X_{k}\in CF^{\bullet}_G(L;\Lambda_0)^{\dagger}$. We write $X_{\ell}=\sum a_{ij} x_i\otimes \lambda_j$ for $\ell\in\{1,\ldots,k\}$. Then, we have
\begin{equation*}
\mathfrak{m}^{G,\dagger}_k(X_1,\ldots,X_k)=(-1)^{\ell}\sum 
a_{ij} \lambda_j\cdot \mathfrak{m}^{G,\dagger}_k(X_1,\ldots,X_{\ell-1},x_i\otimes 1,X_{\ell+1},\ldots,X_k).
\end{equation*}	
\end{theorem}

This shows that $(CF^{\bullet}_G(L;\Lambda_0)^{\dagger}, \mathfrak{m}^{G,\dagger})$ can be defined over the graded coefficient ring $H^{\bullet}_G(\mathrm{pt};\Lambda_0)$, namely
\begin{equation}
(CF^{\bullet}_G(L;\Lambda_0)^{\dagger}, \mathfrak{m}^{G,\dagger})=(CF^{\bullet}(L;H^{\bullet}_G(\mathrm{pt};\Lambda_0)),\fm^{G,\dagger}).
\end{equation}

\subsubsection{Equivariant Maurer-Cartan space}
For an element $b\in CF^{\bullet}_G(L;\Lambda_+)^{\dagger}$, we consider the \emph{equivariant weak Maurer-Cartan equation} defined by
\begin{equation}
\label{eq:equiv_MC}
\mathfrak{m}^{G,\dagger,b}_0(1)=\mathfrak{m}^{G,\dagger}_0(1)+\mathfrak{m}^{G,\dagger}_1(b)+\mathfrak{m}^{G,\dagger}_2(b,b)+\ldots  \in H^{\bullet}_G(\mathrm{pt};\Lambda_0) \cdot \wunit.
\end{equation}
We denote by $\mathcal{MC}_G(L)$ the space of odd degree solutions of \eqref{eq:equiv_MC}. An element $b\in \mathcal{MC}_G(L)$ is called a \emph{weak Maurer-Cartan element over $H^{\bullet}_G(\mathrm{pt};\Lambda_0)$}. We say that $(CF^{\bullet}_G(L;\Lambda_0)^{\dagger}, \mathfrak{m}^{G,\dagger})$ is \emph{weakly unobstructed} if $\mathcal{MC}_G(L)$ is nonempty.



\begin{defn} \label{def:G_unobstr}
	The deformation of $(L,G)$ by $b$ (or simply $(L,G,b)$) is called \emph{unobstructed} (resp. \emph{weakly unobstructed}) if $\mathfrak{m}^{G,\dagger,b}(1)=0$ (resp. $b\in \mathcal{MC}_G(L)$). In particular, if $b=0$, we simply call $(L,G)$ \emph{unobstructed}.   
\end{defn}

Similar to Lemma \ref{lem:unobstr}, we have

\begin{lemma} [{\cite[Lemma~3.9]{KLZ19}}]
	\label{lem:G_unobstr}
	
	For $b\in CF^{1}_G(L;\Lambda_{+})$, suppose that
	\[
	\mathfrak{m}^{G,\dagger,b}_0(1)=W(b) \bunit+\sum_{\deg \lambda =2} \phi_{\lambda}(b)\blambda
	\]
	and the minimal Maslov index of L is nonnegative. 
	Then there exists $b^{\dagger}\in CF^{odd}_G(L;\Lambda_{+})^{\dagger}$ such that 
	\[
	\mathfrak{m}^{G,\dagger,b^{\dagger}}_0(1)=W^{\wT}(b)\wunit+\sum_{\deg \lambda =2} \phi_{\lambda}^{\wT}(b) \wlambda.
	\] 
	In particular, $(CF^{\bullet}_G(L;\Lambda_0)^{\dagger}, \mathfrak{m}^{G,\dagger})$ is weakly unobstructed. 
	
	If the minimal Maslov index of $L$ is assumed to be at least two in addition, then $W^{\wT}=W(b)$ and $\phi_{\lambda}^{\wT}(b)=\phi_{\lambda}(b)$.
\end{lemma}

\begin{corollary} [{\cite[Corollary~3.10]{KLZ19}}] \label{cor:unobstructed}
In the setting of Lemma \ref{lem:G_unobstr}, if $b\in CF^{1}(L,\Lambda_+)$ and $(L,b)$ is weakly unobstructed, then $(L,G,b)$ is weakly unobstructed. 
\end{corollary}

We note that even if the minimal Maslov index of $L$ is $2$ and $(L,b)$ is (strictly) unobstructed, one can only expect $(L,G,b)$ to be weakly unobstructed in general. This is due to the possibility of the constant disc class   contributing to degree $2$ equivariant parameters.

	
\section{Equivariant disc potentials of immersed SYZ fibers} \label{sec:immersed_SYZ}

In this section, we compute the equivariant disc potentials of immersed SYZ fibers in a toric Calabi-Yau manifold, which are homeomorphic to the product of an immersed 2-sphere and a complementary dimensional torus. We begin by deriving a gluing formula between Maurer-Cartan spaces of smooth fibers and immersed fibers, making use of (quasi-)isomorphisms in Lagrangian Floer theory. The disc potential of the immersed fiber can then be obtained by applying the gluing formula to that of a smooth fiber, which is much easier to compute.

\subsection{Gluing formulas}
 

Let $X=X_{\Sigma}$ be the toric Calabi-Yau manifold defined by a fan $\Sigma$ as in Section~\ref{sec:toric_CY}. Let $a_1 \in \textbf{M}_\R/\R\cdot\unu$ be an interior point of a codimension $2$ face $F$ of the polytope dual to $\Sigma$. 
For each $i=0,1,2$, we have the Lagrangian submanifold $L_i\subset X$ given by the inverse image 
of the circle $\ell_i:=w(L_i)$ in the $w$-plane $X\sslash_{a_1} T^{n-1}\cong \C$. Here $w$ is the toric holomorphic function on $X$ corresponding to $\unu$. The circles $\ell_0,\ell_1$ and $\ell_2$ are depicted in Figure~\ref{fig:posvertex1}. 
By construction, $L_1$ and $L_2$ are Lagrangian tori, while $L_0$ is an immersed Lagrangian homeomorphic to $\cS^2\times T^{n-2}$, where $\cS^2$ denotes the immersed two-sphere with exactly one nodal self-intersection point. 

Let $X_{\sigma}\cong \C^n$ be the toric affine chart of $X$ with coordinates $(y_1,\ldots,y_n)$ dual to the basis $\{v_1,\ldots,v_n\}$. We have $w |_{X_{\sigma}} =y_1\ldots y_n$. Note that the Lagrangians $L_i$ for $i=0,1,2$ are contained inside $X_{\sigma}$.
Each base circle $\ell_i$ encloses the point $\epsilon (\ne 0)$ and has winding number $1$ around $\epsilon$ in the $w$-plane. Thus, $L_0$, $L_1$ and $L_2$ are all graded with respect to the holomorphic volume form
\[
\Omega=\frac{dy_1\wedge\ldots\wedge dy_n}{w-\epsilon}
\]
on $X^\circ = X - \{w=\epsilon\}$. Throughout, we study Floer theory of these Lagrangians, regarding them as Lagrangian submanifolds of $X^{\circ}$. 

\subsubsection{Choice of flat $\Lambda_{\mathrm{U}}$-connections}
We will decorate the Lagrangians $L_0$, $L_1$, and $L_2$ with trivial line bundles with flat $\Lambda_\mathrm{U}$-connections. In what follows, we fix parameters and gauge cycles for the flat $\Lambda_\mathrm{U}$-connections, which enable us to compute the gluing formulas and the disc potentials explicitly later on. 

We may assume that the face $F$ of codimension $2$ containing $a_1$ is dual to the $2$-cone $\R_{\geq 0}\cdot\{v_1,v_2\}$ without loss of generality. Then the vector $v_1':=v_2-v_1$ is perpendicular to $F$. We extend $v_1'$ to a basis $\{v_1',v_2',\ldots,v_{n-1}'\}$ of $\unu^\perp \subset \textbf{N}$. For instance, one can take $v_i' = v_{i+1}-v_1$ for $i=1,\ldots,n-1$. 

The restriction of $w: X^{\circ} \to \C \setminus \{w=\epsilon\}$ to $\{y_2 \dots y_n \ne 0 \}$ is a trivial $(\C^{\times})^{n-1}$-fibration with the base coordinate $w$ and  the fiber coordinates $y_2,\ldots,y_n$. The map 
\[
(y_1,\ldots,y_n)\mapsto \left( \cfrac{w-\epsilon}{|w-\epsilon|},\cfrac{y_2}{|y_2|},\ldots,\cfrac{y_{n}}{|y_n|} \right)
\]
trivializes the SYZ torus fibration in the chamber dual to $v_1$ (see \cite{CLL}). 
Also, the trivialization fixes identifications of $L_1$ and $L_2$ with the standard torus $T^n$ whose $\bS^{1}$-factors are in the directions of $v_1,v'_1,\ldots,v'_{n-1}$. We then have
\begin{equation}\label{eqn:trivpi1}
\begin{split}
&\pi_1(L_i) \cong \pi_1(T^n) \cong \textbf{N} = \Z\cdot \{v_1,v'_1,\ldots,v'_{n-1}\} \quad (i=1,2), \\
&\pi_1(L_0) \cong \Z \cdot \{v_1,v'_2,\ldots,v'_{n-1}\}.
\end{split}
\end{equation}

We parametrize the space $\hom(\pi_1(L_1),\Lambda_\mathrm{U})$ of flat connections (up to gauge) by $(z_1,\ldots,z_n)$, which is aligned in parallel with the order of basis elements in \eqref{eqn:trivpi1}. Namely, $z_j \in \Lambda_\mathrm{U}$ is the holonomy of the connection along the $j$-th loop in \eqref{eqn:trivpi1}. 
The flat connection associated to $(z_1,\ldots,z_n)$ will be denoted by $\nabla^{(z_1,\ldots,z_n)}$.
Similarly, one can associate the flat connection $\nabla^{(z'_1,\ldots,z'_n)}$ for $L_2$ to $(z'_1,\ldots,z'_n) \in (\Lambda_\mathrm{U})^n$. 
The holonomies in the monodromy invariant directions of $L_0$ are parametrized by  $(z^{(0)}_2,\ldots,z^{(0)}_{n-1})$ and denoted by $\nabla^{(z^{(0)}_2,\ldots,z^{(0)}_{n-1})}$. 

We then fix the gauge cycles for the flat connections, which are codimension-$1$ submanifolds of a Lagrangian dual to loops in \eqref{eqn:trivpi1}. This enables us to compute parallel transports for the above connections conveniently (see \cite{CHL-toric} for more details).
 For the Lagrangian tori $L_1$ and $L_2$, it suffices to fix a union of cycles of codimension one in $T^n$. Let us fix a perfect Morse function $f^{\bS^1}$ on $\bS^1$. We choose a perfect Morse function $f^{L_i}$ on $L_i$, $i=1,2$ in such a way that under the identification $L_i\cong T^n$, $f^{L_i}$ is the sum of perfect Morse function $f^{\bS^1}$ on the $\bS^1$-factors in the directions of $v'_1,\ldots,v'_{n-1}$. We also fix a perfect Morse function on the $\bS^1$-factor of $L_1$ and $L_2$ in the $v_1$-direction with positions of the critical points depicted in Figure \ref{fig:posvertex1} (the minimum points of $\ell_1$ and $\ell_2$ are denoted by $z_3$ and $z'_3$, respectively and the maximum points are denoted by $\one^{\blT}$ in respective colors.). The unstable chains of the degree one critical points of $f^{L_i}$ are hypertori dual to $v_1,v_1',\ldots,v_{n-1}'\in \pi_1(T^n)$, co-oriented by the $\bS^1$-orbits in the respective directions.  We then choose these hypertori dual to $v_1',\ldots,v_{n-1}'$ to be gauge cycles. That is, the flat connections $\nabla^{(z_1,\ldots,z_n)}$ (resp. $\nabla^{(z'_1,\ldots,z'_n)}$) are trivial away from the gauge hypertori, and have holonomy $z_j$ (resp. $z'_j$) along a path positively crossing the gauge hypertorus dual to $v'_j$ in $L_1$ (resp. $L_2$) once, for $j=1,\ldots,n-1$, and have holonomy $z_n$ (resp. $z'_n$) along a path positively crossing the gauge hypertorus dual to $v_1$.

\begin{figure}[h]
\begin{center}
\includegraphics[scale=0.45]{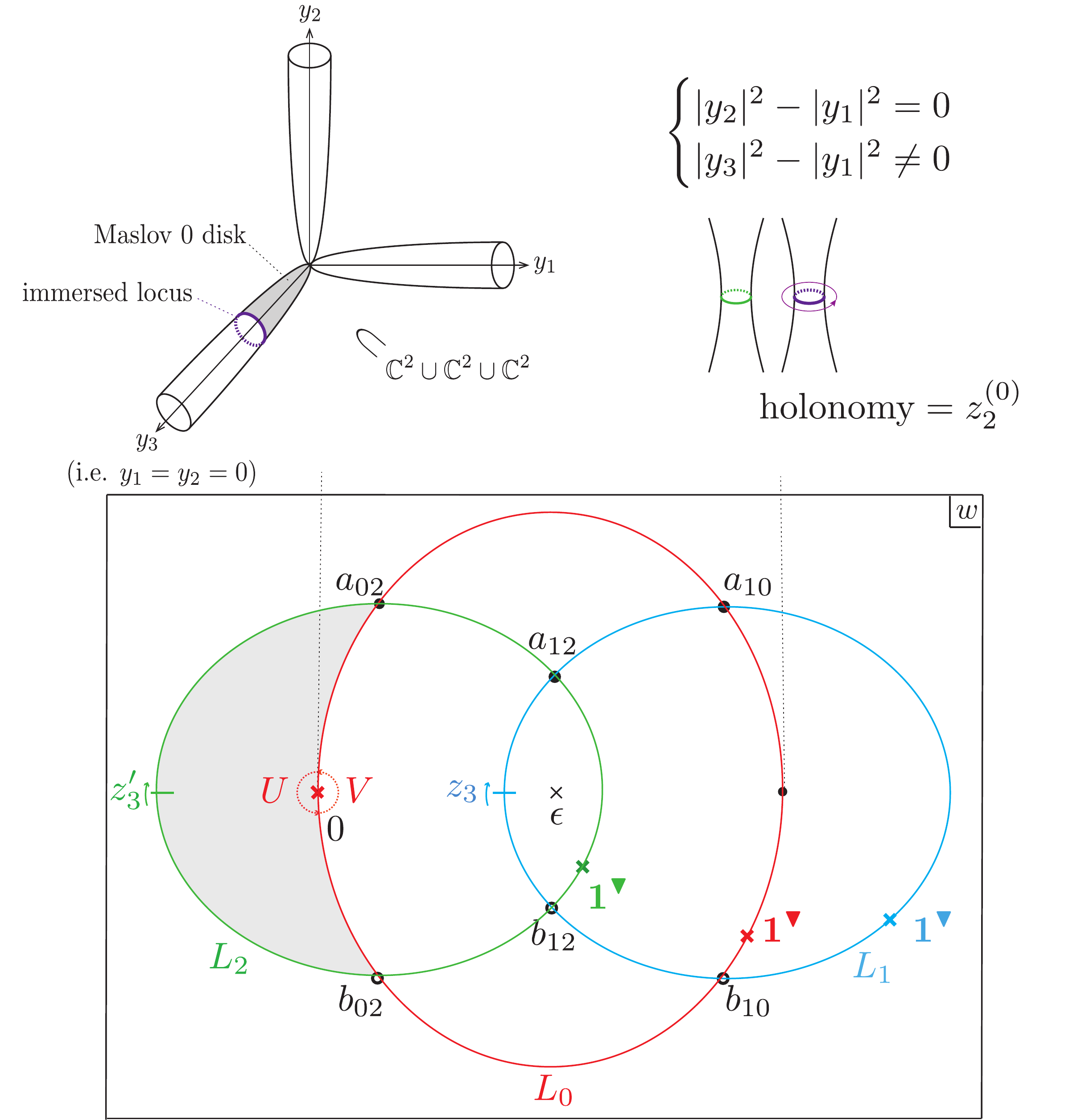}
\caption{$L_0$ in dimension 3. Note that it bounds a nontrivial holomorphic disc whose boundary class is associated with the holonomy $z_2^{(0)}$}\label{fig:posvertex1}
\end{center}
\end{figure}

For the immersed Lagrangian $L_0$, 
we consider a splitting $L_0 \cong \cS^2 \times T^{n-2}$, where the $T^{n-2}$-factor is in the directions of $v'_2,\ldots,v'_{n-1}$. Let $\iota:\widetilde{L}_0=\bS^2 \times T^{n-2}\to X$ be the Lagrangian immersion such that $\iota(\widetilde{L}_0) = L_0$. We denote 
\begin{itemize}
\item by $\{r\}\times T^{n-2}$ the clean self-intersection loci in $L_0$
\item by $\{p\}\times T^{n-2}$ and $\{q\} \times T^{n-2}$ the two disjoint connected components of the inverse image of $\{r\}\times T^{n-2}$ in $\widetilde{L}_0$
\end{itemize}
Then, the fiber product $\bL_0=\widetilde{L}_0\times_{L_0} \widetilde{L}_0$ consists of three components$\colon$ the diagonal component $R_0\cong \widetilde{L}_0$ and the two non-diagonal components $R_{-1}=\{(p,q)\}\times T^{n-2}$ and $R_1=\{(q,p)\}\times T^{n-2}$.

Let $f^{\bS^2}:\bS^2\to \R$ be a perfect Morse function such that the critical points of $f^{\bS^2}$ are away from $p$ and $q$, and the two flow lines connecting $p$ and $q$ to the minimum point are disjoint. 
Using the splitting above, we choose a perfect Morse function $f^{L_0}$ on $\bL_0$ of the form 
\begin{equation}\label{equ_choiceofmorse}
f^{L_0}|_{R_0}=f^{\bS^2}+f^{T^{n-2}} \mbox{ and } f^{L_0}|_{R_{\pm 1}}=f^{T^{n-2}}
\end{equation}
where $f^{T^{n-2}}$ is the sum of $f^{\bS^1}$ on the $\bS^1$-factors of $T^{n-2}$. 
We choose the gauge cycles for $\nabla^{(z^{(0)}_2,\ldots,z^{(0)}_{n-1})}$ to be the product of $\cS^2$ with hypertori in $T^{n-2}$ dual to $v'_2,\ldots,v'_{n-1}\in \pi_1(T^{n-2})$, 

\subsubsection{Unobstructedness of the Lagrangians (in $X^{\circ}$)}
	Let $(CF^{\bullet}(L_i),\fm^{L_i})$ be a unital 
	Morse model associated to the Morse function $f^{L_i}$ as defined in Section \ref{sec:pearl_complex}. We choose the following perturbations for the moduli spaces $\cM_{k+1}(\alpha,\beta,L_0)$ in order to simplify computations. 
	
	The Hamiltonian $T^{n-2}$-action corresponding to the sublattice generated by $\{v_2^\prime, \dots, v_{n-1}^\prime\}$ acts freely on the Lagrangian $L_i$ by rotating its $T^{n-2}$-factor for all $i = 0, 1, 2$. Moreover, the subtorus $T^{n-2}$ preserves the complex structure.
This induces a free $T^{n-2}$-action on $\cM_{k+1}(\alpha,\beta,L_0)$. We can therefore choose $T^{n-2}$-equivariant Kuranishi structures and perturbations for $\cM_{k+1}(\alpha,\beta,L_0)$ as in \cite{FOOO-T} (see also \cite{fukaya-equiv}). 
	
	Let $\tau:X\to X$ be the anti-symplectic involution characterized by that it reverses the regular fibers of the toric moment map (of the $T^n$-action). We note that the complex structure of $X$ is $\tau$-anti-invariant. 
	We will orient the moduli spaces with a $\tau$-invariant relative spin structure. 
	By Theorem \ref{thm:tau_*}, the $T^{n-2}$-equivariant Kuranishi structures on $\cM_{k+1}(\alpha,\beta;L_0)$ and $\cM_{k+1}(\hat{\alpha},\hat{\beta};L_0)$ can be chosen to be isomorphic. The $T^{n-2}$-equivariant perturbations $\fs$ for $\cM_{k+1}(\alpha,\beta;L_0)$ and $\hat{\fs}$ for $\cM_{k+1}(\hat{\alpha},\hat{\beta};L_0)$ can then be chosen to satisfy $(\tau^{main}_*)^*\hat{\fs}=\fs$. 

$H_2^{\mathrm{eff}}(X^{\circ},L_0)$ is generated by the Maslov index $0$ disc classes. For any stable disc $u$ in a non-trivial class $\beta \in H_2^{\mathrm{eff}}(X^{\circ},L_0)$, we observe that $u (\partial D^2) \subset \{r\}\times T^{n-2}$. This observation leads to the following simple consequences for the moduli spaces:
If $\alpha:\{0,\ldots,k\}\to \{-1,0,1\}$ and $\beta\in H_2^{\mathrm{eff}}(X^{\circ},L_0)$, then
\begin{itemize}
\item If $\alpha(i)=0$ for all $i$, and $\beta\ne 0$, then $\cM_{k+1}(\alpha,\beta;L_0)$ has two connected components corresponding to the lifts of the boundary of the holomorphic discs in class $\beta$ to $\{p\}\times T^{n-2}$ and $\{q\}\times T^{n-2}$.
	
\item Suppose $i\in\{0,\ldots,k\}$ and $\alpha(i)\ne 0$, i.e., $z_i$ is a corner. Let $z_j$ be a corner adjacent to $z_i$, namely, $\alpha(j)\ne 0$, and $\alpha(k)=0$ for all boundary marked points $z_k$ between $z_i$ and $z_j$. We then have $\cM_{k+1}(\alpha,\beta;L_0)=\emptyset$ unless $\alpha(i)$ and $\alpha(j)$ have opposite signs. In particular, when $k=0$, we have  $\cM_{1}(\alpha,\beta;L_0)=\emptyset$ unless $\alpha(0)= 0$, i.e., the output marked point lies in $R_0=\widetilde{L}_0$.
\end{itemize}



We denote by $\one_{T^{n-2}}$ the maximum point of $f^{T^{n-2}}$,  by $X_1,\ldots,X_{n-2}$ the degree $1$ critical points, and by $X_{ij}$, $1\le i<j\le n-2$ the degree $2$ critical points. We also denote the maximum and minimum points of $f^{\bS^2}$ by $\one_{\bS^2}$ and $a_{\bS^2}$, respectively. 

We will abuse notations and denote the critical points of $f^{L_0}$ in the form of tensor products. In particular, we use $U=(p,q)\otimes \one_{T^{n-2}}$ and $V=(q,p)\otimes \one_{T^{n-2}}$ to denote the maximum points on $R_{-1}$ and $R_1$ respectively. The holomorphic volume form $\Omega_X$ equips $CF^{\bullet}(L_0)$ with a $\Z$-grading, under which the two critical points $U$ and $V$ are of degree $1$, while the critical points $(p,q)\otimes X_i$ and $(q,p)\otimes X_i$ are of degree $2$.  Let $b=uU+vV$ for $u,v\in \Lambda^2_0$ with $\mathrm{val}(u\cdot v)>0$.

We define $\bbL_0$, $\bbL'_1$, $\bbL'_2$ to be the family of formal Lagrangian deformations (or the corresponding objects of the Fukaya category)
$$(L_0, \nabla^{(z^{(0)}_2,\ldots,z^{(0)}_{n-1})},b), \quad (L_1,\nabla^{(z_1,\ldots,z_n)}) \quad \mbox{and} \quad (L_2,\nabla^{(z'_1,\ldots,z'_n)}),$$ 
respectively. (The notations $\bbL_1$ and $\bbL_2$ is reserved for another Lagrangian brane which will appear later in the section.) We denote by  and $(CF^{\bullet}(\bbL'_i),\fm^{\bbL'_i})$ for $i=1,2$, and $(CF^{\bullet}(\bbL_0),\fm^{\bbL_0})$ the formally deformed $A_\infty$-algebras on the pearl complexes $CF^{\bullet}(L_i)$, $i=0,1,2$. More precisely, the $A_{\infty}$-operations $\fm^{\bbL'_i}_k$ are defined by
\[
\begin{split}
&\fm^{\bbL'_1}_k=\sum_{\Gamma\in\bm{\Gamma}_{k+1}} \bT^{\omega \left(\sum_v \beta_v \right)} \cdot \mathrm{Hol}_{\conn^{(z_1,\ldots,z_n)}}\left(\sum_v \partial\beta_v\right) \cdot \fm_{\Gamma}^{L_1}, \\
&\fm^{\bbL'_2}_k=\sum_{\Gamma\in\bm{\Gamma}_{k+1}} \bT^{\omega \left( \sum_v \beta_v \right)} \cdot \mathrm{Hol}_{\conn^{(z'_1,\ldots,z'_n)}}\left( \sum_v \partial\beta_v \right) \cdot \fm_{\Gamma}^{L_2}, \\
&\fm^{\bbL_0}_k =\sum_{\Gamma\in\bm{\Gamma}_{k+1}} \bT^{\omega(\sum_v \beta_v)} \cdot \mathrm{Hol}_{\conn^{\left(z^{(0)}_2,\ldots,z^{(0)}_{n-1}\right)}}\left( \sum_v \partial\beta_v \right) \cdot (\fm_{\Gamma}^{L_0})^b
\end{split}
\]
where the summation $\sum_v$ is over the interior vertices of the stable tree $\Gamma$, $\fm_{\Gamma}^{L_i}$ was defined in~\eqref{mgammaeq}, and $(\fm_{\Gamma}^{L_0})^b$ is the deformation of $\fm_{\Gamma}^{L_0}$ defined in~\eqref{eq:deformation}.  



\begin{figure}[h]
	\begin{center}
		\includegraphics[scale=0.7]{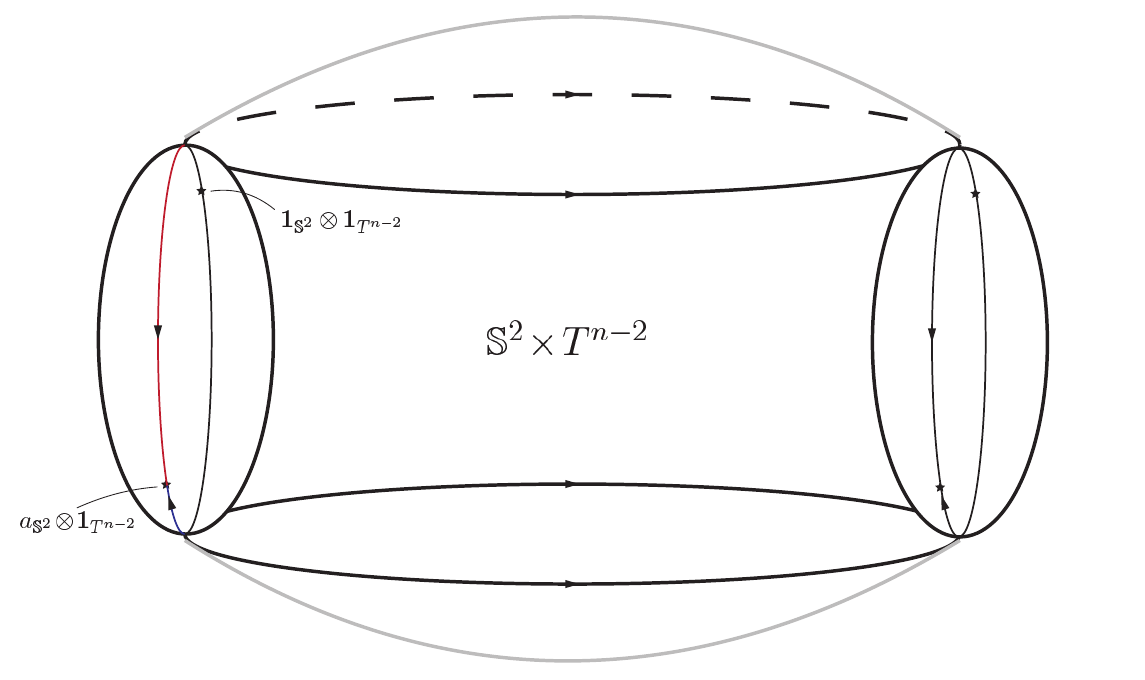}
		\caption{The domain $\bS^2 \times \bS^1$ of the Lagrangian immersion $L_0$.  The immersed loci are shown as top and bottom circles in the figure, which are also boundaries of the holomorphic discs of Maslov index zero.}\label{fig:imm-AV-discs}
	\end{center}
\end{figure}


\begin{lemma}  \label{lemma:L_0}
$\bbL_0$ is unobstructed.
\end{lemma}


\begin{proof}
Recall that $\fm_0^{\bbL_0}(1)$ is the summation over operations indexed by stable trees $\Gamma\in \bm{\Gamma}_{k+1}$, $k\ge 0$, with repeated inputs $b= uU+vV$. Since $U$ and $V$ are of degree $1$ and $\mu_{L_0}(\beta)=0$ for all $\beta \in H_2^{\mathrm{eff}}(X^{\circ},L_0)$, $\fm_0^{\bbL_0}(1)$ is a linear combination of degree $2$ critical points of $f^{L_0}$ of the form $(p,q)\otimes X_i$, $(q,p)\otimes X_i$, $a_{\bS^2}\otimes \bm{1}_{T^{n-2}}$, and $\one_{\bS^2}\otimes X_{ij}$. In the following, we show that the coefficients of these outputs all vanish.
\\

We begin by considering the coefficient of $(p,q) \otimes X_i$. 
Since the output of a stable tree with no input vertices must be in $R_0=\widetilde{L}_0$, it suffices to consider stable trees $\Gamma\in\bm{\Gamma}_{k+1}$ with $k\ge 1$.

We first exclude the contribution of  $\Gamma$ with at least one interior vertex. In this case, $\Gamma$ must have an interior vertex $v$ such that each input vertex of $v$ is an input vertex of $\Gamma$.
Note that this includes vertices $v$ with no input vertices (i.e. tadpoles). We denote the output vertex of $v$ by $w$. Let $(\alpha_v,\beta_v)$ be the stable polygon decorating $v$. Since the unstable chains of $U$ and $V$ are both $T^{n-2}$-invariant, by choosing a generic small $T^{n-2}$-equivariant perturbation for the fiber product at $v$, the output singular chain $\Delta_v$ is then $T^{n-2}$-invariant. 

If $\alpha_v(0)=\pm 1$, i.e., the output marked point is a corner, 
then the virtual dimension of $\Delta_v$ is $n-3$. Since $\Delta_v$ is $T^{n-2}$-invariant, it must be the zero chain. Now, suppose $\alpha_v(0)=0$, i.e., the output marked point is a smooth point. Then, $\Delta_v$ has virtual dimension $n-2$.
The evaluation image of $\Delta_v$ is a union of $T^{n-2}$-orbits $\coprod_{i=1,\ldots,\nu} \{p_i\}\times T^{n-2}\subset \widetilde{L}_0$, where $p_1,\ldots,p_\nu \in \bS^2\setminus \{p,q\}$ are points situated near either $p$ or $q$. (We remark that the evaluation image at smooth output point can be perturbed outside of $\{p\}\times T^{n-2}$ and $\{q\}\times T^{n-2}$ due to the fact that the Kuranishi structures are weakly submersive.) Note that if $w$ is the root vertex, then $\Gamma$ does not contribute to $p \otimes X_i$ since its stable submanifold is contained in $R_{-1}$. We may therefore assume $w$ is an interior vertex and denote its decoration by $(\alpha_{w},\beta_{w})$.

For $\beta_{w}\ne \beta _0$, the evaluation image of $\cM_{\ell(w)}(\alpha_{w},\beta_{w};L_0)$ at the corresponding input marked point is contained in $\{p,q\}\times T^{n-2}$, which does not intersect with the flow lines from $\Delta_v$ since $f^{\bS^2}$ was chosen so that the two flow lines connecting $p$ and $q$ to $a_{\bS^2}$ are disjoint. For $\beta_{w}=\beta_0$, we have the following three cases:
\begin{enumerate}
\item[(i)] $w$ has an input vertex $v'$  such that $v'$ is an input vertex of $\Gamma$.
\item[(ii)] $w$ has an input vertex $v'\ne v$ such that each input vertex of $v'$ is an input vertex of $\Gamma$.
\item[(iii)] There is a subtree $\Gamma'$ of $\Gamma$ ending at $w$ which does not contain $v$. $\Gamma'$ has an interior vertex $v^+$ such that each input vertex of $v^+$ is an input vertex of $\Gamma$ and its output vertex is not $w$. 
\end{enumerate} 
For the case (i), the fiber product at $w$ is empty since the flow lines from $U$ and $V$ are in the immersed sectors. For the case (ii), we again choose generic small $T^{n-2}$-equivariant perturbation for the fiber product at $v'$, and denote the output singular chain by $\Delta_{v'}$. The flow lines from $\Delta_v$ and $\Delta_{v'}$ do not intersect generically. Finally, for the case (iii), we repeat the arguments above for $\Gamma'$. It easy to see that such iterations terminate. This rules out the contribution of $\Gamma$ with at least one interior vertices. 

Consequently, the only possible contributions to $(p,q) \otimes X_i$ are from Morse flow lines, and we know that flow lines from $U$ to $(p,q) \otimes X_i$ cancel pairwise. Hence the coefficient of $(p,q) \otimes X_i$ vanishes. The vanishing of the coefficient of $(q,p) \otimes X_i$ follows from the exact same argument. \\

Next, we consider the coefficient of $a_{\bS^2} \otimes \one_{T^{n-2}}$.  
Since the Morse flow lines from $U$ and $V$ are contained in $R_{-1}$ and $R_1$, respectively,  $\Gamma^0$ does not contribute to $a_{\bS^2} \otimes \one_{T^{n-2}}$. 
It suffices to consider stable trees with at least one interior vertex. Moreover, by the same argument as in (1), the only possible contributions are from stable trees $\Gamma_{k+1}\in\bm{\Gamma}_{k+1}$, $k\ge 0$, with exactly one interior vertex $v$ (in which case its output vertex $w$ is the root vertex). If $v$ has an odd number of inputs, then its output marked point must be a corner, but $a_{\bS^2} \otimes \one_{T^{n-2}}\in \widetilde{L}_0$. Thus, we are left with the case of the stable trees $\Gamma_{2k+1}\in\bm{\Gamma}_{2k+1}$. We note that since the boundary of a stable polygon bounded by $L_0$ is contained in the the immersed loci, the input corners $U$ and $V$ must appear in pairs. 

For $k>1$, since $U$ and $V$ are of degree $1$, the map $\tau_*^{main}:  \cM_{2k+1}(\alpha_v,\beta_v;U,V,\ldots,U,V)\to \cM_{2k+1}(\hat{\alpha}_v,\hat{\beta}_v;V,U,\ldots,V,U)$ is an orientation reversing isomorphism of Kuranishi structures by Theorem \ref{thm:inv}. Let $\hat{\Gamma}_{2k+1}\in\bm{\Gamma}_{2k+1}$ be the tree obtained from $\Gamma_{2k+1}$ by changing the decoration at $v$ to $(\hat{\alpha}_v,\hat{\beta}_v)$. Then, by choosing $\tau_*^{main}$-invariant perturbations for $\cM_{2k+1}(\alpha_v,\beta_v;U,V,\ldots,U,V)$ and  $\cM_{2k+1}(\hat{\alpha}_v,\hat{\beta}_v;V,U,\ldots,V,U)$, the contribution of $\Gamma_{2k+1}$ and $\hat{\Gamma}_{2k+1}$ to $a_{\bS^2} \otimes \one_{T^{n-2}}$ cancel with each other. 

For $k=0$, we note that $\alpha_v:\{0\}\to \{-1,0,1\}$, $\alpha_v(0)=0$ and $\beta_v\ne 0$. Then $\cM_{1}(\alpha_v,\beta_v;L_0)$ has two connected components corresponding to the two possible lifts of the disc boundary to $\{p\}\times T^{n-2}$ and $\{q\}\times T^{n-2}$. By Theorem \ref{thm:inv},  $\tau_*^{main}:\cM_{1}(\alpha_v,\beta_v;L_0)\to \cM_{1}(\alpha_v,\beta_v;L_0)$ is an orientation reversing isomorphism swapping the two components. By choosing $\tau_*^{main}$-invariant perturbation for $\cM_{1}(\alpha_v,\beta_v;L_0)$, the contribution of $\Gamma_{2k+1}$ to $a_{\bS^2} \otimes \one_{T^{n-2}}$ vanishes.\\

Finally, the coefficient of $\one_{\bS^2}\otimes X_{ij}$ is zero for an obvious reason: there are no flow lines from $p$ or $q$ to $\one_{\bS^2}$
\end{proof}

The cancellation of contributions to $a_{\bS^2} \otimes \one_{T^{n-2}}$ can be intuitively visualized as in Figure~\ref{fig:afterpert} if we make a transverse perturbation of the self-intersection loci.

\begin{figure}[h]
	\begin{center}
		\includegraphics[scale=0.45]{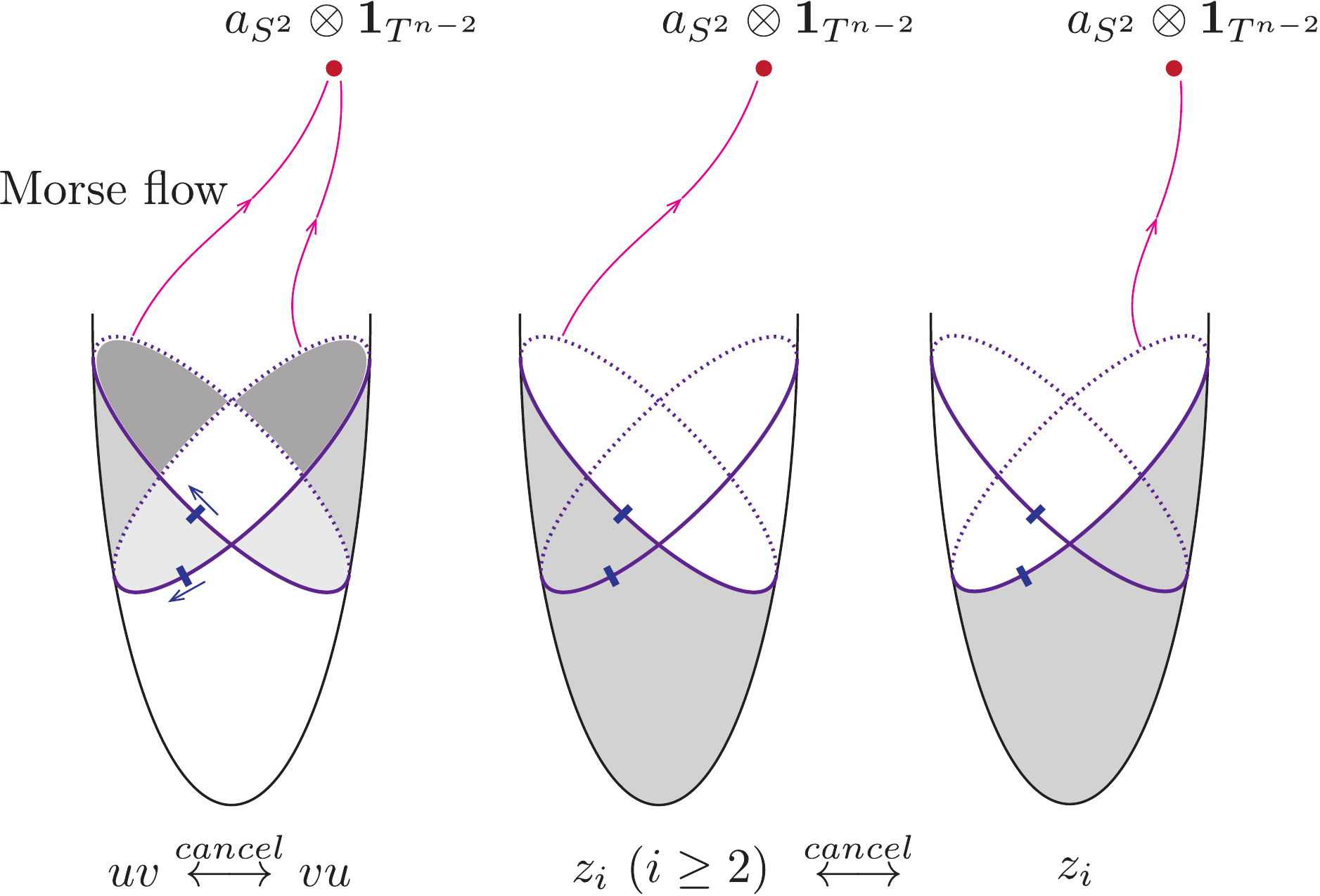}
		\caption{After perturbing the clean intersection loci, some of canceling pairs for $m_0^b$ are intuitively visible.}\label{fig:afterpert}
	\end{center}
\end{figure}

$\bbL'_1$  and $\bbL'_2$ are unobstructed as Lagrangian submanifolds of $X^{\circ}$, since the underlying Lagrangians $L_1$ and $L_2$ do not bound any non-constant holomorphic discs in $X^{\circ}$. Thus we have:

\begin{lemma}
$\bbL'_i$ is unobstructed for $i=1,2$.
\end{lemma}

\subsubsection{Computation of the gluing formula}
For $i<j\in \{0,1,2\}$, we denote the clean intersections of $L_i$ with $L_j$ by $\{a^{\ell_i,\ell_j},b^{\ell_i,\ell_j}\}\times T^{n-1}$, where $a^{\ell_i,\ell_j},b^{\ell_i,\ell_j}$ are the two intersection points of the base circles $\ell_i$ and $\ell_j$ on the $w$-plane as depicted in Figure \ref{fig:posvertex1}, and the $T^{n-1}$-factor lies along the directions of $v_1',\ldots,v_{n-1}'$. 
We choose a perfect Morse function $f^{L_i, L_j}$ on $\{a^{\ell_i,\ell_j},b^{\ell_i,\ell_j}\}\times T^{n-1}$ such that its restriction to each connected component is the sum of $f^{\bS^1}$ on the $\bS^1$-factors.

In \cite{CLL}, it was shown that the disc potentials of $L_1$ and $L_2$ are equal under the change of variables $z_i' = z_i$ for $i = 1,\ldots,n-1$, and
$z_n' = z_n \cdot f(z_1,\ldots,z_{n-1})$ where $f$ is the generating function given in \eqref{eqn:fslabftn}. In the following, we will deduce the same gluing formula by finding a (quasi-)isomorphism between $\bbL'_1$ and $\bbL'_2$ regarded as objects the Fukaya category.

Let $(CF^{\bullet}(\bbL'_1,\bbL'_2),\fm_1^{\bbL'_1,\bbL'_2})$ and $(CF^{\bullet}(\bbL'_2,\bbL'_1),\fm_1^{\bbL'_2,\bbL'_1})$ be the pearl complexes associated to the perfect Morse function $f^{L_1, L_2}$, which are further deformed by the flat connections $\nabla^{(z_1,\ldots,z_n)}$ and $\nabla^{(z'_1,\ldots,z'_n)}$. We denote by 
\begin{align*}
a_{12} \otimes \one_{T^{n-1}}\in CF^{0}(\bbL'_1,\bbL'_2), \quad a_{21} \otimes \one_{T^{n-1}}\in CF^{1}(\bbL'_2,\bbL'_1)\\
b_{12} \otimes \one_{T^{n-1}}\in CF^{1}(\bbL'_1,\bbL'_2),\quad b_{21} \otimes \one_{T^{n-1}}\in CF^{0}(\bbL'_2,\bbL'_1),
\end{align*}
the generators corresponding to the the maximum points on $\{a^{\ell_1,\ell_2}\}\times T^{n-1}$ and $\{b^{\ell_1,\ell_2}\}\times T^{n-1}$, respectively. We will prove that $a_{12}\otimes \one_{T^{n-1}}$ and $b_{21}\otimes \one_{T^{n-1}}$ provide the desired isomorphisms, namely 

\begin{equation} \label{eq:iso}
\begin{aligned}
\fm_{1}^{\bbL'_1,\bbL'_2}(a_{12} \otimes \one_{T^{n-1}})&= 0,\\ 
\fm_{1}^{\bbL'_2,\bbL'_1}(b_{21} \otimes \one_{T^{n-1}})&= 0,\\ 
\fm^{\bbL'_1, \bbL'_2, \bbL'_1}_2(b_{21} \otimes \one_{T^{n-1}},a_{12} \otimes \one_{T^{n-1}})&= \one_{L_1}^{\wT} + \fm^{\bbL'_1}_1 (\gamma_1),\\ 
\fm^{\bbL'_2, \bbL'_1, \bbL'_2}_2(a_{12} \otimes\one_{T^{n-1}},b_{21} \otimes \one_{T^{n-1}})&= \one_{L_2}^{\wT} + \fm^{\bbL'_2}_1 (\gamma_2),	
\end{aligned}
\end{equation}
for some 
$\gamma_1\in CF^{\bullet}(\bbL'_1)$, and $\gamma_2\in CF^{\bullet}(\bbL'_2)$. Notice that $\one_{L_1}^{\wT}$ and $\one_{L_2}^{\wT}$ in the above equations are the strict units of $CF^{\bullet}(\bbL'_1)$ and $CF^{\bullet}(\bbL'_2)$, respectively. 

\begin{remark}
The relation~\eqref{eq:iso} between Clifford type tori and Chekanov type tori was studied by Seidel in his lecture notes, see \cite{Seinote}. 
\end{remark}

We may assume $L_1$ and $L_2$ are chosen such that the holomorphic strip classes $\beta_0, \beta_1\in \pi_2(X^{\circ},L_1\cup L_2)$ (see Corollary \ref{cor:hol-strip-class}) have the same symplectic area. We then have the following.

\begin{theorem} \label{thm:iso12}
	$a_{12}\otimes \one_{T^{n-1}}$ is a quasi-isomorphism between $\bbL'_1$ and $\bbL'_2$ with an inverse $b_{21}\otimes \one_{T^{n-1}}$ if and only if
	\begin{equation} \label{eq:gluing}
	\begin{array}{l}
	z_i' = z_i, \quad  i = 1,\ldots,n-1\\
	z_n' = z_n \cdot f(z_1,\ldots,z_{n-1}),
	\end{array}
	\end{equation}
	where $f$ is given in \eqref{eqn:fslabftn}.
\end{theorem}

The proof requires a few preliminary steps, analyzing relevant moduli spaces of holomorphic strips.

Let us first focus on $\fm_{1}^{\bbL'_1,\bbL'_2}(a_{12} \otimes \one_{T^{n-1}})$. A priori, the output of $\fm_{1}^{\bbL'_1,\bbL'_2}(a_{12} \otimes \one_{T^{n-1}})$ is a linear combination of $b_{12}\otimes\one_{T^{n-1}}$ and $a_{12}\otimes X_i$, where $X_1,\ldots, X_{n-1}$ are the degree $1$ critical points of $f^{T^{n-1}}$.  Since there are no non-constant holomorphic polygon with both input and output corners in $\{a^{\ell_1,\ell_2}\} \times T^{n-1}$, the coefficient of $a_{12}\otimes X_{i}$ is given by the two Morse flow lines from $a_{12} \otimes \one_{T^{n-1}}$ in $\{a^{\ell_1,\ell_2}\} \times T^{n-1}$. 
These flow lines contribute $z_i-z_i'$ to the coefficient of $a_{12}\otimes X_{i}$. 

Let $\cM_2^{L_1,L_2}(\beta; \{a^{\ell_1,\ell_2}\} \times T^{n-1}, \{b^{\ell_1,\ell_2}\} \times T^{n-1})$ be the moduli space of stable holomorphic strips in class $\beta\in \pi_2(X^{\circ},L_1\cup L_2)$ with input corner in $\{a^{\ell_1,\ell_2}\}\times T^{n-1}$ and output corner in $\{b^{\ell_1,\ell_2}\}\times T^{n-1}$. The coefficient of $b_{12}\otimes\one_{T^{n-1}}$ is given by the fiber products 
\begin{equation}\label{eq:stable_strip}
\cM_2^{L_1,L_2}(\beta;\{a^{\ell_1,\ell_2}\}\times T^{n-1},\{b^{\ell_1,\ell_2}\}\times T^{n-1}) \times_{L_1\cap L_2} \{b_{12}\otimes\one_{T^{n-1}}\},
\end{equation}
where $\beta$ has Maslov index $1$. 

In terminology of Section \ref{sec:Morse model}, \eqref{eq:stable_strip} corresponds to the moduli spaces of Maslov index $2$ stable polygons with one input critical point $a_{12} \otimes \one_{T^{n-1}}$ (whose unstable chain is $\{a^{\ell_1,\ell_2}\}\times T^{n-1}$) and the output corner in $\{b^{\ell_1,\ell_2}\}\times T^{n-1}$. By degree reason, the output is a multiple of $b_{12}\otimes\one_{T^{n-1}}$. Since $b_{12}\otimes\one_{T^{n-1}}$ is the maximum point on $\{b^{\ell_1,\ell_2}\}\times T^{n-1}$, any contributing stable polygon must intersect it at the output corner.


To make this more explicit, we first consider the example $X=\C^n$.  Let $y_1,\ldots,y_n$ be the standard complex coordinates on $\C^n$.  Then, $w=y_1\ldots y_n. $
The moment map of the $T^{n-1}$ action is given by 
\[
\mu(y_1,\ldots,y_n)=(|y_2|^2-|y_1|^2, \ldots, |y_n|^2-|y_1|^2).
\]
For the moment, let us take $a_1=(0,\ldots,0)$ so that the Lagrangians $L_i$ for $i=1,2$ satisfy $|y_1|=\ldots=|y_n|$.  Moreover, we take $\ell_2$ to be $|w|=1$ and $\ell_1$ to be the image of $\R\cup \{\infty\}$ under the transformation $\frac{w-\bar{\alpha}}{1-
	\alpha w}$ for some $\alpha \in \mathbf{i}\cdot (-1,0)$ in the $w$-plane, where $\mathbf{i}:=\sqrt{-1}$.

\begin{lemma}
	Let $X=\C^n$ and $L_1,L_2$ be given as above and let $\beta\in \pi_2(X^{\circ},L_1\cup L_2)$ be an holomorphic strip class of Maslov index $1$. The moduli space $\cM_2^{L_1,L_2}(\beta;\{a^{\ell_1,\ell_2}\}\times T^{n-1},\{b^{\ell_1,\ell_2}\}\times T^{n-1})$ is regular and is isomorphic to $T^{n-1}$. 
\end{lemma}
\begin{proof}
In the punctured $w$-plane $\C\setminus \{\epsilon\}$, $\epsilon\in \mathbf{i} \cdot (-1,\frac{\alpha}{\mathbf{i}})$, there are two holomorphic strips $u_L$ (which contains $w=0$) and $u_R$ bounded by $\ell_1$ and $\ell_2$. The projection of a holomorphic strip $u$ in class $\beta$ to the $w$-plane must cover either $u_L$ or $u_R$. We denote the class of holomorphic strips covering $u_R$ by $\beta_0$. It is easy to see that the Lemma holds for $\beta_0$ by trivializing the $(\C^{\times})^{n-1}$-fibration $w:X^{\circ}\to \C\setminus\{\epsilon\}$ away from $\{w=0\}$. 
The strip classes covering $u_L$ can be identified with the Maslov index $1$ disc classes bounded by a regular toric fiber of $\C^n$, and are hence regular (see \cite{CO}). Moreover, the holomorphic strips $u$ must intersects a toric prime divisor $\{y_i=0\}\subset X^{\circ}$ once for exactly one $i\in\{1,\ldots,n\}$. We will denote the class of holomorphic strips intersecting $\{y_i=0\}$ by $\beta_i$.  

Without loss of generality, we may assume $i=1$. The image of $u$ is contained in the complement of $\{y_j=0\}$ for $j=2,\ldots,n$, which can be identified with $\C \times (\C^\times)^{n-1}$, equipped with coordinates $y_2,\ldots,y_n \in \C^\times$ and $w=y_1\ldots y_n \in \C$.  Then, $u$ can be written as a holomorphic map
\[
u(\zeta)=(w(\zeta),y_2(\zeta),\ldots,y_n(\zeta))
\]
from the upper-half disc $\{\zeta\in \C \mid |\zeta|\leq 1, \mathrm{Im} \zeta \geq 0\}$. We have $w(\zeta) = \frac{\zeta-\bar{\alpha}}{1-\alpha \zeta}$, up to a reparametrization of the domain, which satisfies the boundary conditions $|w|=1$ on the upper arc $|\zeta|=1$ and $\{\frac{w-\bar{\alpha}}{1-\alpha w} \mid w \in \R\cup\{\infty\}\}$ on the lower arc $[-1,1]$.
	
For the $y_j(\zeta)$-components of $u(\zeta)$, the boundary condition $|y_1|=\ldots=|y_n|$ implies $|w|=|y_j|^n$ for $j=2,\ldots,n$. We consider the holomorphic map $\tilde{w}(\zeta)= \frac{\zeta-\alpha}{1-\bar{\alpha} \zeta}$, which satisfies $|\tilde{w}|=|w|=1$ on the upper arc $|\zeta|=1$ and $|\tilde{w}|=|w|$ on the lower arc $\zeta\in [-1,1]$.  Thus, it also satisfies $|\tilde{w}|=|y_j|^n$.  Moreover, $\tilde{w}(\zeta)\ne 0$ on the upper half disc $\{\zeta\in \C \mid |\zeta|\leq 1, \mathrm{Im } \zeta \geq 0\}$.  By the maximum principle, $\tilde{w}=\rho y_j^n$ for some constant $\rho \in U(1)$.  This determines $y_j = \tilde{w}^{1/n}$ up to a rotation by $\rho$.  Thus, the moduli space is isomorphic to $T^{n-1}$ (on which $T^{n-1}$ acts freely).  
\end{proof}

In the proof above, we have chosen base loops $\ell_1$ and $\ell_2$, and a level $a_1$ for the simplicity of argument. It is easy to see that the lemma holds for other choices of a level $a_1$ and base loops $\ell_1$ and $\ell_2$ enclosing $\epsilon$ and intersecting transversely at two points. 



Now, let $X$ be any toric Calabi-Yau manifold.  The $T^n$-moment map is of the form 
\[
\left(\rho_1(|y_1|^2,\ldots,|y_n|^2),\ldots,\rho_n(|y_1|^2,\ldots,|y_n|^2) \right)
\]
for a toric chart $(y_1,\ldots,y_n)$, where $\rho$ is a Legendre transform determined by the toric K\"ahler metric \cite{Gui94}.  

\begin{lemma} \label{lem:hol-strip-X}
	For a toric Calabi-Yau manifold $X$ and $L_1, L_2$ given in the beginning of this section, let $\beta\in \pi_2(X^{\circ}, L_1\cup L_2)$ be a holomorphic strip class of Maslov index $1$. The moduli space $\cM_2^{L_1,L_2}(\beta;\{a^{\ell_1,\ell_2}\}\times T^{n-1},\{b^{\ell_1,\ell_2}\}\times T^{n-1})$ is regular and isomorphic to $T^{n-1}$.
\end{lemma}
\begin{proof}
	Under the projection to the punctured $w$-plane, a holomorphic strip $u$ in class $\beta$ must cover one of the two strips $u_L$ (which contains $w=0$) and $u_R$ bounded by the two base circles $\ell_1$ and $\ell_2$. Since it has Maslov index $1$, it intersects at most one of the toric prime divisors in $\{w=0\}\subset X^{\circ}$. If follows that $\Image{u} $ is contained in a certain toric $\C^n$-chart. Note that holomorphic strips in a class $\beta$ are contained in the same toric chart, since they have the same the intersection numbers with the toric prime divisors. Moreover, by the maximum principle they are contained in a compact subset of $\C^n$, and thus the moduli space is compact by Gromov's compactness theorem.
	
	Both the standard symplectic form $\omega_{\mathrm{std}}$ on $\C^n$ and the restriction of the symplectic form $\omega$ on $X$ to the toric $\C^n$ chart are $T^n$-invariant.
	By a $T^{n}$-equivariant Moser argument, we have a one-parameter family of $T^n$-equivariant symplectomorphic embeddings $\rho_t:(\C^n,\omega) \to (\C^n,\omega_{\mathrm{std}})$ for $t\in[0,1]$ such that $\rho_0=\mathrm{Id}$ and $\rho_1^*\,\omega_{\mathrm{std}}=\omega$.  The isotopy $\rho_t$ can be realized as follows: take a one-parameter family of toric K\"ahler forms $\omega_t$ on $X$ whose corresponding moment map polytopes $P_t\subset \R_{\geq 0}^n$ interpolate between that of $\omega$ and $\R_{\geq 0}^n$.  
	Using the symplectic toric coordinates, $\rho_t$ are simply given by the inclusions $P^o \to \R_{\geq 0}^n$ where $P^o$ is the part of the moment polytope that corresponds to the toric $\C^n$-chart.
	
	Let $L_{i,t}=\rho_t^{-1}(L_i)$ be the Lagrangians in $(\C^n,\rho_t^*\omega_{\mathrm{std}})$. We fix the standard complex structure on $\C^n$ which is compatible with $\rho_t^*\omega_{\mathrm{std}}$ for all $t$. 
	The moduli space of holomorphic discs bounded by $L_{i,t}$ for $t\in[0,1]$ gives a cobordism of moduli spaces. Indeed, this is the simplest case of Fukaya's trick \cite{Fukaya10}: since $\beta$ has the minimal Maslov index and all holomorphic discs in $\beta$ are supported in $\C^n$, disc and sphere bubbling do not occur, and therefore the cobordism has no extra boundary component.  
\end{proof}

The above proof gives a classification of the holomorphic strip classes:

\begin{corollary} \label{cor:hol-strip-class}
	The holomorphic strip classes in $\pi_2(X^{\circ}, L_1\cup L_2)$ of Maslov index $1$ with input corner in $\{a^{\ell_1,\ell_2}\}\times T^{n-1}$ and output corner in $ \{b^{\ell_1,\ell_2}\}\times T^{n-1}$ are given by $\beta_0$ which does not intersects $w=0$, and $\beta_i$, $i=1,\ldots,m$, which intersects the toric divisor $D_i$ exactly once but does not intersect $D_j$ for $j\not=i$.  This gives a one-to-one correspondence between $\beta_1,\ldots,\beta_m$ and the Maslov index $2$ basic disc classes $\beta_1^L,\ldots,\beta_m^L\in \pi_2(X,L)$ bounded by a regular toric fiber $L$ of $X$.
\end{corollary}

The correspondence above can be realized by smoothing the corners $\{a^{\ell_1,\ell_2}\}\times T^{n-1}$  and $ \{b^{\ell_1,\ell_2}\}\times T^{n-1}$.  The $T^{n-1}$-equivariant smoothing of $L_1\cup L_2$ at these two corners (by using the symplectic reduction $X\sslash T^{n-1}$) gives a union of Clifford torus  and a Chekanov torus. Under this smoothing, the strips classes $\beta_i$ for $i=1,\ldots,m$ correspond to discs classes bounded by the Clifford torus, while the strip class $\beta_0$ corresponds to the basic disc class bounded by the Chekanov torus.

Since $L_1$ and $L_2$ bound no non-constant holomorphic discs in $X^{\circ}$, a stable holomorphic strip class has no disc components. On the other hand, recall that $H_2(X;\Z)$ is generated by primitive effective curve classes $\{C_{n+1},\ldots,C_m\}$. We have rational curves in $X$ in classes $\alpha\in \Z_{\ge 0}\cdot \{C_{n+1},\ldots,C_{m}\}$, and $c_1(\alpha)=0$. Moreover, all rational curves are contained in the divisor $D=\sum_{i=1}^m D_i$. This means we could have stable holomorphic strip classes of the form $\beta_i+\alpha$, where $\beta_i$ is a holomorphic strip class in Corollary \ref{cor:hol-strip-class}. We now classify all stable holomorphic strip classes of Maslov index $1$.

\begin{lemma} \label{stable-strip}
	The holomorphic strip classes $\beta\in\pi_2(X^{\circ}, L_1\cup L_2)$ of Maslov index $1$ with input corner in $\{a^{\ell_1,\ell_2}\} \times T^{n-1}$ and output corner in $ \{b^{\ell_1,\ell_2}\}\times T^{n-1}$ are of the form $\beta=\beta_i + \alpha$ for $i=0,\ldots,m$, and $\alpha\in H_2(X;\Z)$ is an effective curve class. In particular $\alpha=0$ for $\beta_i=\beta_0$.
\end{lemma}
\begin{proof}
This follows from the S.E.S \eqref{eq:toric_ses} and the correspondence between holomorphic strip class covering $u_L$ (resp. $u_R$) and the disc classes bounded by the Clifford (resp. Chekanov) torus. 
\end{proof}

Since $\cM_2^{L_1,L_2}(\beta_i;\{a^{\ell_1,\ell_2}\}\times T^{n-1},\{b^{\ell_1,\ell_2}\}\times T^{n-1})\cong T^{n-1}$ and $ev_0:\cM_2^{L_1,L_2}(\beta_i;\{a^{\ell_1,\ell_2}\}\times T^{n-1},\{b^{\ell_1,\ell_2}\}\times T^{n-1})\to  \{b^{\ell_1,\ell_2}\}\times T^{n-1}$ is $T^{n-1}$-equivariant for $i=0,\ldots,m$, there is exactly one holomorphic strip in class $\beta_i$ passing through $\{b_{12}\}\otimes\one_{T^{n-1}}$ and hence $\fm_{1,\beta_i}^{L_1,L_2}(a_{12} \otimes \one_{T^{n-1}})=b_{12} \otimes \one_{T^{n-1}}$. To compute the remaining terms $\fm_{1,\beta_i+\alpha}^{\bbL'_1,\bbL'_2}(a_{12} \otimes \one_{T^{n-1}})$, we identify the moduli space of holomorphic strips with the input $a_{12} \otimes \one_{T^{n-1}}$ and the output $b_{12} \otimes \one_{T^{n-1}}$ with moduli space of holomorphic discs bounded by a Clifford torus passing through a generic point.  This allows us to use the result of \cite{CCLT13} and relate the counts of the strips with the inverse mirror map.

\begin{prop} \label{prop:strip=disc}
The moduli space 
	 \[
	 \cM_2^{L_1,L_2}(\beta_i+\alpha;\{a^{\ell_1,\ell_2}\}\times T^{n-1},\{b^{\ell_1,\ell_2}\}\times T^{n-1}) \times_{L_1\cap L_2} \{b_{12}\otimes\one_{T^{n-1}}\}
	 \]
 for $i=1,\ldots,m$, and $\alpha\in H_2(X;\Z)$ is an effective curve class, is isomorphic to $\cM_1(\beta_i^L+\alpha,L) \times_{L} \{p\}$ as Kuranishi structures, for a Lagrangian toric fiber $L$ and a generic point $p\in L$.
\end{prop}
\begin{proof}
Let $L$ be a regular toric fiber and $p\in L$ a generic point. It is well known that  $\cM_1(\beta_i^L+\alpha,L)\cong T^n$ and $ev_0:\cM_1(\beta_i^L,L)\to L$ is $T^n$-invariant (see \cite{CLL}). It follows that 
\[
\cM_2^{L_1,L_2}(\beta_i;\{a^{\ell_1,\ell_2}\}\times T^{n-1},\{b^{\ell_1,\ell_2}\}\times T^{n-1}) \times_{L_1\cap L_2} \{b_{12}\otimes\one_{T^{n-1}}\}=\cM_1(\beta_i^L,L) \times_{L} \{p\}=\{\pt\}.
\]

The holomorphic strips in class $\beta_i$ intersect the toric prime divisor $D_i$ at a certain point $q$. For a rational curve in class $\alpha$ to contribute to $\cM_2^{L_1,L_2}(\beta_i+\alpha;\{a^{\ell_1,\ell_2}\}\times T^{n-1},\{b^{\ell_1,\ell_2}\}\times T^{n-1}) \times_{L_1\cap L_2} \{b_{12}\otimes\one_{T^{n-1}}\}$, it has to intersect $q$ in order to connect with the strip component. Let $L$ be a regular toric fiber such that the holomorphic disc in class $\beta^L_i$ passing through $p$ intersects $D_i$ exactly at $q$. For instance, we can take a toric chart $\C^n$ with coordinates $(z_1,\ldots,z_n)$ containing $q=(0,c_2\ldots,c_{n})$, $c_j\ne 0\in \C$, $j=2,\ldots,n$, and let $L = \{|z_1|=|c_1|,\ldots,|z_n|=|c_n|\} \subset \C^n$ for some $c_1\ne 0\in \C$. The choice of $p\in L$ is arbitrary. For instance, we can take $p= (c_1,\ldots,,c_n)$. Then, the holomorphic disc $u$ in $\cM_1(\beta_i^L,L)$ which passes through $p$ on the boundary and intersects $q$ is given by $u(\zeta)=(c_1\zeta,c_2,\ldots,c_n)$.

Let $\cM^{\textrm{sph}}_1(\alpha)$ denote the moduli space of $1$-pointed genus $0$ stable maps to $X$ in class $\alpha$. We have  
	\begin{align*} \label{eq:beta_alpha_1}
	&\cM_2^{L_1,L_2}(\beta_i+\alpha;\{a^{\ell_1,\ell_2}\}\times T^{n-1},\{b^{\ell_1,\ell_2}\}\times T^{n-1}) \times_{L_1\cap L_2} \{b_{12}\otimes\one_{T^{n-1}}\}\\
	\cong& \left(\cM_{2,1}(\beta_i;\{a^{\ell_1,\ell_2}\}\times T^{n-1},\{b^{\ell_1,\ell_2}\}\times T^{n-1}) \times_{L_1\cap L_2} \{b_{12}\otimes\one_{T^{n-1}}\}\right) \times_{X} \cM^{\textrm{sph}}_1(\alpha),
	\end{align*}
	and 
    \begin{equation*} \label{eq:beta_alpha_2}
    \cM_1(\beta_i^L+\alpha) \times_{L} \{p\} \cong (\cM_{1,1}(\beta_i^L) \times_{L} \{p\}) \times_{X} \cM^{\textrm{sph}}_1(\alpha),
    \end{equation*}
where the subscript $(-,1)$ means there is one interior marked point, and the fiber product over $X$ is taken using the interior evaluation maps. In both of the expression above, the first factor of the RHS is unobstructed (it is topologically a unit disc), and therefore the obstruction simply comes from $\cM^{\textrm{sph}}_1(\alpha)$.
\end{proof}


We are now ready to prove Theorem \ref{thm:iso12}.

\begin{proof}[Proof of Theorem \ref{thm:iso12}]
$\fm_{1}^{\bbL'_1,\bbL'_2}(a_{12} \otimes \one_{T^{n-1}})$ is a linear combination of $b_{12}\otimes\one_{T^{n-1}}$ and $a_{12}\otimes X_i$, $i=1,\ldots,n-1$. The coefficient of $a_{12}\otimes X_i$ is $z_i-z'_i$ as per previous discussion. On the other hand,  
the coefficient of $b_{12}\otimes \one_{T^{n-1}}$ is given by the fiber products $\cM_2^{L_1,L_2}(\beta_i+\alpha;\{a^{\ell_1,\ell_2}\}\times T^{n-1},\{b^{\ell_1,\ell_2}\}\times T^{n-1}) \times_{L_1\cap L_2} \{b_{12}\otimes\one_{T^{n-1}}\}$, 
where $\beta_i+\alpha$ is a stable holomorphic strip class in Lemma \ref{lem:hol-strip-X} and \cite{CCLT13}. For the strip class $\beta_0$, the fiber product is simply one point. Due to our choices of gauge cycles, $\partial \beta_0$ only intersects the gauge hypertorus dual to $v_1$ in $L_2$ (with intersection number $1$). Thus, $\beta_0$ contributes the term $-\bT^{\omega(\beta_0)}$. Since $L_1$ and $L_2$ are chosen such that $\beta_0$ and $\beta_1$ have the same symplectic area, this equals to $-\bT^{\omega(\beta_1)}$.

	
	For $\beta_i+\alpha$, $i=1,\ldots,m$, the number of holomorphic strips in the fiber product equals to the open Gromov-Witten invariant $n_1(\beta^L_i+\alpha)$ of a regular toric fiber $L$ of $X$ by Proposition~\ref{prop:strip=disc}. The boundary $\partial \beta^L_i$ corresponds to the primitive generator $v_i = (v_i - v_1) + v_1$ and has holonomy $z_n z^{\vec{v}_i'}$. The generating function of open Gromov-Witten invariants of $L$ is given by
	\[
	 \sum_{i=1}^m \bT^{\omega(\beta^L_i)} z_n \vec{z}^{v_i'} \sum_{\alpha} n_1(\beta^L_i+\alpha) \bT^{\omega{(\alpha})}. 
	\]
	Thus, the coefficient of $b_{12}\otimes \one_{T^{n-1}}$ equals to
	\[
	\begin{array}{l}
	 \bT^{\omega(\beta_1)} \left( -1 + (z_n')^{-1}z_n \sum_{i=1}^m \bT^{\omega(\beta_i-\beta_1)} \vec{z}^{v_i'} \sum_{\alpha} n_1(\beta^L_i+\alpha) \bT^{\omega{(\alpha)}}\right) \\
	  = \bT^{\omega(\beta_1)} \left( -1 + (z_n')^{-1} z_n f(z_1,\ldots,z_{n-1})\right). 
	  \end{array}
	\]
Therefore, the cocycle condition $\fm_{1}^{\bbL'_1,\bbL'_2}(a_{12} \otimes \one_{T^{n-1}})=0$ implies the gluing formula \eqref{eq:gluing}. Here we have implicitly chosen $L$ that $\omega(\beta_i)=\omega(\beta^L_i)$ for $i=1,\ldots,m$. Conversely, the gluing formula \eqref{eq:gluing} implies $\fm_{1}^{\bbL'_1,\bbL'_2}(a_{12} \otimes \one_{T^{n-1}})=0$, and similarly $\fm_{1}^{\bbL'_2,\bbL'_1}(b_{21} \otimes \one_{T^{n-1}})=0$. 

$\fm^{\bbL'_1, \bbL'_2, \bbL'_1}_2(b_{21} \otimes \one_{T^{n-1}},a_{12} \otimes \one_{T^{n-1}})$ and $\fm^{\bbL'_2, \bbL'_1, \bbL'_2}_2(a_{12} \otimes \one_{T^{n-1}},b_{21} \otimes \one_{T^{n-1}})$ are given by the moduli spaces
\[
\cM_3^{L_1,L_2,L_1}(\beta_i+\alpha;\{a^{\ell_1,\ell_2}\}\times T^{n-1},\{b^{\ell_1,\ell_2}\}\times T^{n-1}) \times_{L_1} \{\one_{L_1}\}, 
\]
and 
\[
\cM_3^{L_2,L_1,L_2}(\beta_i+\alpha;\{b^{\ell_1,\ell_2},\{a^{\ell_1,\ell_2}\}\times T^{n-1}\}\times T^{n-1}) \times_{L_2} \{\one_{L_2}\}, 
\]
where the target of the evaluation maps at the third marked points are $L_1$ and $L_2$, respectively. Since $f^{L_1}$ and $f^{L_2}$  were chosen such that the image of the maximum points $\one_{L_1}$ and $\one_{L_2}$ lie on the boundary of $u_R$, the moduli spaces are empty except for $\beta_0$. In which case we get
\[
\fm^{\bbL'_1, \bbL'_2, \bbL'_1}_2(b_{21} \otimes \one_{T^{n-1}},a_{12} \otimes \one_{T^{n-1}})= \bT^{\omega(\beta_0)} \one_{L_1},
\]
and
\[
\fm^{\bbL'_2, \bbL'_1, \bbL'_2}_2(a_{12} \otimes \one_{T^{n-1}},b_{21} \otimes \one_{T^{n-1}}) = \bT^{\omega(\beta_0)} \one_{L_2},
\]
where $\one_{L_1}$ and $\one_{L_2}$ are the maximum points on $L_1$ and $L_2$, respectively. This bears no effect on the gluing formula
\end{proof}

\begin{remark}
Recall from Section \ref{sec:pearl_complex} that the maximum points $\one_{L_1}=\one_{L_1}^{\blT}$ and $\one_{L_2}=\one_{L_2}^{\blT}$ are only homotopy units. Since $L_1$ and $L_2$ do not bound any non-constant holomorphic discs, we have
\[
\fm_1^{\bbL'_i}(\one_{L_i}^{\gT})=\one_{L_i}^{\wT}-\one_{L_i}^{\blT}, \quad i=1,2,
\]
where $\one_{L_i}^{\gT}$ is the degree $-1$ homotopy between $\one_{L_i}^{\wT}$ and $\one_{L_i}^{\blT}$. This means $\one_{L_1}^{\blT}$ and $\one_{L_2}^{\blT}$ are cohomologous to the strict units, and therefore the isomorphism equations \eqref{eq:iso} are indeed satisfied.
\end{remark}

We next derive the gluing formula between $\bbL'_i$, $i=1,2$, and the immersed Lagrangian brane $\bbL_0$.  
Let $(CF^{\bullet}(\bbL'_1,\bbL_0),\fm_1^{\bbL'_1,\bbL_0})$ and $(CF^{\bullet}(\bbL_0,\bbL'_2),\fm_1^{\bbL_0,\bbL'_2})$ be the pearl complexes associated to the perfect Morse functions $f^{L_0,L_1}$ and $f^{L_0,L_2}$, respectively, formally deformed by the flat connections $\nabla^{(z_1,\ldots,z_n)}$, $\nabla^{(z'_1,\ldots,z'_n)}$, and $\nabla^{(z^{(0)}_2,\ldots,z^{(0)}_{n-1})}$, and the weak bounding cochain $b=uU+vV$. We denote by
\begin{align*}
a_{10}\otimes \one_{T^{n-1}}\in CF^{0}(\bbL'_1,\bbL_0), \quad b_{10}\otimes \one_{T^{n-1}} \in CF^{1}(\bbL'_1,\bbL_0),\\
a_{02}\otimes \one_{T^{n-1}}\in CF^{0}(\bbL_0,\bbL'_2), \quad b_{02}\otimes \one_{T^{n-1}} \in CF^{1}(\bbL_0,\bbL'_2),
\end{align*}
 generators corresponding to the maximum points on $\{a^{\ell_0,\ell_1}\}\times T^{n-1}$, $\{b^{\ell_0,\ell_1}\}\times T^{n-1}$, $\{a^{\ell_0,\ell_2}\}\times T^{n-1}$, and $\{b^{\ell_0,\ell_2}\}\times T^{n-1}$, respectively.



There are two holomorphic strip classes $\beta_L^{L_0,L_i}, \beta_R^{L_0,L_i}\in \pi_2(X^{\circ}, L_0\cup L_i)$ of Maslov index $1$ with the input corner in $\{a^{\ell_0,\ell_i}\}\times T^{n-1}$ and output corner in $\{b^{\ell_0,\ell_i}\}\times T^{n-1}$. They are labeled such that the image in the $w$-plane of a holomorphic strip in class $\beta_L^{L_0,L_i}$ contains $0$. We assume the Lagrangians are chosen such that $\beta_L^{L_0,L_i}$ and $\beta_R^{L_0,L_i}$ have the same symplectic area $A$. There are exactly one holomorphic strip in each of these strip classes passing through $b_{10}\otimes \one_{T^{n-1}}$ at the output corner. The lift of boundaries of these holomorphic strips in $\widetilde{L}_0\cong \bS^2 \times T^{n-2}$ are curve segments in the $\bS^2$-factor and points in the $T^{n-2}$-factor. In particular, the lower arc of the holomorphic strip in class $\beta_L^{L_0,L_i}$ is a curve segment connecting $(p,\one_{T^{n-2}})$ and $(q,\one_{T^{n-2}})$. The perfect Morse function $f^{\bS^2}$ on $\bS^2$ is chosen such that the two flow lines connecting $p$ and $q$ to the minimum point $a_{\bS^2}$ do not intersect with these curve segments. (See \cite[Section 3.3]{HKL}.)  

The proof of the following gluing formula is similar to that of \cite[Theorem 3.7]{HKL}, except that here we need to take care of the contributions of non-constant holomorphic discs of Maslov index $0$ bounded by $L_0$ by using $T^{n-2}$-equivariant perturbations introduced in \cite{FOOO-T}.

\begin{theorem} \label{thm:iso-01}
	There exists a series $g(uv,z_2^{(0)},\ldots,z_{n-1}^{(0)})$ such that $a_{10}\otimes \one_{T^{n-1}}\in CF^{0}(\bbL'_1,\bbL_0)$, $a_{02}\otimes \one_{T^{n-1}} \in CF^{0}(\bbL_0,\bbL'_2)$, and $a_{12} \otimes \one_{T^{n-1}}\in CF^{0}(\bbL'_1,\bbL'_2)$ are isomorphisms if and only if
	\begin{equation}\label{eqn:L0L1pos}
	\begin{array}{l}
	z_1=z_1'= g(uv,z_2^{(0)},\ldots,z_{n-1}^{(0)})\\
	z_i = z_i' = z_i^{(0)} \textrm{ for } i=2,\ldots,n-1\\
	z_n= u^{-1}\\
	z_n'=v.
	\end{array}
	\end{equation}
	Moreover, $z=g(uv,z_2^{(0)},\ldots,z_{n-1}^{(0)})$ satisfies
	\[
	 uv = f(z,z_2^{(0)},\ldots,z_{n-1}^{(0)}),
	\]
with the same series $f$ as in \eqref{eqn:fslabftn} (also in Theorem \ref{thm:iso12}).
\end{theorem}

\begin{proof}
Let us first consider the pair $(\bbL'_1,\bbL_0)$. The output of $\fm_1^{\bbL'_1,\bbL_0}(a_{10}\otimes \one_{T^{n-1}})$ is a priori a linear combination of $b_{10}\otimes \one_{T^{n-1}}$ and $a_{10}\otimes X_i$ for $i=1,\ldots,n-1$, where $X_1,\ldots,X_{n-1}$ are the degree one critical points in $\{a^{\ell_0,\ell_1}\}\times T^{n-1}$. Let $\Gamma_{k+1}\in \bm{\Gamma}_{k+1}$ denote the stable trees with exactly $k$ input vertices and one interior vertex $v$. 
The only possible stable trees contributing to $b_{10}\otimes \one_{T^{n-1}}$ are the following:
\begin{enumerate}
\item $\Gamma_2$ with input $a_{10}\otimes \one_{T^{n-1}}$ and decoration $\beta_v=\beta^{L_0,L_1}_R$. 
\item $\Gamma_{2k+1}$, $k\ge 1$, with inputs of the form $a_{10}\otimes \one_{T^{n-1}},U,V,\ldots,U,V,U$ and decoration $\beta_v=\beta^{L_0,L_1}_L+\beta$, where $\beta$ is a stable disc class of Maslov index $0$ bounded by $L_0$.
\item Stable trees $\Gamma$ in $L_0$ with at least one interior vertex and inputs $U$ and $V$ attached to the interior vertex $\Gamma_2$ in (1) or $\Gamma_{k+1}$ in (2) via flow lines.
\end{enumerate}

For case (1), there is exactly one holomorphic strip in class $\beta^{L_0,L_1}_R$ with input corner in $\{a^{\ell_0,\ell_1}\}\times T^{n-1}$ and output corner intersecting $b_{10}\otimes \one_{T^{n-1}}$, similar to the proof of Theorem~\ref{thm:iso12}. 
This contributes $-\bT^{A}$ to $b_{10}\otimes \one_{T^{n-1}}$. 

For case (2), the moduli spaces in consideration are of the form
\begin{equation} \label{eq:strip+polygon}
\cM_{2k+1}(\beta_L^{L_0,L_1}+\beta;a_{10}\otimes \one_{T^{n-1}},U,V,\ldots,U,V,U). 
\end{equation}
Since all domain discs in $\cM_{2k+1}(\beta_L^{L_0,L_1}+\beta)$ are singular (for $\beta=0$, we require $k\ge 2$), the underlying Kuranishi structure of the moduli spaces in \eqref{eq:strip+polygon} is the fiber product
\begin{equation} \label{eq:fiber_product}
\cM_3(\beta_L^{L_0,L_1};a_{10}\otimes \one_{T^{n-1}},U)\times_{\bL_0} \cM_{2k}(\beta;U,V,\ldots,U,V,U).
\end{equation}
We choose $T^{n-2}$-equivariant Kuranishi structures for both factors of \eqref{eq:fiber_product} so that their fiber product is also $T^{n-2}$-equivariant. Let $E_{1}\oplus E_{2}$ denote the obstruction bundles on \eqref{eq:fiber_product}. We choose $T^{n-2}$-equivariant multi-sections  $\fs_1$, $\fs_2$, $(s_1,s_2)$, of $E_1$, $E_2$ and $E_{1}\oplus E_{2}$ (transversal to the zero section), respectively. By compatibility of perturbations, we have $s_1=\fs_1$ and $(s_2)|_{s_1^{-1}(0)}=\fs_2$. Since  $\fs_2^{-1}(0)$ is a $T^{n-2}$-invariant chain with expected dimension $n-3$, it must be the zero chain. This means $(s_1,s_2)^{-1}(0)=\emptyset$, and the moduli spaces in \eqref{eq:strip+polygon} do not contribute unless $\beta=0$ and $k=1$. There is exactly one holomorphic strip in $\cM_{3}(\beta_L^{L_0,L_1};a_{10}\otimes \one_{T^{n-1}},U)$ with output corner passing through $b_{10}\otimes \one_{T^{n-1}}$. This contributes $z_nu\bT^{A}$ to $b_{10}\otimes \one_{T^{n-1}}$


Finally, for case (3), if a stable tree $\Gamma$ in $L_0$ has at least two interior vertices and inputs $U$ and $V$, then by choosing $T^{n-2}$-equivariant perturbations for the fiber products at its interior vertices whose inputs are $U$ and $V$, the moduli space of such configurations is empty by the same argument as in Lemma \ref{lemma:L_0}. If a stable tree $\Gamma$ has exactly one interior vertex, then by choosing $T^{n-2}$-equivariant perturbations for the fiber product at the interior vertex, the output are free $T^{n-2}$-orbits situated near either $p\times T^{n-2}$ or $q\times T^{n-2}$, which do not flow to the boundary of holomorphic strip in class $\beta^{L_0,L_1}_R$ or $\beta^{L_0,L_1}_L$ intersecting $b_{10}\otimes \one_{T^{n-1}}$ due to our choice of the Morse function $f^{\bS^2}$. Therefore, the coefficient of $b_{10}\otimes \one_{T^{n-1}}$ is $(u-z_n^{-1})\bT^A$.


Let us now consider the coefficient of $a_{10}\otimes X_i$, $i=1,\ldots,n-1$. We have a pair of Morse flow lines from $a_{10}\otimes\one_{T^{n-1}}$ to $a_{10}\otimes X_i$. For $i=2,\ldots,n-1$, this gives $(z_i-z_i^{(0)})\cdot a_{10}\otimes X_i$. For $i=1$, one of the flow lines intersects the gauge hypertorus dual to $v'_1$ in $L_1$ and contributes $z_1$ to $a_{10}\otimes X_1$. While the other flow line does not intersect with any gauge cycles, it can however be attached by Maslov index $0$ polygons bounded by $L_0$ via an isolated Morse flow line from $\{p\}\times T^{n-2}$ or $\{q\}\times T^{n-2}$ to $\{a^{\ell_1,\ell_2}\}\times T^{n-1}$. Recall that the corners $U$ and $V$ must appear in pairs for a stable polygon bounded by $L_0$, since its boundary is contained in the the immersed loci. Thus, this contributes a series $-g(uv,z_2^{(0)},\ldots,z_{n-1}^{(0)})$ to $a_{10}\otimes X_1$. In conclusion, we have
	\begin{equation}\label{eq:m_10}
\begin{aligned}
\fm_1^{\bbL'_1,\bbL_0}(a_{10}\otimes \one_{T^{n-1}})= (z_nu-1)\bT^{A}\cdot b_{10}\otimes \one_{T^{n-1}}  &+ (z_1-g(uv,z_2^{(0)},\ldots,z_{n-1}^{(0)})) \cdot a_{10}\otimes X_1 \\ &+ \sum_{i=2}^{n-1} (z_i-z_i^{(0)}) \cdot a_{10}\otimes X_i.
\end{aligned}
	\end{equation}
	
Similarly, for the pair $(\bbL_0,\bbL'_2)$, we have 
\begin{equation} \label{eq:m_02}
\begin{aligned}
\fm_1^{\bbL_0,\bbL'_2} (a_{02}\otimes \one_{T^{n-1}})=((z'_n)^{-1}v-1)\bT^{A}\cdot b_{02}\otimes \one_{T^{n-1}}  &+ (z_1'-\tilde{g}(uv,z_2^{(0)},\ldots,z_{n-1}^{(0)})) \cdot a_{02}\otimes X'_1\\ &+ \sum_{i=2}^{n-1} (z_i'-z_i^{(0)}) \cdot a_{02}\otimes X'_i
\end{aligned}
\end{equation}
for some series $\tilde{g}$, where $X'_1,\ldots,X'_{n-1}$ are the degree one critical points in $\{a^{\ell_0,\ell_2}\}\times T^{n-1}$. $a_{10}\otimes \one_{T^{n-1}}$ and $a_{02}\otimes \one_{T^{n-1}}$ are isomorphisms if and only if the coefficients in \eqref{eq:m_10} and \eqref{eq:m_02} are zero.
	
	We have $\fm_2^{\bbL'_1,\bbL_0,\bbL'_2}(a_{02}\otimes \one_{T^{n-1}},a_{10}\otimes \one_{T^{n-1}}) =\bT^{\Delta} \cdot a_{12}\otimes \one_{T^{n-1}}$ contributed by a holomorphic section of symplectic area $\Delta$ over the triangle in the $w$-plane with corners $a^{\ell_0,\ell_1}$, $a^{\ell_0,\ell_2}$ and $a^{\ell_1,\ell_2}$. Recall from Theorem \ref{thm:iso12} that $a_{12}\otimes \one_{T^{n-1}}$ is an isomorphism between $\bbL_1$ and $\bbL_2$ if only if $z_i' = z_i$ for $i = 1,\ldots,n-1$ and $z_n' = z_n \cdot f(z_1,\ldots,z_{n-1})$. Thus, we have $g=z_1=z_1'=\tilde{g}$, and 
	\[
	uv=z'_n z_n^{-1}=f(z_1,\ldots,z_{n-1})=f(g,z_2^{(0)},\ldots,z_{n-1}^{(0)}).
	\]
\end{proof}

The series $g$ is the generating function of Maslov index $0$ stable polygons bounded by $L_0$. 
The stable polygons can attach to an isolated Morse trajectory in $L_0$ flowing into $\{a^{\ell_0,\ell_1}\}\times T^{n-1}\subset L_1 \cap L_0$ and contribute to $a_{10}\otimes X_1$.  It is difficult to compute $g$ directly since the moduli spaces involved are highly obstructed. However, Theorem \ref{thm:iso-01} gives a way to compute $g$ by solving for $z_1$ from the equation $uv=f(z_1,\ldots,z_{n-1})$.

\begin{example}
	Let $X=\C^3$.  We have 
	\[
	uv = z_3' z_3^{-1} = 1+ z_1 + z_2,
	\]
	and hence, $g(uv,z_2^{(0)})=z_1=uv-z_2-1$.
\end{example}



\subsubsection{Deformation families $\bbL_1$ and $\bbL_2$}
Let $\bm{X}_1\in\Crit(f^{L_1})$ and $\bm{X}'_1\in \Crit(f^{L_2})$ be the degree $1$ critical points whose unstable chain is the hypertori dual to $v'_1$ in $L_1$ and $L_2$, respectively. Let $x_1, x'_1\in \Lambda_+$. We denote the family of formal Lagrangian deformations $(L_1, \nabla^{(1,z_2\ldots,z_n)},x_1\bm{X}_1)$ and $(L_2, \nabla^{(1,z'_2\ldots,z'_n)},x'_1\bm{X}'_1)$ by $\bbL_1$ and $\bbL_2$, respectively. The reason for introducing these new deformation families is that they will fit better into our $\bS^1$-equivariant theory.

\begin{lemma} \label{lemma:L_1_2}
$\bbL_1$ and $\bbL_2$ are unobstructed.
\end{lemma}

\begin{proof}
Since $L_1$ and $L_2$ bound no non-constant holomorphic discs in $X^{\circ}$, the only possible contributions come from Morse theory. There are two flow lines from $\bm{X}_1$ (resp. $\bm{X}'_1$) to each of the degree $2$ critical points in $L_1$ (resp. $L_2$) which cancel with each other. The Morse flow trees with repeated inputs $\bm{X}_1$ (resp. $\bm{X}'_1$) are empty for generic perturbations.
\end{proof}

In \cite[Theorem 4.7]{KLZ19}, a particular type of perturbations was used for the computation of the toric superpotential so that the divisor axiom holds, and we have the change of variables $z=e^{x}$. Below, we choose analogous  perturbations for the gluing formulas between the Lagrangian branes $\bbL_0$, $\bbL_1$, and $\bbL_2$.
	
\begin{theorem} \label{thm:iso12-x}
	$a_{12}\otimes \one_{T^{n-1}}$ is an isomorphism between $\bbL_1$ and $\bbL_2$ with an inverse $b_{21}\otimes \one_{T^{n-1}}$ if and only if
	\begin{equation} \label{eq:gluing}
	\begin{array}{l}
	e^{x'_1}=e^{x_1},\\
	z_i' = z_i, \quad i = 2,\ldots,n-1,\\
	z_n' = z_n \cdot f(e^{x_1},\ldots,z_{n-1}),
	\end{array}
	\end{equation}
	where $f$ is given by \eqref{eqn:fslabftn}.
\end{theorem}

\begin{proof}
The proof parallels that of Theorem \ref{thm:iso12}, so we will simply highlight the parts that depend on such choices of perturbations.

$\fm_{1}^{\bbL'_1,\bbL_2}(a_{12} \otimes \one_{T^{n-1}})$ is a linear combination of $b_{12}\otimes\one_{T^{n-1}}$ and $a_{12}\otimes X_i$, $i=1,\ldots,n-1$. The coefficient of $a_{12}\otimes X_i$ for $i=2,\ldots,n-1$ is $z_i-z'_i$ as before. For $a_{12}\otimes X_1$, there are two Morse flow lines $\gamma_1$ and $\gamma_2$ from $a_{12} \otimes \one_{T^{n-1}}$. The flow line $\gamma_1$ which originally contributed $z_1$ in the setting of Theorem \ref{thm:iso12} can now intersect with flow lines from $k$ copies of $X_1$ in $L_1$ at a point and then follow to $a_{12}\otimes X_1$, forming Morse flow trees. 
The moduli spaces of such configurations are not transverse for $k>1$, since the unstable chain of $\bm{X}_1$ (which is the hypertorus dual in $L_1$ to $v'_1$) is not transverse with itself. Therefore, we have to choose transversal perturbations. For $k>1$, let $\bm{X}_1^{(1)},\ldots,\bm{X}_1^{(k)}$ be disjoint small perturbations of the hypertorus along the direction of $\gamma_1$ (and ordered along that direction). We choose the perturbation to be the average of all permutations of $\bm{X}_1^{(1)},\ldots,\bm{X}_1^{(k)}$, resulting in the factor $\frac{x_1^k}{k!}$. These configurations together contribute $e^{x_1}$ to $a_{12}\otimes X_1$. Similarly, the Morse flow trees formed by the flow line $\gamma_2$ and the flow lines from copies of $\bm{X}'_1$ contribute $e^{x'_1}$ to $\cdot a_{12}\otimes X_1$

For the coefficient of $b_{12}\otimes\one_{T^{n-1}}$, the strip class $\beta_0$ contributes $-\bT^{\omega(\beta_0)}$ as before. On the other hand, flow lines from $k$ copies of $X_1$ can now attach to the upper arc (in $L_1$) of stable strips in classes $\beta_i+\alpha$, $i=1,\ldots,m$, forming pearly trees. Again, these fiber products are not transversal. The transversal perturbations we choose is the following. For the strip class $\beta_i+\alpha$ and $k>1$, let $\bm{X}_1^{(i,1)},\ldots,\bm{X}_1^{(i,k)}$ be disjoint small perturbations of the hypertorus along the direction of $\partial \beta_i$ (and ordered along that direction). We choose the perturbation of the fiber product to be the average of all permutations of $\bm{X}_1^{(i,1)},\ldots,\bm{X}_1^{(i,k)}$. By choosing such perturbations, we get $\bT^{\omega(\beta_1)} \left( -1 + (z_n')^{-1}z_n f(e^{x_1},\ldots,z_{n-1})\right) \cdot b_{12}\otimes\one_{T^{n-1}}$. 

Finally, we can eliminate the possibility of Morse flow trees with repeated inputs $X_1$ attaching to the above configurations, since the moduli space of these flow trees are empty for generic perturbations.
\end{proof}

By choosing similar perturbations as in the proof above, we also have the gluing formulas between $\bbL_i$ and $\bbL_0$.

\begin{theorem} \label{thm:iso-01-x}
 $a_{10}\otimes \one_{T^{n-1}}\in CF^0(\bbL_1,\bbL_0)$, $a_{02}\otimes \one_{T^{n-1}}\in CF^0(\bbL_0,\bbL_2)$, and $a_{12} \otimes \one_{T^{n-1}}\in CF^{0}(\bbL'_1,\bbL'_2)$ are isomorphisms if and only if
	\begin{equation}\label{eqn:L0L1pos}
	\begin{array}{l}
	e^{x_1}=e^{x'_1}= g(uv,z_2^{(0)},\ldots,z_{n-1}^{(0)})\\
	z_i = z_i' = z_i^{(0)} \textrm{ for } i=2,\ldots,n-1\\
	z_n= u^{-1}\\
	z_n'=v.
	\end{array}
	\end{equation}
	Moreover, $z=g(uv,z_2^{(0)},\ldots,z_{n-1}^{(0)})$ satisfies
	\[
	uv = f(z,z_2^{(0)},\ldots,z_{n-1}^{(0)}),
	\]
	where the series $f$ is same as in \eqref{eqn:fslabftn} and Theorem \ref{thm:iso12}, and the series $g$ as in Theorem \ref{thm:iso-01}.
\end{theorem}

\subsection{$\bS^1$-equivariant disc potential of $L_0$}
For the $T^n$-action on a compact toric semi-Fano manifold $X$ with $\dim_{\C}X=n$, the equivariant disc potential $W_{T^n}$ of regular toric fibers in $X$ was computed in \cite{KLZ19}. It recovers the $T^n$-equivariant toric Landau-Ginzburg mirror of Givental \cite{Givental}, namely,
\[
W_{T^n} = W(e^{x_1},\ldots,e^{x_n}) + \sum_{i=1}^n x_i \lambda_i.
\]
Here, $W$ is the superpotential of Givental and Hori-Vafa  \cite{Givental, HV,LLY3}, $x_1,\ldots,x_n$ are parameters for the boundary deformations by hypertori,
and $\lambda_1,\ldots,\lambda_n$ are equivariant parameters generating  $H^{\bullet}_{T}(\pt;\Lambda_0)=\Lambda_0 [\lambda_1,\ldots,\lambda_n]$. The $\bS^1$-equivariant disc potential for immersed $2$-sphere $\cS^2$ was also computed in \cite{KLZ19} in the unobstructed setting using gluing formulas. 

In this section, we study the $\bS^1$-equivariant disc potential of the immersed SYZ fiber $L_0\cong \cS^2\times T^{n-2}$ in a toric Calabi-Yau $n$-fold $X$.
The Lagrangians $L_0$, $L_1$ and $L_2$ and their pair-wise clean intersections are invariant under the $T^{n-1}$-action rotating along the directions of $v'_1,\ldots, v'_{n-1}$. 
The $\bS^1(\subset T^{n-1})$-action of our interest is the rotation along the $v'_1$-direction. It acts on $L_1$ and $L_2$ by rotating the second $\bS^1$-factor and acts on $L_0$ by rotating the $\bS^2$-factor fixing the nodal point.

For $i=0,1,2$, let $(CF_{\bS^1}^{\bullet}(L_i),\fm^{(L_i,\bS^1)})$ the $\bS^1$-equivariant Morse model such that 
\[
CF_{\bS^1}^{\bullet}(L_i)=CF^{\bullet}(L_i; H^{\bullet}_{\bS^1}(\pt;\Lambda_0)),
\]
where $CF^{\bullet}(L_i; H^{\bullet}_{\bS^1}(\pt;\Lambda_0))$ is generated by critical points of the Morse function $f^{L_i}$ as explained in \ref{sec:partial_units}. Here, we have suppressed the superscript $\dagger$ for simplicity of notations.
 
Let us make a decomposition $L_i=\bS^1 \times T^{n-1}$ for $i=1,2$ so that $\bS^1$ acts freely on the $\bS^1$-factor and trivially on the $T^{n-1}$-factor. Then
\[
(L_i)_{\bS^1}= (\bS^{1})_{\bS^1}\times T^{n-1},
\]
and
\[
(L_0)_{\bS^1} =(\cS^{2})_{\bS^1}\times T^{n-2},
\]
We have
\[
\pi_1((L_i)_{\bS^1})\cong \pi_1(T^{n-1}) =\Z\cdot \{v_1,v'_2,\ldots,v'_{n-1}\}, \quad i=0,1,2.
\] 
The flat $\Lambda_\mathrm{U}$-connections $\nabla^{(z^{(0)}_2,\ldots,z^{(0)}_{n-1})}$, $\nabla^{(1,z_2,\ldots,z_n)}$, and $\nabla^{(1,z'_2,\ldots,z'_n)}$ are $\bS^1$-equivariant, and thus induce flat connections on $(L_i)_{\bS^1}$ for $i=0,1,2$. 

Let $\bm{X}_1\in CF_{\bS^1}^{\bullet}(L_1)$, $\bm{X}'_1\in CF_{\bS^1}^{\bullet}(L_2)$, and $U, V\in  CF_{\bS^1}^{\bullet}(L_0)$. We denote by $(\bbL_0,\bS^1)$, $(\bbL_1,\bS^1)$ and $(\bbL_2,\bS^1)$ the families of formal deformations of $(L_0,\bS^1)$ by $\nabla^{(z^{(0)}_2,\ldots,z^{(0)}_{n-1})}$ and $uU+vV$ for $u,v\in \Lambda_0^2$ with $\mathrm{val}(u\cdot v)>0$, $(L_1,\bS^1)$ by $\nabla^{(1,z_2,\ldots,z_n)}$ and $x_1\bm{X}_1$, $x_1\in \Lambda_+$, and $(L_2,\bS^1)$ by $\nabla^{(1,z'_2,\ldots,z'_n)}$ and $x'_1\bm{X}'_1$, $x_1\in \Lambda_+$, respectively. By Corollary \ref{cor:unobstructed}, Lemma \ref{lemma:L_0} and Lemma \ref{lemma:L_1_2}, $(\bbL_i,\bS^1)$ is weakly unobstructed for $i=0,1,2$. This implies that
\begin{align*}
\fm^{(\bbL_0,\bS^1)}_0(1) &=W_{\bS^1}^{\bbL_0}(u,v,z^{(0)}_2,\ldots,z^{(0)})\cdot \one_{L_0},\\ 
\fm^{(\bbL_1,\bS^1)}_0(1) &= W_{\bS^1}^{\bbL_1}(x_1,z_2,\ldots,z_n)\cdot \one_{L_1},\\
\fm^{(\bbL_2,\bS^1)}_0(1) &= W_{\bS^1}^{\bbL_2}(x'_1,z'_2,\ldots,z'_n)\cdot \one_{L_2}.	
\end{align*}
for some series $W_{\bS^1}^{\bbL_i}$ ($i=0,1,2$).
We will refer to the functions $W_{\bS^1}^{\bbL_i}$ as the $\bS^1$-equivariant disc potential of $\bbL_i$. 

Since $L_1$ does not bound any non-constant holomorphic discs in $X^{\circ}$, the following is an immediate consequence of \cite[Lemma 4.4]{KLZ19}.
\begin{prop} 
	\label{prop:W_T}
	The $\bS^1$-equivariant disc potential  $W_{\bS^1}^{\bbL_1}$ of $\bbL_1$ is given by
	\[
	W_{\bS^1}^{\bbL_1}= \lambda_1 x_1,
	\]
where $\lambda_1$ is the equivariant parameter of the $\bS^{1}$-action, i.e., $H^{\bullet}_{\bS^1}(\pt;\Lambda_0)$=$\Lambda_0 [\lambda_1]$. 	
\end{prop}

Let $(CF^{\bullet}_{\bS^1}(\bbL_1,\bbL_0),\fm^{(\bbL_1,\bS^1),(\bbL_0,\bS^1)})$ and $(CF^{\bullet}_{\bS^1}(\bbL_0,\bbL_2),\fm^{(\bbL_0,\bS^1),(\bbL_2,\bS^1)})$ be the $\bS^{1}$-equivariant Morse models with
\[
CF^{\bullet}_{\bS^1}(\bbL_1,\bbL_0)=CF^{\bullet}(\bbL_1,\bbL_0;H^{\bullet}_{\bS^1}(\pt;\Lambda_0))
\]
and
\[
CF^{\bullet}_{\bS^1}(\bbL_0,\bbL_2)=CF^{\bullet}(\bbL_0,\bbL_2;H^{\bullet}_{\bS^1}(\pt;\Lambda_0)).
\]
The gluing formulas between $(\bbL_0,\bS^1)$, $(\bbL_1,\bS^1)$ and $(\bbL_2,\bS^1)$ remain the same as before.

\begin{prop} 	\label{prop:gluing-equiv}
$a_{10}\otimes \one_{T^{n-1}}\in CF^0_{\bS^1}(\bbL_1,\bbL_0)$ and $a_{02}\otimes \one_{T^{n-1}}\in CF^0_{\bS^1}(\bbL_0,\bbL_2)$ are isomorphisms if and only if
\begin{equation}\label{eqn:L0L1pos}
\begin{array}{l}
e^{x_1}=e^{x'_1}= g(uv,z_2^{(0)},\ldots,z_{n-1}^{(0)})\\
z_i = z_i' = z_i^{(0)} \textrm{ for } i=2,\ldots,n-1\\
z_n= u^{-1}\\
z_n'=v.
\end{array}
\end{equation}
Moreover, $z=g(uv,z_2^{(0)},\ldots,z_{n-1}^{(0)})$ satisfies
\[
uv = f(z,z_2^{(0)},\ldots,z_{n-1}^{(0)}),
\]
where the series $f$ is the same as in \eqref{eqn:fslabftn} (or in Theorem \ref{thm:iso12}), and hence, $g$ equals the one in Theorem \ref{thm:iso-01}.
\end{prop}

\begin{proof}
Since the equivariant parameter $\lambda_1$ has degree $2$, it does not appear in the isomorphism equations \eqref{eq:iso}. Thus, we have 
\begin{align*}
\fm_1^{(\bbL_1,\bS^1),(\bbL_0,\bS^1)}(a_{10}\otimes \one_{T^{n-1}})  &=\fm^{\bbL_1,\bbL_0}(a_{10}\otimes \one_{T^{n-1}}),\\ \fm_1^{(\bbL_0,\bS^1),(\bbL_2,\bS^1)}(a_{02}\otimes \one_{T^{n-1}})  &=\fm_1^{\bbL_0,\bbL_2}(a_{02}\otimes \one_{T^{n-1}}),\\
\fm_2^{(\bbL_1,\bS^1),(\bbL_0,\bS^1),(\bbL_2,\bS^1)}(a_{02}\otimes \one_{T^{n-1}},a_{10}\otimes \one_{T^{n-1}})&=\fm_2^{\bbL'_1,\bbL_0,\bbL'_2}(a_{02}\otimes \one_{T^{n-1}},a_{10}\otimes \one_{T^{n-1}}).
\end{align*}	
The statement then follows from Theorem \ref{thm:iso-01} and Theorem \ref{thm:iso-01-x}.
\end{proof}
 
Since the disc potentials are compatible with the gluing formulas (see \cite[Theorem~4.7]{CHL-glue}), we can compute $W_{\bS^1}^{\bbL_0}$ in terms of $g$ making use of the propositions above.

\begin{theorem} \label{thm:equiv}
	The $\bS^1$-equivariant disc potential $W_{\bS^1}^{\bbL_0}$ of $\bbL_0$ is given by
	\begin{equation} \label{eq:equiv_potential}
	W_{\bS^1}^{\bbL_0} = \lambda_1 \log g(uv,z_2^{(0)},\ldots,z_{n-1}^{(0)}).
	\end{equation}
\end{theorem}


\begin{remark}
The expression $\log g(uv,z_2^{(0)},\ldots,z_{n-1}^{(0)})$ depends on choices of a framing $\{v'_1,\ldots,v'_{n-1}\}$.
\end{remark}

In dimension three, by setting $u=v=0$, we obtain the equivariant term $\log g(0, z_2^0)$.  Integration of this term is the disc potential of the Aganagic-Vafa brane \cite{AKV,GZ}.  Our formulation has the advantage that it works in all dimensions.

\subsection{Examples} \label{sec:eg}
Before we finish the section, we provide explicit numerical computations for a few interesting toric Calabi-Yau manifolds in various dimensions.

\subsubsection{Inner branes in $K_{\bP^2}$}
	For $X=K_{\bP^2}$, the primitive generators of the fan $\Sigma$ are 
	\[
	v_1 = (0,0,1), \quad v_2=(1,0,1), \quad v_3=(0,1,1),\quad v_4=(-1,-1,1).
	\]
	Let $L_0 \cong \cS^2 \times \bS^1$ be an immersed Lagrangian at the `inner branch' dual to the cone $\R_{\geq 0}\langle v_1,v_2\rangle$, which bounds a primitive holomorphic disc $\beta$ with area $A$, as depicted in Figure \ref{fig:KP2}.  The area of the curve class is denoted by $\tau$.  We fix the compact chamber dual to $v_1$ and the framing $\{v_i'=v_{i+1}-v_1\mid i=1,2\}$ to compute the $\bS^1$-equivariant disc potential of $L_0$. Note that different choices of a chamber and a framing give different Morse functions on $L_0$, and the disc potentials will change accordingly.  The choice of chamber matters for $(uv)$ but not the discs with no corners.
	
	The gluing formula between $L_0$ and a regular toric fiber $L_1$ reads
	\[
	uv = (1+\delta(\bT^\tau)) + \exp x_1 + \bT^A z_2 + \bT^{\tau - A} \exp (-x_1) \cdot z_2^{-1} 
	\]
	where 
	\[
	\delta(\bT^\tau)=-2q+5q^2-32q^3+286q^4-3038q^5+O(q^6), \quad q=\bT^\tau, 
	\]
	is the generating function of open Gromov-Witten invariants bounded by a regular toric fiber of $K_{\bP^2}$ computed in \cite{CLL}.
	
	However, notice that the left hand side lies in $\Lambda_+$.  The right hand side has leading term $2$ and hence does not lie in $\Lambda_+$! Hence the gluing formula has empty solution.  It means that the formally deformed Lagrangian $L_1$ is not isomorphic to the formally deformed immersed Lagrangian $L_0$.

	\begin{figure}[h]
		\begin{center}
			\includegraphics[scale=0.6]{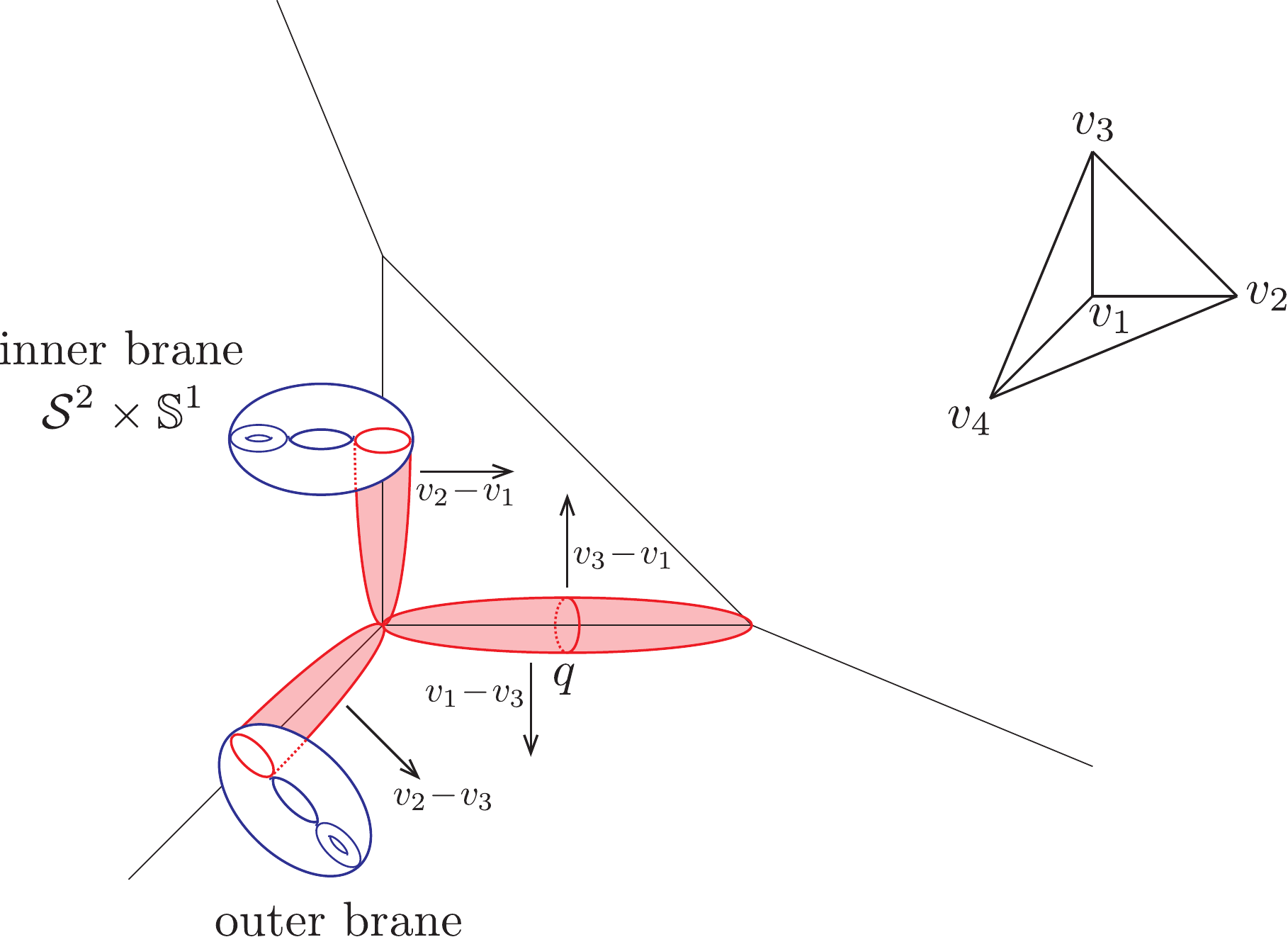}
			\caption{Immersed Lagrangian branes in $K_{\bP^2}$}\label{fig:KP2}
		\end{center}
	\end{figure}
	
	To remedy this, we take a \emph{non-trivial} spin structure of the tori $L_1$ and $L_2$ in the $v'_1$-direction.  This systematically introduces extra signs to the orientations of moduli spaces.  Formally it gives the change of coordinates $x_1 \mapsto x_1 + i\pi$.  Then the gluing formula becomes
	\[
	 uv = (1+\delta(\bT^\tau)) - \exp (x_1) + \bT^A z_2 - \bT^{\tau - A} \exp (-x_1) \cdot z_2^{-1}.
	\]
	Both sides are now in $\Lambda_+$.  We can then solve $\exp (x_1)$ in terms of $uv$ and $z=\bT^A z_2$, which gives the series $g$ in Theorem \ref{thm:equiv}.
	
	A direct calculation gives that the coefficient of $\lambda_1$ equals to $a_0 + a_1 uv + a_2 \cdot (uv)^2 +O((uv)^3)$ where
	\begin{align*}
	a_0 =&  \left(z - \frac{z^2}{2}+\frac{z^3}{3}+O(z^4)\right) + (-z^{-1} - z + 2 z^2 - 3z^3 + O(z^4)) q \\
	+& \left(-\frac{3}{2}z^{-2}+2 z^{-1} + 5z - \frac{27}{2} z^2 + 27 z^3 +O(z^4)\right)q^2 +O(q^3) \\
	a_1 =& (-1+z-z^2 + z^3 + O(z^4))+(-2 z^{-1} + 4 - 8z + 14 z^2 - 22 z^3 + O(z^4)) q \\
	+& (-6 z^{-2} + 18 z^{-1} - 41 + 92 z - 189 z^2 + 356 z^3 + O(z^4)) q^2 + O(q^3) \\
	a_2 =& \left(-\frac{1}{2} + z - \frac{3}{2} z^2 + 2 z^3 + O(z^4)\right) + \left(-3 z^{-1} + 10 - 24z + 48 z^2 - 85 z^3 + O(z^4) \right) q \\
	+& \left(-15 z^{-2} + 66 z^{-1} - 196 + 489 z - 1080 z^2 + 2170 z^3 + O(z^4) \right) q^2 + O(q^3). \\
	\end{align*}	
	This differs by $z \mapsto -z$ from physicists' convention, which results from different ways of parametrizing the flat connections.  As noted in \cite{AKV}, all the $z^j$-coefficients of the leading order term ($q^0$) in $a_0$ multiplied by $j$ are integers. This reflects that the automorphism group of a holomorphic disc in class $j\beta$ ($\beta$ is the primitive disc class) is of order $j$. On the other hand, we note that all the $z^j$-coefficients in $\ell \cdot a_\ell$ are \emph{integers} in the examples we have computed. 
	
	
	
	We write the series as $\sum -a_{jk\ell} (-z)^j (-q)^k (uv)^\ell$ and record the coefficients in the following table.
	
	{\small
		\noindent \begin{tabular}{|c|c|c|c|c|} 
			\hline
			\multicolumn{5}{|c|}{$\mathrm{ord}(uv)=0$} \\
			\hline
			& \multicolumn{4}{|c|}{$\mathrm{ord}(q)$} \\
			\hline
			$\mathrm{ord}(z)$ & $0$ & $1$ & $2$ & $3$ \\ 
			\hline
			$-3$ & $0$ & $0$ & $0$ & $10/3$ \\ 
			\hline
			$-2$ & $0$ & $0$  & $3/2$ & $8$ \\ 
			\hline
			$-1$ & $0$ & $1$ & $2$ & $12$ \\ 
			\hline
			$0$ & $0$ & $0$ & $0$ & $0$ \\ 
			\hline	
			$1$ & $1$ & $1$ & $5$ & $40$  \\ 
			\hline					
			$2$ & $1/2$ & $2$ & $27/2$ & $122$ \\ 
			\hline	
			$3$ & $1/3$ & $3$ & $27$ & $838/3$ \\ 
			\hline	
			$4$ & $1/4$ & $4$ & $47$ & $560$ \\ 
			\hline				
		\end{tabular}
		\begin{tabular}{|c|c|c|c|c|} 
			\hline
			\multicolumn{5}{|c|}{$\mathrm{ord}(uv)=1$} \\
			\hline
			& \multicolumn{4}{|c|}{$\mathrm{ord}(q)$} \\
			\hline
			$\mathrm{ord}(z)$ & $0$ & $1$ & $2$ & $3$ \\ 
			\hline
			$-3$ & $0$ & $0$ & $0$ & $20$ \\ 
			\hline
			$-2$ & $0$ & $0$  & $6$ & $80$ \\ 
			\hline
			$-1$ & $0$ & $2$ & $18$ & $218$ \\ 
			\hline
			$0$ & $1$ & $4$ & $41$ & $520$ \\ 
			\hline	
			$1$ & $1$ & $8$ & $92$ & $1224$  \\ 
			\hline					
			$2$ & $1$ & $14$ & $189$ & $2704$ \\ 
			\hline	
			$3$ & $1$ & $22$ & $356$ & $5582$ \\ 
			\hline	
			$4$ & $1$ & $32$ & $623$ & $10828$ \\ 
			\hline				
		\end{tabular}
		\begin{tabular}{|c|c|c|c|c|} 
			\hline
			\multicolumn{5}{|c|}{$\mathrm{ord}(uv)=2$} \\
			\hline
			& \multicolumn{4}{|c|}{$\mathrm{ord}(q)$} \\
			\hline
			$\mathrm{ord}(z)$ & $0$ & $1$ & $2$ & $3$ \\ 
			\hline
			$-3$ & $0$ & $0$ & $0$ & $70$  \\ 
			\hline
			$-2$ & $0$ & $0$ & $15$  & $380$ \\ 
			\hline
			$-1$ & $0$ & $3$ & $66$ & $1320$ \\ 
			\hline
			$0$ & $1/2$ & $10$ & $196$ & $3762$ \\ 
			\hline	
			$1$ & $1$ & $24$ & $489$ & $9544$  \\ 
			\hline					
			$2$ & $3/2$ & $48$ & $1080$ & $22128$ \\ 
			\hline	
			$3$ & $2$ & $85$ & $2170$ & $47600$ \\ 
			\hline	
			$4$ & $5/2$ & $138$ & $4041$ & $96050$ \\ 
			\hline				
		\end{tabular}
	}	
	
	For $u=v=0$, it is $z\partial_z$ of the physicists' superpotential for the Aganagic-Vafa brane \cite{AKV,GZ}, which was mathematically formulated and verified via localization in \cite{Katz-Liu,Fang-Liu,fang-liu-tseng}.

\subsubsection{Outer brane in $K_{\bP^2}$}
	We may also consider an outer brane in $K_{\bP^2}$, corresponding to the codimension-two strata dual the cone $\R_{\geq 0}\langle v_2, v_3 \rangle$.  
	We fix the chamber dual to the generator $v_3$ and fix the framing $\{v_2-v_3, v_1-v_3\}$ (corresponding to the holonomy variables $\tilde{z}_1$ and $\tilde{z}_2$ respectively).
	The gluing formula becomes
	\[
	 uv = (1+\delta) \bT^A \tilde{z}_2 - \exp (\tilde{x}_1) + 1 - \bT^{\tau+3A} \exp (-\tilde{x}_1) \tilde{z}_2^3,
	 \]
	where $A$ is the area of the primitive holomorphic disc bounded by the outer brane.
	
	Putting  $z=\bT^A \tilde{z}_2$ and writing the series as $\sum -a_{jk\ell} (-z)^j (-q)^k (uv)^\ell$, record the coefficients in the following table.

	{\large
			\noindent \begin{tabular}{|c|c|c|c|c|} 
			\hline
			\multicolumn{5}{|c|}{$\mathrm{ord}(uv)=0$} \\
			\hline
			& \multicolumn{4}{|c|}{$\mathrm{ord}(q)$} \\
			\hline
			$\mathrm{ord}(z)$ & $0$ & $1$ & $2$ & $3$ \\ 
			\hline
			$0$ & $0$ & $0$ & $0$ & $0$ \\ 
			\hline	
			$1$ & $1$ & $2$ & $5$ & $32$  \\ 
			\hline					
			$2$ & $1/2$ & $2$ & $7$ & $42$ \\ 
			\hline	
			$3$ & $1/3$ & $3$ & $9$ & $164/3$ \\ 
			\hline	
			$4$ & $1/4$ & $4$ & $15$ & $80$ \\ 
			\hline				
		\end{tabular}
		\begin{tabular}{|c|c|c|c|c|} 
			\hline
			\multicolumn{5}{|c|}{$\mathrm{ord}(uv)=1$} \\
			\hline
			& \multicolumn{4}{|c|}{$\mathrm{ord}(q)$} \\
			\hline
			$\mathrm{ord}(z)$ & $0$ & $1$ & $2$ & $3$ \\ 
			\hline
			$0$ & $1$ & $0$ & $0$ & $0$ \\ 
			\hline	
			$1$ & $1$ & $2$ & $5$ & $32$  \\ 
			\hline					
			$2$ & $1$ & $4$ & $14$ & $84$ \\ 
			\hline	
			$3$ & $1$ & $8$ & $27$ & $164$ \\ 
			\hline	
			$4$ & $1$ & $14$ & $56$ & $310$ \\ 
			\hline				
		\end{tabular}
		\begin{tabular}{|c|c|c|c|c|} 
			\hline
			\multicolumn{5}{|c|}{$\mathrm{ord}(uv)=2$} \\
			\hline
			& \multicolumn{4}{|c|}{$\mathrm{ord}(q)$} \\
			\hline
			$\mathrm{ord}(z)$ & $0$ & $1$ & $2$ & $3$ \\ 
			\hline
			$0$ & $1/2$ & $0$ & $0$ & $0$ \\ 
			\hline	
			$1$ & $1$ & $2$ & $5$ & $32$  \\ 
			\hline					
			$2$ & $3/2$ & $6$ & $21$ & $126$ \\ 
			\hline	
			$3$ & $2$ & $15$ & $54$ & $328$ \\ 
			\hline	
			$4$ & $5/2$ & $32$ & $134$ & $760$ \\ 
			\hline				
		\end{tabular}
	}

Note that the first column (for $\mathrm{ord}(q) = 0$) are the same for both branes.  This reflects the fact that the holomorphic polygons (without sphere bubbling) are the same.

\subsubsection{$X=K_{\bP^3}$}
	Our method also works for higher dimensions.  For instance, the toric Calabi-Yau 4-fold $X=K_{\bP^3}$ is associated with the fan $\Sigma$ which has generators 
	\[
	v_1=(0,0,0,1),v_2=(1,0,0,1),v_3=(0,1,0,1),v_4=(0,0,1,1),v_5=(-1,-1,-1,1).
	\]  
	Let $L_0=\cS^2\times T^2$ be an inner brane at the codimension-two toric stratum corresponding to the cone $\R_{\geq 0} \langle v_1,v_2\rangle$ (a toric divisor of $\bP^3$). We fix the chamber dual to $v_1$ and the basis $\{v_2-v_1,v_3-v_1,v_4-v_1\}$. 
	
	The SYZ mirror of $K_{\bP^3}$ is given by
	\[
	 uv = (1+\delta(\bT^\tau)) + z_1 + z_2 + z_3 + \bT^\tau z_1^{-1}z_2^{-1}z_3^{-1}
	\]
	where $1+\delta(\bT^\tau)$ is the generating function counting stable discs bounded by a Lagrangian toric fiber.  By \cite{CLT11,CCLT13}, it can be computed from the mirror map as follows.  The mirror map for $K_{\bP^3}$ is given by $q = Q e^{f(Q)}$, where $q=\bT^\tau$ is the K\"ahler parameter for the primitive curve class in $K_{\bP^3}$, $Q$ is the mirror complex parameter and $f(Q)$ is the hypergeometric series
	\[
	f(Q)=\sum_{k=1}^\infty \frac{(4k)!}{k(k!)^4} Q^k.
	\]
	Taking the inverse, we get the inverse mirror map
	\[
	 Q(q) = q - 24q^2 - 396 q^3 - 39104q^4 - 4356750 q^5 - O(q^6).
	\]
	Then the generating function of open Gromov-Witten invariants of a Lagrangian toric fiber is given by 
	\[
	 1 + \delta(q) = \exp (f(Q(q))/4) = 1+6q+189q^2+14366q^3+1518750q^4+O(q^5).
	\]
	(see also \cite[Theorem 1.1]{CLLT17}).
	
	Now we use the above method to deduce from this the $\bS^1$-equivariant disc potential for the immersed $L_0$.   Namely, the gluing formula between a smooth torus and $L_0$ is given by substituting $z_1 = -e^{x_1}$ (and replacing $z_2,z_3$ by $\bT^{A_2}z_2,\bT^{A_3}z_3$ respectively) in the above equation for the SYZ mirror.  Then we solve $x_1$ in terms of $z_2,z_3,u,v$, and substitute into $x_1 \lambda_1$ (which is the $\bS^1$-equivariant disc potential for $L_1$).  The following table shows the leading coefficients $a_{jk\ell \mu}$ of the generating function $\sum -a_{jk\ell \mu} z_1^j z_2^k q^\ell (uv)^\mu$ for $\mu=0$ (corresponding to stable discs with no corners).
	 \begin{center}
	 	{\small
	 	\noindent \begin{tabular}{|c|c|c|c|c|} 
	 		\hline
	 		\multicolumn{5}{|c|}{$\mathrm{ord}(q)=0$} \\
	 		\hline
	 		& \multicolumn{4}{|c|}{$\mathrm{ord}(z_2)$} \\
	 		\hline
	 		$\mathrm{ord}(z_1)$ & $0$ & $1$ & $2$ & $3$\\ 
	 		\hline
	 		$0$ & $0$ & $1$  & $1/2$ & $1/3$ \\ 
	 		\hline
	 		$1$ & $1$ & $1$ & $1$ & $1$ \\ 
	 		\hline
	 		$2$ & $1/2$ & $1$ & $3/2$ & $2$ \\ 
	 		\hline	
	 		$3$ & $1/3$ & $1$ & $2$ & $10/3$  \\ 
	 		\hline								
	 	\end{tabular}
		\begin{tabular}{|c|c|c|c|c|} 
		\hline
		\multicolumn{5}{|c|}{$\mathrm{ord}(q)=1$} \\
		\hline
		& \multicolumn{4}{|c|}{$\mathrm{ord}(z_2)$} \\
		\hline
		$\mathrm{ord}(z_1)$ & $-1$ & $0$ & $1$ & $2$ \\ 
		\hline
		$-1$ & $1$ & $2$  & $3$ & $4$ \\ 
		\hline
		$0$ & $2$ & $6$ & $12$ & $20$ \\ 
		\hline
		$1$ & $3$ & $12$ & $30$ & $60$ \\ 
		\hline	
		$2$ & $4$ & $20$ & $60$ & $140$  \\ 
		\hline								
		\end{tabular} \\
	 	\begin{tabular}{|c|c|c|c|c|c|} 
	 		\hline
	 		\multicolumn{6}{|c|}{$\mathrm{ord}(q)=2$} \\
	 		\hline
	 		& \multicolumn{5}{|c|}{$\mathrm{ord}(z_2)$} \\
	 		\hline
	 		$\mathrm{ord}(z_1)$ & $-2$ & $-1$ & $0$ & $1$ & $2$ \\ 
	 		\hline
	 		$-2$ & $3/2$ & $6$ & $15$ & $30$ & $105/2$  \\ 
	 		\hline
	 		$-1$ & $6$ & $36$ & $108$  & $246$ & $480$ \\ 
	 		\hline
	 		$0$ & $15$ & $108$ & $387$ & $1020$ & $2250$ \\ 
	 		\hline
	 		$1$ & $30$ & $246$ & $1020$ & $3060$ & $7560$ \\ 
	 		\hline	
	 		$2$ & $105/2$ & $480$ & $2250$ & $7560$ & $20685$  \\ 
	 		\hline								
	 	\end{tabular}
	 }	
	 \end{center}

\subsubsection{Local Calabi-Yau surfaces}	
We next consider the local Calabi-Yau surface $X_{(d)}$ of type $\tilde{A}_{d-1}$.
	The surface can be realized as the $\Z$-quotient of the following infinite-type toric Calabi-Yau surface.  Such a toric construction was found by Mumford \cite{Mumford,AMRT} and its mirror symmetry was studied by Gross-Siebert \cite{GS-Ab,ABC}.
	
	Let $\textbf{N}= \Z^2$.  For $i \in \Z$, define the cone
	\[
	\sigma_i=\R_{\ge 0}\cdot \langle (i,1), (i+1,1) \rangle \subset \textbf{N}_\R. 
	\]
	The fan $\Sigma_0=\bigcup_{i \in \Z} \sigma_i \subset \R^2$ is defined as the infinite collection of these cones (and their boundary cones). 
	The corresponding toric surface $X = X_{\Sigma_0}$ is Calabi--Yau since all the primitive generators $(i,1) \in \textbf{N}$ have second coordinates being $1$.

	The fan $\Sigma_0$ has an obvious symmetry of the infinite cyclic group $\Z$ given by $k \cdot (a,b) = (a+k,b)$ for $k\in \Z$ and $(a,b) \in N$.  We can take an open neighborhood $X^o$ of the toric divisors which is invariant under the $\Z$-action, and take the quotient $X_{(d)} =X^o / (d\cdot\Z)$.  
	
	The SYZ mirror of $X_{(d)} =X^o / (d\cdot\Z)$ was constructed in \cite{KL16}.  For simplicity, we examine the case $d=1$, only.  The SYZ mirror is given by
	\[
	 uv = \prod_{i=1}^\infty(1+q^iz^{-1})\prod_{j=0}^\infty(1+q^jz).
	\]
	The right hand side can be rewritten as
	\[
      \prod_{k=1}^\infty\frac{1}{1-q^k}\cdot\sum_{\ell=-\infty}^\infty q^{\frac{\ell(\ell-1)}{2}}z^\ell = \frac{e^{\frac{\pi \mathbf{i} \tau}{12}}}{\eta(\tau)}\cdot \vartheta \left( \zeta - \frac{\tau}{2}; \tau \right)
    \]
	where $q:=e^{2\pi \mathbf{i} \tau}$, $z:=e^{2\pi \mathbf{i} \zeta}$, $\eta$ is the Dedekind eta function, and $\vartheta$ is the Jacobi theta function. 
	
	Now, we consider  the equivariant disc potential of the immersed $2$-sphere $\cS^2$.  We have fixed the chamber dual to $(0,1)$ and the framing $v_1' =(1,0)$ for the equivariant Morse model of $\cS^2$.  The above mirror equation can be understood as the relation between the immersed variables $u,v$ of the immersed sphere $\cS^2$ and the formal deformations $x$ of the torus $T^2$, where $z = - e^x$.  The $\bS^1$-equivariant disc potential of $T^2$ is given by $x \lambda$, and that of $\cS^2$ can be found by solving $x$ in the above equation.
	
	We solve the equation order-by-order in $q = \bT^t$ where $t$ is the area of the primitive holomorphic sphere class.  Namely, the equation is rewritten as
	\[
	 (1-uv-e^{x}) +(uv - e^{-x} + e^{2x})q + O(q^2) = 0.
	\]
	Then we put $x = \sum_{k\geq 0} h_k(uv) q^k$ and solve $h_k$ order by order.  For instance, the leading term is $h_0 = \log (1-uv) = -\sum_{i>0} \cfrac{(uv)^i}{i}$.  Write the generating function as $-\sum a_{k\ell} (uv)^\ell q^k$ and the coefficients are shown in the following table.
	
	\begin{center}	
	\noindent \begin{tabular}{|c|c|c|c|c|c|c|} 
	\hline
	& \multicolumn{6}{|c|}{$\mathrm{ord}(q)$} \\
	\hline
	$\mathrm{ord}(uv)$ & $0$ & $1$ & $2$ & $3$ & $4$ & $5$ \\ 
	\hline
	$0$ & $0$ & $0$ & $0$ & $0$ &$0$ & $0$ \\ 
	\hline	
	$1$ & $1$ & $2$ & $5$ & $10$ & $20$ & $36$ \\ 
	\hline					
	$2$ & $1/2$ & $2$ & $7$ & $20$ & $105/2$ & $126$ \\ 
	\hline	
	$3$ & $1/3$ & $3$ & $18$ & $245/3$ & $315$ & $1071$ \\ 
	\hline	
	$4$ & $1/4$ & $4$ & $33$ & $192$ & $1815/2$ & $3696$ \\ 
	\hline	
	$5$ & $1/5$ & $5$ & $55$ & $410$ & $2415$ & $60252/5$ \\
	\hline			
	\end{tabular}	
	\end{center}
	
	The first column records the counts of constant polygons and it is the same as the result for $X=\C^2$ in \cite[Theorem 4.6]{KLZ19} (namely they are coefficients of the series $-\log (1-uv)$).  The constant polygons are merely local and are not affected by the presence of the holomorphic $A_1$-fiber.  We also note that the coefficients multiplied with the corresponding orders of $(uv)$ are integers.
	

\subsubsection{The total space of a family of Abelian surfaces}
	This is the three-dimensional version of the last example.  Consider the fan $\Sigma$ consisting of the maximal cones
	\[
	\langle (i,j,1), (i+1,j,1),(i,j+1,1),(i+1,j+1,1)\rangle, \ \ \ i,j\in\Z.
	\]
	It admits an action by $\Z^2$: $(k,\ell) \cdot (a,b,c) = (a+k,b+\ell,c)$ on $N$. 
	A crepant resolution is obtained by refining each maximal cone (which is a cone over a square) into $2$ simplices, and the refinement is taken to be $\Z^2$-invariant.  Then we take a toric neighborhood $X^o_\Sigma$ of the toric divisors where $\Z^2$ acts freely.   This is a toric Calabi-Yau manifold, which is also the total space of a family of Abelian surfaces.
	
	The SYZ mirror was constructed in \cite[Theorem 5.6]{KL16}.  It is given by 
	\[
	uv = \Delta(\Omega) \cdot \sum_{(j,k)\in\Z^2}  q_\sigma^{\frac{(j+k)(j+k-1)}{2}}q_2^{\frac{j(j-1)}{2}}q_1^{\frac{k(k-1)}{2}} z_1^j z_2^k = \Delta(\Omega) \cdot 
	\Theta_2
	\begin{bmatrix}
	0 \\
	(- \frac{\tau}{2},- \frac{\rho}{2})
	\end{bmatrix}
	\left(\zeta_1, \zeta_2; \Omega\right). 
	\]
	Here $\Omega:=\begin{bmatrix}
	\tau   &  \sigma  \\
	\sigma& \rho  \\
	\end{bmatrix}$.  There are three K\"ahler parameters $q_\tau = e^{2\pi \mathbf{i} \tau}=q_1q_\sigma$, $q_\rho= e^{2\pi \mathbf{i} \rho}=q_2q_\sigma$, $q_\sigma= e^{2\pi \mathbf{i} \sigma}$.
	$z_i = e^{2\pi \mathbf{i} \zeta_i}$, $\Theta_2$ is the genus $2$ Riemann theta function, and
	\begin{equation} \label{eq:Delta}
	\Delta(\Omega)
	= \exp \left(\sum_{j \geq 2} \frac{(-1)^j }{j } 
	\sum_{\substack{({\bf l}_i=(\ell_i^1, \ell_i^2)\in \Z^{2} \setminus 0)_{i=1}^j  \\ s.t. \sum_{i=1}^j  {\bf l}_i = 0}}
	\exp\left(\sum_{k=1}^j \pi \mathbf{i} {\bf l}_k\cdot \Omega \cdot {\bf l}^{\tr}_{k}\right) \right).
	\end{equation}

	As before, we choose a Lagrangian immersion $L_0\cong \cS^2 \times \bS^1$ which bounds a primitive holomorphic disc with area $A$ whose boundary has holonomy parametrized by $z_2$.  We have fixed the chamber dual to $(0,0,1)$ and the framing $v_1'=(1,0,0), v_2'=(0,1,0)$ for the $\bS^1$-equivariant Morse model of $L_0$. Then the above mirror equation can be understood as the gluing formula between $L_0$ and a embedded Lagrangian torus, with $z_2$ replaced by $\bT^A  z_2$, and $z_1 = -e^{x_1}$.  The equivariant disc potential is given by $x_1 \lambda_1$, where $x_1$ is solved from the above equation.  
	
	Setting $u=v=0$, we obtain the generating function of stable discs with no corners.  This gives an enumerative meaning of the zero locus of the Riemann theta function (on $(z_1,z_2) \in (\C^\times)^2$) as a section over the $z_2$-axis.  The leading-order terms are given as follows (where $w=\bT^A z_2$):
	{\tiny
	\begin{align*}
	&\left(\left(\frac{3 q_1^2 q_2^2}{2}-2 q_1^2 q_2+\frac{q_1^2}{2}\right)+4 q_1^2 q_2 q_\sigma +  \left(\frac{39 q_1^2 q_2^2}{2}-\frac{q_1^2}{2}\right)q_\sigma^2\right) w^{-2}\\
	&+\left((-q1 + q1 q2 - 2 q1^2 q2 + 
	3 q1^2 q2^2) + (q1 - q1^2 - 3 q1 q2 + 11 q1^2 q2 + 
	3 q1 q2^2) qs + (q1^2 + 3 q1 q2 - 23 q1^2 q2 - 9 q1 q2^2 + 
	126 q1^2 q2^2) qs^2\right)w^{-1}\\
	&+\left(\left(6 q_1^2 q_2^2-q_1^2 q_2-3 q_1 q_2^2+2 q_1 q_2-q_2+1\right)+\left(15 q_1^2 q_2+33 q_1 q_2^2-11 q_1 q_2+q_1-3 q_2^2+3 q_2-1\right)q_\sigma \right.\\
	&\left.+ \left(600 q_1^2 q_2^2-62 q_1^2 q_2+q_1^2-126 q_1 q_2^2+23 q_1 q_2-q_1+9 q_2^2-3 q_2\right) q_\sigma^2\right)w\\
	&+\left(\left(-21 q_1^2 q_2^2+2 q_1^2 q_2+12 q_1 q_2^2-4 q_1 q_2-\frac{3 q_2^2}{2}+2 q_2-\frac{1}{2}\right)+ \left(-24 q_1^2 q_2-96 q_1 q_2^2+16 q_1 q_2+12 q_2^2-4 q_2\right)q_\sigma\right.\\
	&\left.+ \left(-\frac{2577 q_1^2 q_2^2}{2}+78 q_1^2 q_2-\frac{q_1^2}{2}+264 q_1 q_2^2-20 q_1 q_2-\frac{39 q_2^2}{2}+\frac{1}{2}\right)q_\sigma^2\right)w^2.
	\end{align*}
	}

\section{Equivariant disc potentials of immersed Lagrangian tori}\label{sec:immtoriequiv}
So far, we have mainly focused on the immersed Lagrangian $\cS^2\times T^{n-2}$, whose Maurer-Cartan deformation space fills the codimension-two toric strata of a toric Calabi-Yau manifold.  In this section, we consider an immersed Lagrangian torus that copes with lower dimensional toric strata.  The construction of a Landau-Ginzburg mirror associated to such an immersed torus, using the method in \cite{CHL}, has been introduced in the survey article \cite{Lau18}. The immersed torus and its application to HMS have also been addressed in recent talks by Abouzaid.

The immersed torus $\cL$ is constructed by symplectic reduction.  Recall from Section \ref{sec:review} that we have the reduction $X \sslash_{a_1} T^{n-1} := \rho^{-1}\{a_1\} / T^{n-1} \stackrel{w}{\cong} \C$ where $\rho$ denotes the $T^{n-1}$-moment map.  Previously, we chose $a_1$ to be in the image of a strictly codimension-two toric stratum so that a path passing through $0 \in \C$ gives rise to a Lagrangian immersion in $X$.  In this section, we remove this restriction and allow $a_1$ to lie in the image of any toric stratum.  

Let us take the immersed curve $C$ in the $w$-plane shown in Figure \ref{fig:C}.  This immersed circle was first introduced by Seidel \cite{Seidel-g2} for proving homological mirror symmetry of genus-two curves.  Here the punctured plane $\C - \{0,1\}$ is identified with $\bP^1 - \{0,1,\infty\}$.

\begin{figure}[htb!]
	\centering
	\begin{subfigure}[b]{0.3\textwidth}
		\centering
		\includegraphics[width=\textwidth]{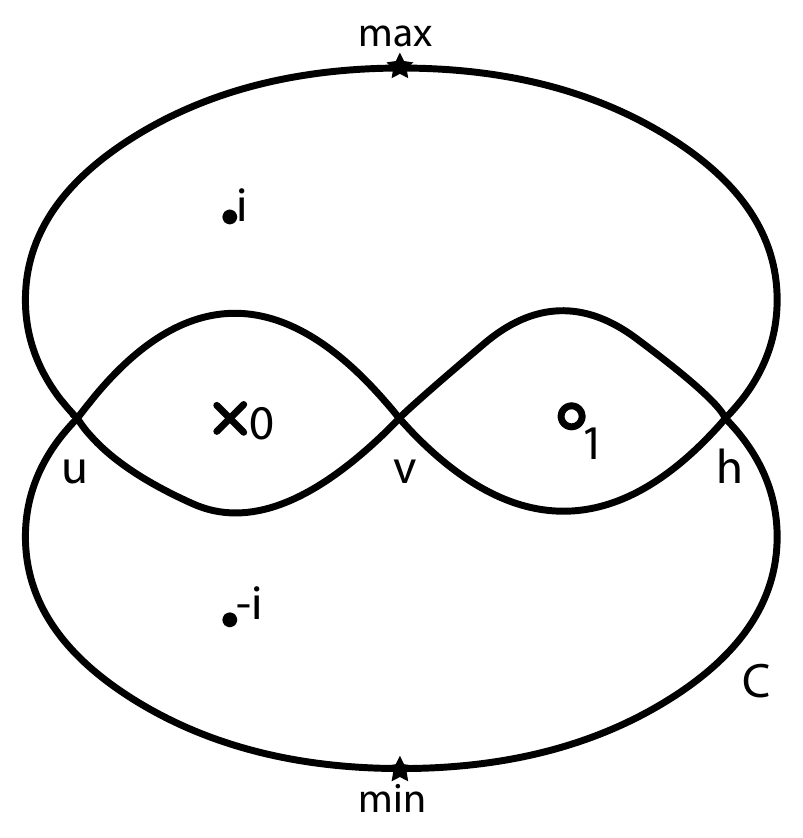}
		\caption{The curve $C$.}
		\label{fig:C}
	\end{subfigure}
	\hspace{10pt}
	\begin{subfigure}[b]{0.3\textwidth}
		\centering
		\includegraphics[width=\textwidth]{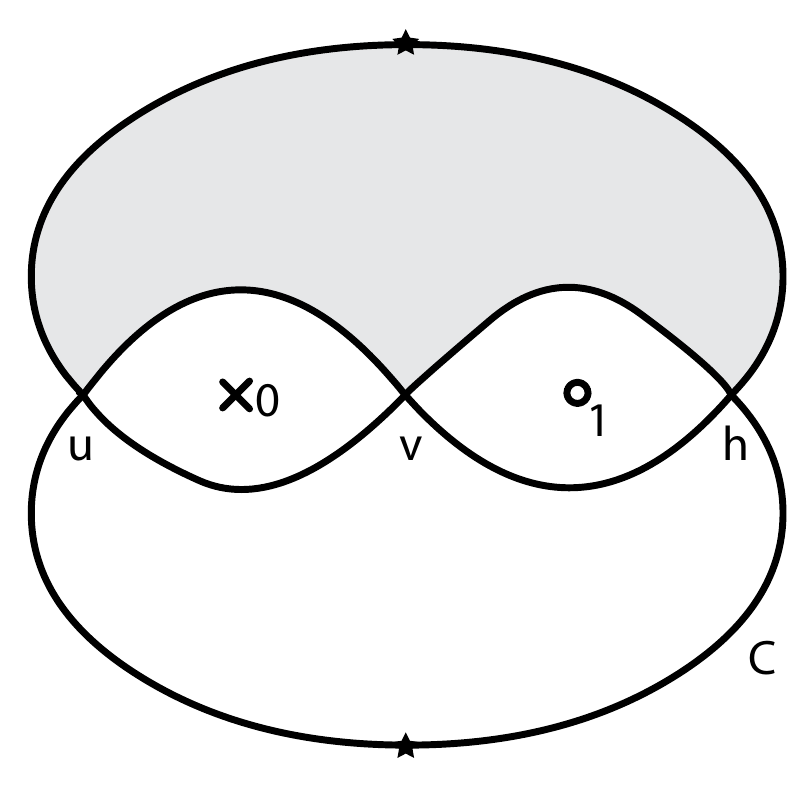}
		\caption{The $uvh$-polygon.}
		\label{fig:disc-uvh}
	\end{subfigure}
	\hspace{10pt}
	\begin{subfigure}[b]{0.3\textwidth}
		\centering
		\includegraphics[width=\textwidth]{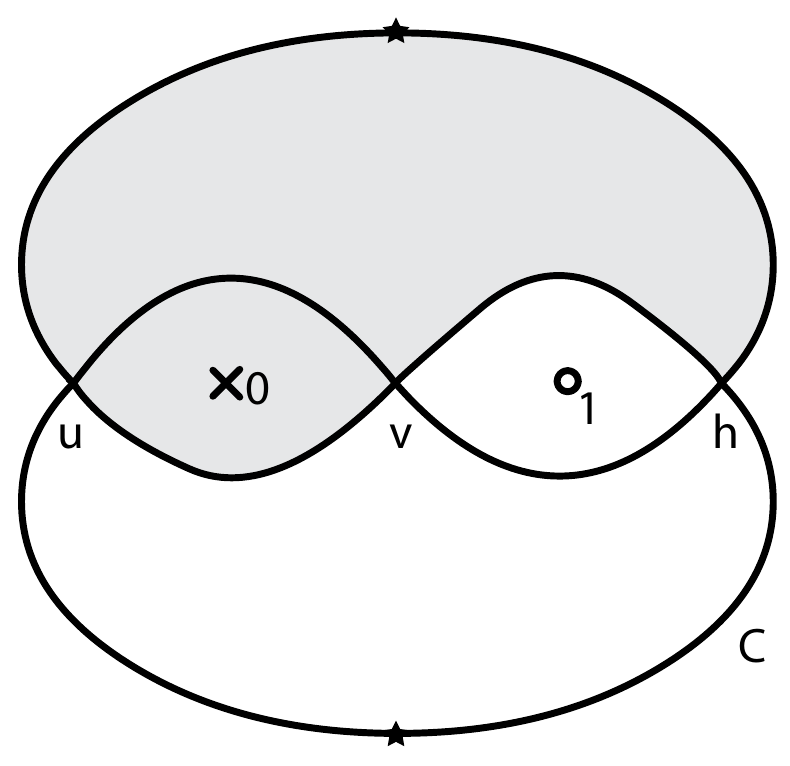}
		\caption{The polygons in $h f(z)$.}
		\label{fig:disc-h}
	\end{subfigure}
	\caption{The curve $C$ and polygons bounded by $C$.}
	\label{fig:Cdisc}
\end{figure}

The curve $C$ has three self-intersection points $U,V,H$, each of which corresponds to two immersed generators (in the Floer complex of $C$).  By abuse of notations, they are denoted by $U,V,H$ (with odd degree) and $\bar{U},\bar{V},\bar{H}$ (with even degree) respectively.

Since $C$ does not pass through $0 \in \C$ (that would create additional  singularity), its preimage $\cL$ in $\rho^{-1}\{a_1\} \subset X$ is a Lagrangian immersion in $X$ from an $n$-dimensional torus, where the fiber over each point of $C$ is a torus $T^{n-1}$.  Thus $\cL$ is an immersed torus which has three clean self-intersections isomorphic to $T^{n-1}$ that lie over the three self-intersection points $U,V,H$ of $C$.  Let us denote the immersion by $\iota: \widetilde{\cL} \to \cL$ where $\widetilde{\cL}=T^n$ is an $n$-dimensional torus.

In order to express $\cL$ as the product of the immersed curve $C$ with a fiber $T^{n-1}$ precisely (this splitting is not canonical), we take a trivialization of the holomorphic fibration $w: X \to \C$ as follows.  All the fibers of $w$ are $(\C^\times)^{n-1}$ except $w^{-1}(0)$, which is the union of toric prime divisors.  By relabeling $v_i$ if necessary, we may assume that the toric stratum that $a_1$ is located is adjacent to the $v_1$-facet.  We fix a basis $v_i'\in \unu^\perp \subset \textbf{N}$ for $i=1,\ldots,n-1$.  Then $\{v_1,v_1',\ldots,v_{n-1}'\}$ forms a basis of $\textbf{N}$, whose dual basis equals to $\{\unu,\nu_1,\ldots,\nu_{n-1}\}$ for some $\nu_1,\ldots,\nu_{n-1} \in \textbf{M}$.  This gives a set of coordinate functions (where $\unu$ corresponds to $w$), which gives a biholomorphism between the complement of all the toric divisors corresponding to $v_i$ for $i\not=1$ and $\C \times (\C^\times)^{n-1}$.  The projections to the factors $\C$ and $(\C^\times)^{n-1}$ give the splitting $\cL \cong C \times T^{n-1}$.  This also fixes a basis $\{\gamma_1,\ldots,\gamma_n\}$ of $\pi_1(\widetilde{\cL})$.

The Floer complex of $\cL$ admits the following description.
Observe that $\widetilde{\cL} \times_{\cL} \widetilde{\cL}$ consists of seven connected components, as the self intersection loci of $\cL$ has three components.  One of them is simply $\widetilde{\cL}$, which is responsible for non-immersed generators, which are represented by critical points of Morse functions on $\widetilde{\cL}$.  The remaining six components are copies of $T^{n-1}$, and they give rise to the generators of the form $G \otimes X_I$ for $G = U,V,H,\bar{U},\bar{V},\bar{H}$ and $I \subset \{1,\ldots,n-1\}$, which are represented by critical points in the corresponding $T^{n-1}$-components of $\widetilde{\cL} \times_{\cL} \widetilde{\cL}$. We will specify the choice of relevant Morse functions shortly. 
For simplicity, we will often write $G$ to denote $G \otimes \one_{T^{n-1}}$ when there is no danger of confusion.


We have the holomorphic volume form 
\[
\Omega = (dw / (w-1)) \wedge d \log z_1 \wedge \ldots d \log z_{n-1}
\] 
on the divisor complement $X^\circ = X-\{w=1\}$, where $z_1, \ldots, z_{n-1}$ are the local toric coordinates corresponding to $\{nu_1,\ldots,\nu_{n-1}\}$.   It induces the one-form $d \log (w-1)$ on the reduced space $X \sslash_{a_1} T^{n-1} \cong \C$.

\begin{lemma}
	$\cL$ is graded with respect to $\Omega$.  The generators $U,V,H $ have degrees $1,1,-1$ respectively, and the generators $\bar{U} ,\bar{V} ,\bar{H}$ have degrees $0,0,2$. Here $\one_{T^{n-1}}$ denotes the maximum point of the Morse function on $T^{n-1}$-component.
\end{lemma}

To be more precise, the degree of $G \otimes X_I$ for $G = U,V,H,\bar{U},\bar{V},\bar{H}$ and $I \subset \{1,\ldots,n-1\}$ is obtained by adding the degree of $X_I$ to that of $G$ given in this lemma.

\begin{proof}
	By the reduction, it suffices to check that the curve $C$ is graded with respect to $-\mathbf{i} d \log (w-1)$.  The preimage of $C$ under $w = e^{\mathbf{i}y} + 1$ is a union of `figure-8' as shown in Figure \ref{fig:Seidel-11-1}.  There is a well-defined phase function for the one-form $dy$ on (the normalization of) all the figure-8 components, which is invariant under translation by $2\pi$.  Thus $C$ is graded.  The degrees of the immersed generators can be directly computed from the phase function.
\end{proof}

\begin{figure}[htb!]
	\centering
	\begin{subfigure}[b]{0.45\textwidth}
		\centering
		\includegraphics[width=\textwidth]{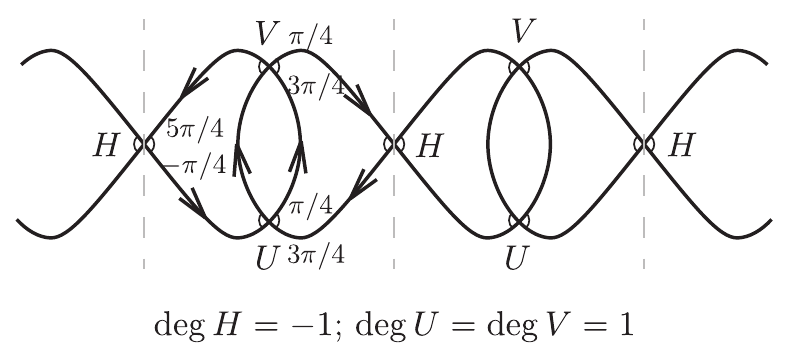}
		\caption{The lifting under $y = -i \log (w-1)$.}
		\label{fig:Seidel-11-1}
	\end{subfigure}
	\hspace{10pt}
	\begin{subfigure}[b]{0.51\textwidth}
		\centering
		\includegraphics[width=\textwidth]{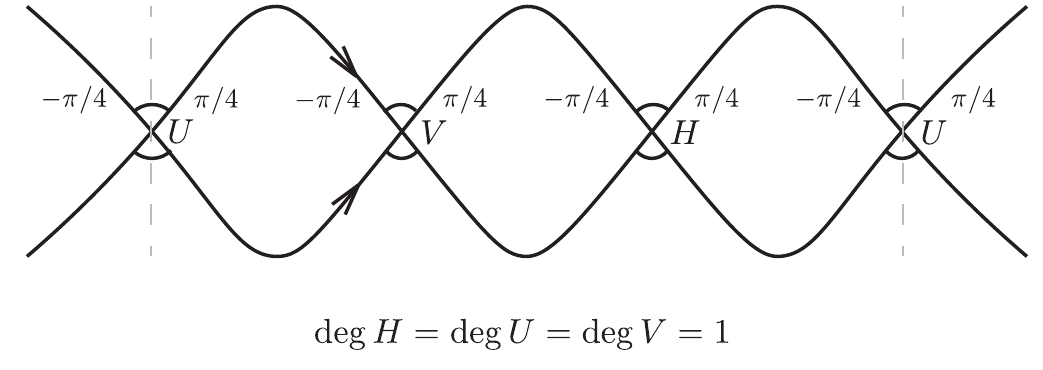}
		\caption{The lifting under $y' = \log \frac{w+i}{w-i}$.}
		\label{fig:Seidel-111}
	\end{subfigure}
	\caption{The gradings on $C$ by two different holomorphic volume forms.}
	\label{fig:Seidel-grading}
\end{figure}

For the purpose of computing Maslov indices, we also consider another grading for $\cL$ given as follows.  We take the meromorphic volume form 
\[
\tilde{\Omega} = \frac{-2\mathbf{i} \, dw \wedge d\log z_1 \wedge \ldots \wedge d \log z_{n-1}}{(w+\mathbf{i})(w-\mathbf{i})}
\]
on restricted on $X^\circ$, which corresponds to $d \log \frac{w+\mathbf{i}}{w-\mathbf{i}}$ in the reduced space.  It has the pole divisors $w=\pm \mathbf{i}$.  Similar to the above lemma, we can directly check the following.  See Figure \ref{fig:Seidel-111}.

\begin{lemma}
	$\cL$ is graded with respect to $\tilde{\Omega}$.  The generators $U ,V,H$ all have degree one, and the generators $\bar{U},\bar{V},\bar{H}$ all have degree zero.
\end{lemma}

We extend the Maslov index formula \cite{CO, Auroux07} to the immersed setting.

\begin{lemma}
	Let $L$ be an immersed Lagrangian graded by a meromorphic nowhere zero $n$-form $\Omega$.  For a holomorphic polygon in class $\beta$ bounded by $L$ with corners at degree one immersed generators of $L$, its Maslov index equals to $2 \beta \cdot D \geq 0$ where $D$ is the pole divisor of $\Omega$.
\end{lemma}
\begin{proof}
	We trivialize $TX$ pulled back over the domain polygon $\Delta$, and we have a Lagrangian sub-bundle $TL$ over the boundary edges of the polygon (which can be understood as a disc with boundary punctures).  It is extended to be a Lagrangian sub-bundle over $\partial \Delta$ by taking positive-definite paths at the corners.  Since the corners have degree one with respect to (the pull-back of) $\Omega$, we have a well-defined real-valued phase function for the Lagrangian sub-bundle.  Thus the phase change with respect to $\Omega$ equals to zero with respect to $\Omega$.
	
	The pull-back of $\Omega$ has poles at $a_i \in \Delta$.  Then $\left(\prod_i \frac{z-a_i}{1-\bar{a}_i z}\right) \Omega$ becomes a nowhere zero holomorphic section of $\bigwedge^n T^*X|_\Delta$.  Each factor $\frac{z-a_i}{1-\bar{a}_i z}$ results in adding $2 \pi$ for the phase change.  Thus the total phase change equals to $2 \pi k$ where $k$ is the number of poles.  Thus the Maslov index equals to  Maslov index equals to $2 \beta \cdot D$.
\end{proof}

For the Fukaya category of $X^\circ$, the objects are Lagrangians graded by $\Omega$.  The grading of $\cL$ by $\tilde{\Omega}$ is auxiliary.  We apply the above lemma to $\cL$ using the grading by $\tilde{\Omega}$.  Thus the Maslov index of a holomorphic polygon with corners at $U,V,H$ equals to two times the intersection number with the pole divisor $\{w=i\}\cup \{w=-i\}$.  In particular $L_T$ does not bound any holomorphic polygon with corners at $U,V,H$ which has negative Maslov index.

We take a perfect Morse function on each component of $\widetilde{\cL} \times_{\cL} \widetilde{\cL}$.
First we take a perfect Morse function on $\widetilde{\cL} \cong \bS^1 \times T^{n-1}$.  We take a perfect Morse function on the normalization $\bS^1$ of $C$, whose maximum and minimum points lie in the upper and lower parts of $C$ respectively, see Figure \ref{fig:Cdisc}.  Let us denote the maximum and minimal points by $p_{\mathrm{max}}$ and $p_{\mathrm{min}}$ respectively.
We take a perfect Morse function on the $T^{n-1}$-factor, whose unstable hypertori of the degree-one critical points are dual to the $\bS^1$-orbits of $v_i'$.  The sum of these two functions gives the desired perfect Morse function.  

We also take such a perfect Morse function on the $T^{n-1}$-components of $\widetilde{\cL} \times_{\cL} \widetilde{\cL}$.  They are identified with the clean intersections of $\cL$.  The perfect Morse function is taken such that the unstable hypertori of the degree-one critical points are dual to the $\bS^1$-orbits of $v_i'$ for $i=1,\ldots,n-1$.

As a result, the Morse complex of $\widetilde{\cL} \times_{\cL} \widetilde{\cL}$ has the following generators.  The component $\widetilde{\cL} \cong T^n$ has the generators $X_I$ for $I \subset \{0,\ldots,n-1\}$ (where $X_\emptyset = \one_{\widetilde{\cL}}$).  We also have the immersed generators $G \otimes X_I$ for $G = U,V,H,\bar{U},\bar{V},\bar{H}$ and $I \subset \{1,\ldots,n-1\}$.  They are critical points in the corresponding $T^{n-1}$-components of $\widetilde{\cL} \times_{\cL} \widetilde{\cL}$. (Recall that we have written $G$ for  $G \otimes \one_{T^{n-1}}$ by an abuse of notation so far.)


Let us equip $\cL$ with flat $\Lambda_{\mathrm{U}}$-connections.  We do it in a `minimal way'.  Namely, we only consider connections which are trivial along $\gamma_1$, because such a deformation direction is already covered by the formal deformations of the immersed generators $U,V,H$.  We denote by $z_i$ the holonomy variables associated to the loops $\gamma_i$.  
We consider the formally deformed immersed Lagrangian $(\cL,uU + vV + hH, \nabla^{(z_1,\ldots,z_{n-1})})$.  We will prove that these are weak bounding cochains. Notice that $\deg u = \deg v = 0$ and $\deg h = 2$ with respect to $\Omega$, whereas, with respect to $\tilde{\Omega}$, $\deg u = \deg v = \deg h = 0$.  (This ensures $b = uU + vV + hH$ always has degree one.)

The only non-constant holomorphic polygons bounded by $C$ of Maslov index two are the two triangles with corners $U,V,H$ (shown in Figure \ref{fig:disc-uvh}), or the two one-gons with the corner $H$ (shown in Figure \ref{fig:disc-h}).  Using this, we classify holomorphic polygons bounded by $\cL$ in what follows.

\begin{lemma}
	Any non-constant holomorphic polygon bounded by $\cL$ must project to a non-constant holomorphic polygon bounded by $C$ under $w$.
\end{lemma}

\begin{proof}
	If $w$ is constant, the holomorphic polygon is contained in the corresponding fiber of $w$, which is $(\C^\times)^{n-1}$.  $\cL$ intersects this fiber at $T^{n-1}$ which is isotopic to the standard torus (the product of unit circles) in $(\C^\times)^{n-1}$.  By the maximal principle, such a torus does not bound any non-constant holomorphic disc.
\end{proof}

Since $C$ does not bound any non-constant Maslov-zero holomorphic disc, by the above lemma, there is no non-constant Maslov-zero holomorphic disc bounded by $\cL$.  

The classification of holomorphic disc classes bounded by $\cL$ is similar to Lemma \ref{lem:hol-strip-X} and Corollary \ref{cor:hol-strip-class}.  Moreover Lemma \ref{stable-strip} and Proposition \ref{prop:strip=disc} also have their counterparts for $\cL$. The proofs are parallel to the corresponding ones for $L_0 \cong \cS^2\times T^{n-2}$, and hence omitted.

\begin{prop} 
\begin{enumerate}
\item[(i)]
	The only holomorphic polygon classes of Maslov index two bounded by $\cL$ are $\beta^\pm_i$ for $i=0,\ldots,m$, where $\beta^+_0$ (or $\beta^-_0$) is the class which never intersects $w=0$ and projects to the triangle passing through $p_{\mathrm{max}}$ in $C$ (or passing through $p_{\mathrm{min}}$ respectively); $\beta^+_i$ (or $\beta^-_0$) is the class which intersects the toric divisor $D_i$ once but not the other $D_j$ for $j\not=i$, and projects to the one-gon passing through $p_{\mathrm{max}}$ in $C$ (or passing through $p_{\mathrm{min}}$ respectively).
\item[(ii)]
	The stable polygon classes of Maslov index two bounded by $\cL$ are $\beta^\pm_i + \alpha$ for $i=0,\ldots,m$ and $\alpha$ is an effective curve class.  For $i=0$, $\alpha=0$.
\item[(iii)]
	The moduli space 
	\[
	\cM^{\cL}_2(H; \beta_i+\alpha) \times_{\ev_0} \{p_{\mathrm{max}}\}
	\] 
	for $i=1,\ldots,m$ is isomorphic to $\cM_1(\beta_i^L+\alpha) \times_{\ev} \{p\}$ for a certain Lagrangian toric fiber $L$ and a certain point $p\in L$.  
\item[(iv)]
	The moduli space
	\[
	\cM^{\cL}_4(H, V, U; \beta_0) \times_{\ev_0} \{p_{\mathrm{max}}\}
	\]
	is simply a point.
\end{enumerate}
\end{prop}


We equip $\cL$ with a non-trivial spin structure, which is represented by the $T^{n-1}$-fiber of the point $p_{\mathrm{max}} \in C$. By abuse of notation, we denote the three generators which has base degree $1$ and fiber degree $0$ by the same letters $U,V,H$ corresponding to the three immersed points of $C$.  Take the formal deformations $\bb = uU + vV + hH$ for $u,v,h \in \C$.  We also have a flat $\C^\times$ connection in the fiber $T^{n-1}$-direction over $C$; its holonomy is given by $(z_1,\ldots,z_{n-1}) \in (\C^\times)^{n-1}$.

Using cancellation in pairs of holomorphic polygons due to symmetry along the dotted line shown in Figure \ref{fig:C}, we prove the following statement.
\begin{prop}
	$(\cL,\bb,\nabla^z)$ is weakly unobstructed.
\end{prop}
\begin{proof}
This is similar to the proof of weakly unobstructedness for the Seidel Lagrangian in \cite{CHL}. The anti-symplectic involution identifies the moduli spaces $\cM_3(U,V,H;\beta_0^+ + \alpha)$ with $\cM_3(H,V,U;\beta_0^- + \alpha)$, and $\cM_1(H;\beta_i^+ + \alpha)$ with $\cM_1(H;\beta_i^- + \alpha)$ for $i=1,\ldots,m$.  Since the holomorphic polygons in $\beta_i^+$ for $i=0,\ldots,m$ pass through the spin cycle $\{p_{\mathrm{max}}\} \times T^{n-1}$, while the holomorphic polygons in $\beta_i^-$ do not, the pairs of moduli spaces have opposite signs.  Thus their contributions to $\bar{U},\bar{V},\bar{H}$ cancel and equal to zero.
\end{proof}

In particular, the superpotential associated to $(\cL,\bb,\nabla^z)$ is well-defined. Moreover, we can compute it explicitly.

\begin{theorem}
	The superpotential of $(\cL,\bb,\nabla^z)$ is
	\[
	W = -uvh + h f(z)
	\]
	defined on $((u,v,h),z) \in \C^3 \times (\C^\times)^{n-1}$, where $f$ is given in \eqref{eqn:fslabftn}.  Its critical locus is
	\[
	\check{X} = \left\{((u,v),z) \in \C^2 \times (\C^\times)^{n-1}\mid uv = f(z) \right\}.
	\]
\end{theorem}

\begin{proof}
	Since the smooth fibers are conics which topologically do not bound any non-constant discs, the image of a Maslov-two disc must be either one of the regions shown in Figure \ref{fig:Cdisc}.  For the region with corners $u,v,h$, there is no singular conic fiber and hence there is only one holomorphic polygon over it passing through a generic marked point (corresponding to the constant section).  This gives the term $-uvh$.  For the region with one corner $h$, by Riemann mapping theorem the polygons over it are one-to-one corresponding to those bounded by a toric fiber.  They contribute $h\cdot f(z)$ to $W$.
\end{proof}


We fix $a_1 \in \textbf{M}_\R/\unu_\R^\perp$, which is the level of the symplectic reduction $X \sslash_{a_1} T^{n-1} \cong \C$, by requiring that $a_1$ lies in image of strictly codimension-two toric strata of $X$ under the moment map $\rho$.  This guarantees that we have the immersed Lagrangian $\cS^2\times T^{n-2}$ on this level, and enables us to compare with $\cL$.  

To distinguish the variables between $\cL$ and $\cS^2\times T^{n-2}$, we denote the degree-one immersed generators of $\cL$ by $U_0,V_0,H$ and that of $\cS^2\times T^{n-2}$ by $U,V$.  $\cL$ and $\cS^2\times T^{n-2}$ intersect cleanly at two tori $T^{n-1}$.  We fix a basis of generators of $H^1(T^{n-2})$ (the factor contained in $\cS^2\times T^{n-2}$) and extend it to $H^1(T^{n-1})$ (the clean intersections).  Denote the corresponding holonomy variables for $\cL$ by $z^{(0)}_1,\ldots,z^{(0)}_{n-1}$, and those for $\cS^2\times T^{n-2}$ by $z^{(1)}_2,\ldots,z^{(1)}_{n-1}$.

\begin{figure}[htb!]
	\centering
	\begin{subfigure}[b]{0.45\textwidth}
		\centering
		\includegraphics[width=\textwidth]{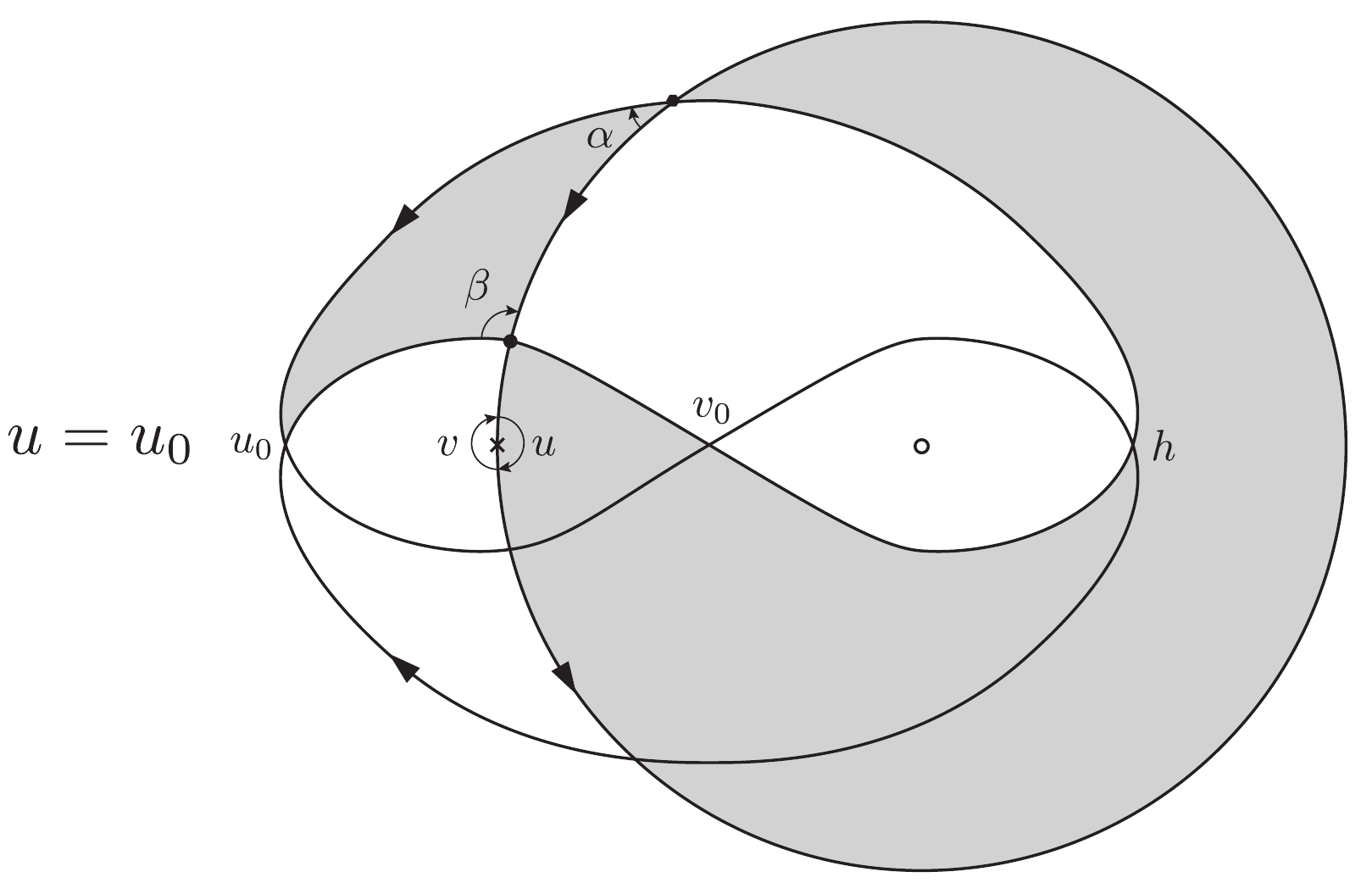}
	\end{subfigure}
	\hspace{10pt}
	\begin{subfigure}[b]{0.45\textwidth}
		\centering
		\includegraphics[width=\textwidth]{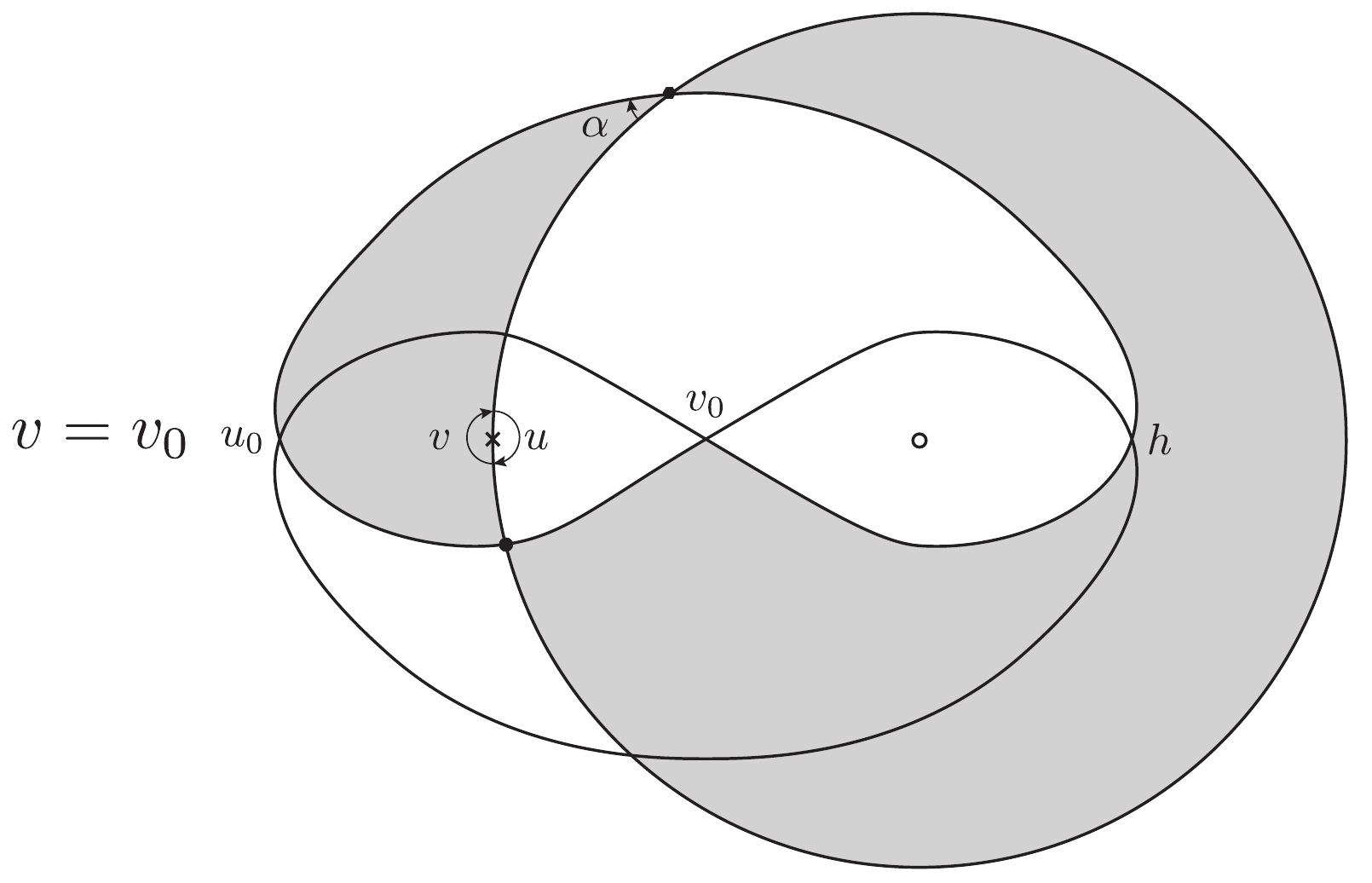}
	\end{subfigure}
	\caption{The strips in computing the embedding of $\cS^2\times T^{n-2}$ to $L_T^n$.}
	\label{fig:}
\end{figure}

\begin{theorem} \label{thm:immTS}
	There exists a non-trivial morphism $\alpha$ from $\bbL_0:=(\cS^2\times T^{n-2},uU+vV,\nabla^{\vec{z}^{(1)}})$ to $\bbL=(\cL,u_0U+v_0V+hH,\nabla^{\vec{z}^{(0)}})$ for $u_0=u, v_0=v,h=0, z^{(0)}_i=z^{(1)}_i$ for $i=2,\ldots,n-1$, and $z^{(0)}_1=g(uv,z^{(1)}_2,\ldots,z^{(1)}_{n-1})$, where $g$ is the same as that in Theorem \ref{thm:equiv}.  Moreover
	$\alpha$ has a one-sided inverse $\beta$, namely $\fm_2^{\bbL_0,\bbL,\bbL_0}(\beta,\alpha)=\one_{L_0}$.
\end{theorem}

We have different choices of the one-sided inverse $\beta$. Note that $m_2^{\bbL,\bbL_0,\bbL}(\alpha,\beta)$ has a non-zero output to $\bar{U}_0$.  Thus $\beta$ can not be a two-sided inverse. This is natural since the Maurer-Cartan deformation space of $\bbL$ is strictly bigger than that of $\bbL_0$.
On the other hand, one can check that if neither $u_0$ nor $v_0$ vanishes, $\bar{U}_0$ determines an idempotent, and we speculate that the corresponding object in the split-closed Fukaya category is actually isomorphic to $\bbL_0$ (under the coordinate change above).

\begin{remark}
	While the mirror chart of the immersed torus $\cL$ already covers the mirror chart of $\cS^2\times \bS^1$, the study of $\cS^2\times \bS^1$ is still interesting since it passes through the discriminant locus and has an analogous role of the Aganagic-Vafa brane \cite{AV} as they bound the same set of Maslov index zero holomorphic discs.
\end{remark}	

We can also consider $\bS^1$-equivariant theory of $\cL$.  It is similar to the previous sections and so we will be brief.

\begin{theorem}
	The equivariant disc potential for $\bbL=(\cL,u_0U+v_0V+hH,\nabla^{\vec{z}^{(0)}})$ equals to
	\[
	 W_{\bS^1}^{\bbL} = -uvh + h\cdot f(\exp x_1^{(0)},\ldots,\exp x_{n-1}^{(0)}) +  \sum_{i=1}^{n-1} x_i^{(0)} \lambda_i.
	\]
	It restricts to the equivariant disc potential of $\bbL_0$ via the embedding in Theorem \ref{thm:immTS}.
\end{theorem}

Thus we obtain the LG model $(\Lambda_+^{n+2} \times \mathrm{Spec}(\Lambda_+[\lambda]),W_{\bS^1}^{\bbL})$.  This LG model can be understood as an equivariant mirror for the toric Calabi-Yau manifold $X$.


\bibliographystyle{amsalpha}
\bibliography{geometry}	
\end{document}